\documentclass{article}
\usepackage{graphicx} %
\usepackage{longtable,booktabs,array}
\usepackage{xspace}
\usepackage{relsize}
\newcommand{\Jplus}{J_+}
\newcommand{\Jminus}{J_-}
\newcommand{\primaldomain}{\mathcal{D}_P}
\newcommand{\dualdomain}{\mathcal{D}_D}

\newcommand{\Qtilde}{\widetilde{Q}}

\newcommand{\wor}{\text{WoR}\xspace}
\newcommand{\Bernoulli}{\mathrm{Bernoulli}}

\newcommand{\lv}{\left \lvert}
\newcommand{\rv}{\right \rvert}
\newcommand{\Xsample}{\mathcal{X}}

\newcommand{\optdualparams}{\rho^*, \lambda^*, \boldgamma^*}

\usepackage{tikz}
\usetikzlibrary{patterns}
\usepackage{pgfplots}
\usepackage{pgfplotstable}
\pgfplotsset{compat=1.17}

\title{Tighter Confidence Intervals under  Without Replacement Sampling via Empirical Rate Functions} 
\author{Shubhanshu Shekhar \\ EECS Department \\ University of Michigan, Ann Arbor \\ \texttt{shubhan@umich.edu} \and 
Aaditya Ramdas \\ Department of Statistics and Data Science \\  Carnegie Mellon University, Pittsburgh \\ \texttt{aramdas@cmu.edu}}
\date{}

\usepackage{amsmath,amssymb,amsthm,amsfonts,latexsym,bbm,xspace,graphicx,float,mathtools,mathdots}
\usepackage{braket,caption,ellipsis,xcolor,textcomp}
\usepackage{combelow} %
\usepackage{booktabs}
\usepackage{hyperref}
\usepackage[capitalize,noabbrev]{cleveref}
\crefname{ineq}{inequality}{inequalities}
\creflabelformat{ineq}{#2{\upshape(#1)}#3}

\usepackage[letterpaper,margin=1in]{geometry}
\usepackage{enumitem}

\newtheorem{theorem}{Theorem}

\usepackage{chngcntr}
\counterwithin{theorem}{section}

\newtheorem*{namedtheorem}{\theoremname}
\newcommand{\theoremname}{testing}

\newtheorem{lemma}[theorem]{Lemma}

\newtheorem{fact}[theorem]{Fact}

\newtheorem{assumption}[theorem]{Assumption}

\theoremstyle{definition}
\newtheorem{definition}[theorem]{Definition}
\newtheorem{remark}[theorem]{Remark}

\newcommand{\calC}{\mathcal{C}}

\newcommand{\calF}{\mathcal{F}}

\newcommand{\calH}{\mathcal{H}}

\newcommand{\calK}{\mathcal{K}}

\newcommand{\calO}{\mathcal{O}}
\newcommand{\calP}{\mathcal{P}}

\newcommand{\calS}{\mathcal{S}}
\newcommand{\calT}{\mathcal{T}}

\newcommand{\calX}{\mathcal{X}}
\newcommand{\calY}{\mathcal{Y}}

\newcommand{\ignore}[1]{}

\usepackage{autonum}

\usepackage{natbib}
\setcitestyle{authoryear,open={(},close={)}} %

\DeclareMathOperator*{\argmin}{argmin}
\DeclareMathOperator*{\argmax}{argmax}

\usepackage{subcaption}

\newcommand{\iid}{\text{i.i.d.}\xspace}

\newcommand{\mc}[1]{\mathcal{#1}}

\newcommand{\lb}{\left[}
\newcommand{\rb}{\right]}
\newcommand{\lp}{\left(}
\newcommand{\rp}{\right)}
\newcommand{\lbr}{\left\{}
\newcommand{\rbr}{\right\}}

\newcommand{\defined}{\coloneqq}

\newcommand{\hatepsilon}{\widehat{\epsilon}} %

\newcommand{\dkl}{d_{\text{KL}}}

\newcommand{\muhat}{\widehat{\mu}}

\newcommand{\lambdahat}{\widehat{\lambda}}
\newcommand{\Lambdahat}{\widehat{\Lambda}}

\newcommand{\Phat}{\widehat{P}}

\newcommand{\barbeta}{\bar{\beta}}
\newcommand{\boldgamma}{\boldsymbol{\gamma}}
\newcommand{\qtext}[1]{\quad \text{#1} \quad}
\newcommand{\CSch}{C_n^{(\text{Sch})}}
\newcommand{\epsilonSch}{\epsilon_{n, \text{Sch}}}

\newcommand{\ghatplus}{\widehat{g}_+}
\newcommand{\gplus}{{g}_+}
\newcommand{\ghatminus}{\widehat{g}_-}
\newcommand{\gminus}{{g}_-}

\newcommand{\Ctilde}{\widetilde{C}}
\newcommand{\Alower}{\underline{A}}
\newcommand{\Aupper}{\bar{A}}
\newcommand{\Cemp}{C^{\mathrm{emp}}}
\newcommand{\Corc}{C^{\mathrm{orc}}}

\newcommand{\Norc}{N^{\mathrm{orc}}}
\newcommand{\Nemp}{N^{\mathrm{emp}}}

\begin{document}

\maketitle

\begin{abstract}
    We consider the problem of constructing confidence intervals~(CIs) for the population mean of $N$ values $\{x_1, \ldots, x_N\} \subset \Sigma^N$ based on a random sample of size $n$, denoted by $X^n \equiv (X_1, \ldots, X_n)$, drawn uniformly without replacement~(\wor).
    We begin by focusing on the finite alphabet~($|\Sigma| = k <\infty$) and \emph{moderate accuracy}~($\log(1/\alpha_N) \gg (k+1)\log N$) regime, and derive a fundamental lower bound on the width of any level-$(1-\alpha_N)$ CI in terms of the inverse of the \wor rate functions from the theory of large deviations. Guided by this lower bound, we propose a new level-$(1-\alpha_N)$ CI using an empirical inverse rate function, and show that in certain asymptotic regimes the width of this CI matches the lower bound up to constants. We also derive a dual formulation of the inverse rate function that enables efficient computation of our proposed CI. 
    We then move beyond the finite alphabet case and use a Bernoulli coupling idea to construct an \emph{almost sure} CI for $\Sigma = [0,1]$, and a conceptually simple nonasymptotic CI for the case of $\Sigma$ being a $(2,D)$ smooth Banach space. 
    For both finite and general alphabets, our results employ classical large deviation techniques in novel ways, thus establishing new connections between estimation under \wor sampling and the theory of large deviations. 
\end{abstract}

\section{Introduction}
\label{sec:introduction}
    Let $\mc{X} \equiv \mc{X}_N = \{x_1, x_2, \ldots, x_N\}$ denote a collection of $N$ items lying in a support set (or  alphabet) $\Sigma$, and define the population mean and variance associated with $\mc{X}_N$ as 
    \begin{align}
        \mu \equiv \mu_N = \frac{1}{N} \sum_{i=1}^N x_i, \quad \text{and} \quad \sigma^2 \equiv \sigma_N^2 = \frac{1}{N} \sum_{i=1}^N (x_i - \mu_N)^2.  \label{eq:population-mean-var}
    \end{align}
    For an integer $n < N$, let $X_1, X_2, \ldots, X_n$ denote $n$ observations drawn uniformly without replacement~(\wor) from $\mc{X}_N$; that is, 
        $X_1 \sim \mathrm{Uniform}(\mc{X}_N), \quad \text{and} \quad 
        X_i \sim \mathrm{Uniform}(\mc{X}_N \setminus \{X_1, \ldots, X_{i-1}\} ), \; \text{for } i \geq 2.$ 
    Given these observations and a confidence parameter $\alpha \in (0, 1]$, our goal is to construct a level-$(1-\alpha)$ confidence interval~(CI) for the unknown population mean. Formally, we want to construct a subset $C_n \equiv C_n(X_1, \ldots, X_n, \alpha)$, such that 
        $\mathbb{P}\lp \mu_N \in C_n \rp \geq 1-\alpha$, 
    for each fixed $n$. 
    Note that the only source of randomness in the above expression is due to the sampling scheme; the elements $\{x_1, \ldots,  x_N\}$ are themselves nonrandom quantities. 
    
    Our first goal in this paper is to understand the fundamental instance-dependent limits on the width achievable by any level-$(1-\alpha)$ CI. Focusing on the finite alphabet case (i.e., $\Sigma \subset [0,1], |\Sigma|<\infty$), we derive a new information-theoretic lower bound on the width of any level-$(1-\alpha)$ CI in~Theorem~\ref{theorem:lower-bound-1}, and then build upon this result to propose a new CI in~\Cref{subsec:new-CI-finite}. Both the lower and upper bounds depend on the \wor large deviation rate function terms, and in~\Cref{theorem:inverse-information-projection}, we show that in certain regimes~(that we refer to as \emph{moderate accuracy} regimes; see~Definition~\ref{def:moderate-accuracy}) the width of our proposed CI nearly matches the lower bound modulo constants~(see~Remark~\ref{remark:empirical-and-population-inverse-information-projections} for a more precise discussion). 
    Next, in~\Cref{theorem:dual-finite}, we derive a dual representation of these inverse rate function terms that makes their evaluation computationally feasible. These results employ elementary  type-counting arguments that rely crucially on the finiteness of $\Sigma$. To handle the case of continuous alphabets, such as $\Sigma = [0,1]$ or when $\Sigma$ is a smooth Banach space, we use a Bernoulli coupling identity to construct improved CIs with strong theoretical and empirical performance. 

     The construction of tight CIs for the population mean under \wor sampling is a fundamental problem that arises in many important applications. Some typical examples include risk-limiting audits~\citep{lindeman2012gentle}, survey sampling~\citep{cochran1977sampling}, subsampled Markov chain Monte Carlo~(MCMC) methods for scalable Bayesian inference~\citep{bardenet2014towards}, and large-scale kernel embedding~\citep{schneider2016probability}.  In~\Cref{subsec:banach-space-CI}, we further explore the kernel mean embedding application and empirically compare our proposed CI for smooth Banach spaces with that of~\citet{schneider2016probability}.

 \section{Preliminaries}
\label{sec:preliminaries}

    Let $\mathcal{X}_N = \{x_1, \ldots, x_N\}$ denote a collection of $N$ items $x_i \in \Sigma$ for all $i \in [N] \defined \{1, \ldots, N\}$, with population mean and variance denoted by $\mu_N$ and $\sigma_N^2$. Let $X^n = (X_1, \ldots, X_n)$ denote a \wor sample from the set $\mathcal{X}_N$ of size $n$. Denote the fraction $n/N$ by $\beta_N$, and let $\barbeta_N$ denote $1-\beta_N$. The subscripts $N$ are there to denote an implicit triangular array setup; whenever there is no confusion, we will drop the subscripts and simply use $\mathcal{X}, \mu, \sigma^2$, $\beta$ and $\barbeta$, respectively.  
    Given the \wor sample $X^n$, we can construct the empirical distribution over the alphabet $\Sigma$ as 
        $\Phat_n \equiv \Phat(X^n) = \frac{1}{n} \sum_{i=1}^n \delta_{X_i}$, 
    where $\delta_{X_i}$ denotes the point mass at $X_i$. 
     Many results of this paper rely on the following fundamental coupling identity that relates \wor sampling distribution $\Phat_n$ to a simpler \iid ``Bernoulli Sampling'' construction~\citep[\S~2]{hajek1960limiting}. 
    \begin{fact}[Key Coupling Identity]
        \label{fact:coupling} Suppose $\Sigma$ denotes an arbitrary alphabet,
        and let $X^n = (X_1, \ldots, X_n)$ denote a \wor sample from a finite population $\calX = \{x_1, \ldots, x_N\}$ with each $x_i \in \Sigma$. Now, let $J_1, \ldots, J_N$ denote $N$ \iid $\mathrm{Bernoulli}(\beta)$ random variables, where $\beta = n/N$. Let $\Phat_n \equiv \Phat(X^n) = \frac{1}{n} \sum_{i=1}^n \delta_{X_i}$ and $\Phat_{N, J} = \frac{1}{n} \sum_{i=1}^N J_i \delta_{x_i}$ denote the empirical measures associated with $X^n$ and $J^N$. Then, we have 
            $\Phat_n  \; =_d
            \; \Phat_{N, J} \mid \{S_{N, J} = n\}$,
        where $=_d$ indicates equality in distribution or law, and  $S_{N, J} \defined \sum_{i=1}^N J_i$. 
    \end{fact}   
    This identity allows us to employ the standard concentration techniques for independent data to study the behavior of \wor samples for general alphabets, and we illustrate its utility in~\Cref{sec:general-alphabet} by developing new CIs when $\Sigma$ is $[0,1]$ or a smooth Banach space. For finite alphabets, this identity yields a more directly applicable finite-sample approximation result that we discuss next. 
    
    In the  finite-alphabet setting~(when $\Sigma = \{s_1, \ldots, s_k\} \subset [0,1]$), the empirical distributions are also referred to as ``types'', following~\citet{csiszar2011information, dembo2009large}.  Let $\mathcal{T}_n \equiv \mathcal{T}_n(\Sigma)$ denote the set of all possible distinct types based on $n$ observations from a finite $\Sigma$; that is, 
         \begin{align}
            \mc{T}_n(\Sigma) = \left\{t: t= \Phat(x^n) \; \text{for some } x^n \in \Sigma^n \right\}, \quad \text{where} \quad \widehat{P}(x^n) = \frac{1}{n} \sum_{i=1}^n \delta_{x_i}.  
        \end{align}
    It is easy to verify that $|\mathcal{T}_n| = \binom{n+k-1}{k-1} \leq (n+1)^k$; see~\citet[Lemma 11.1.1]{thomas2006elements}. For any type $t \in \mathcal{T}_n$, let $S_t = \{x^n \in \Sigma^n: \Phat(x^n) = t\}$ denote the ``type-class'' of $t$; the collection of length-$n$ sequences in $\Sigma$ whose type is equal to $t$.
    Using Stirling's approximation, we can then obtain the following bound, which we will use often in~\Cref{subsec:lower-bound-finite} and~\Cref{subsec:new-CI-finite}. 
    \begin{fact}[Lemma 2.1.33 of \citet{dembo2009large}]
        \label{fact:dembo-zeitouni} Let $t$ denote any type in $\mathcal{T}_n(\Sigma)$, and let $\mc{X} = \{x_1, \ldots, x_N\}$ denote a collection of $N$ elements~(with possible repetition) from $\Sigma$. Let $\Phat_N \equiv \Phat(\calX) = \frac{1}{N} \sum_{i=1}^N \delta_{x_i}$ denote the type or empirical distribution associated with $\mathcal{X}$. Suppose $X^n = (X_1, \ldots, X_n)$ denotes a \wor sample from $\mathcal{X}$, and denote its type by $\Phat_n \equiv \Phat(X^n) = \frac{1}{n} \sum_{i=1}^n \delta_{X_i}$. Then we have 
        \begin{align}
            \left \lvert \frac{1}{n} \log \lp \mathbb{P}\lp \Phat_n = t \rp \rp + I\lp t, \frac{n}{N}, \Phat_N \rp \right\rvert \leq 2 \lp |\Sigma| + 1 \rp \lp \frac{\log(N+1)}{n} \rp, 
        \end{align}
        where 
            $$I(P, \beta, Q) \defined \dkl(P \parallel Q) + \frac{1-\beta}{\beta} \dkl\lp \frac{Q - \beta P}{1-\beta} \parallel Q\rp,$$ %
        with $\dkl(P \parallel Q) = \sum_{s \in \Sigma} P(s) \log P(s)/Q(s)$ for $P \ll Q$ denoting the relative entropy~(or KL divergence) between the two discrete distributions $P$ and $Q$ over $\Sigma$.  This definition is valid under the condition that $Q(s) - \beta P(s) \geq 0$ for all $s \in \Sigma$, which implies that $(Q-\beta P)/(1-\beta)$ is a valid probability distribution.
    \end{fact}
    This result establishes the probability of observing a type $t$ constructed using $X^n$ sampled \wor from a type $\Phat_N \equiv \Phat(\calX)$ is approximately $\exp \lp -n I\lp t, n/N, \Phat_N \rp \rp$. Thus, this rate function $I$ introduced in~Fact~\ref{fact:dembo-zeitouni} plays the same role as relative entropy does for  with-replacement sampling.

\paragraph{Prior Work.}
         In his seminal paper, \citet{hoeffding1963probability} showed that the expectation of a continuous convex function $f$ applied to a sum of i.i.d. random variables upper bounds the expectation of the same function applied to the sum of \wor samples. Formally, let $Y^n = (Y_1, \ldots, Y_n)$ denote \iid $\mathrm{Uniform}(\Xsample_N)$ random variables drawn with replacement, and let $X^n = (X_1, \ldots, X_n)$ denote a \wor sample drawn from the same $\Xsample_N$. Then, for any continuous and convex function $f$, 
                \(\mathbb{E}\lb f \lp \sum_{i=1}^n X_i \rp \rb  \leq \mathbb{E}\lb f \lp \sum_{i=1}^n Y_i \rp \rb.  \)
        Combining this inequality with the exponential deviation bound for $[0,1]$-valued \iid random variables~\citep[Theorem~1]{hoeffding1963probability} yields that for any $\epsilon>0$: 
        \(
            \mathbb{P}\lp \lv \frac{1}{n} \sum_{i=1}^n X_i - \mu\rv \geq \epsilon \rp  \leq 2 \exp \lp -2n\epsilon^2 \rp. 
        \)
        Inverting this concentration inequality, we obtain the classic Hoeffding $1-\alpha$ CI under \wor sampling: \begin{align}
            C_n^{(H)} = \lb \frac{1}{n} \sum_{i=1}^n X_i \pm w_n^{(H)} \rb, \quad \text{with} \quad w_n^H = \sqrt{\frac{\log(2/\alpha)}{2 n}}.  \label{eq:Hoeffding-CI}
        \end{align}
        An important feature in the \wor case, not present in the \iid setting, is that the uncertainty about the population mean reduces to $0$ as $n$ increases to $N$~(i.e., $\beta \uparrow 1$). Hoeffding CI stated in~\eqref{eq:Hoeffding-CI} does not capture this distinguishing feature of the \wor setting. 
        \citet{serfling1974probability} addressed this issue and sharpened  Hoeffding CI in~\eqref{eq:Hoeffding-CI} by introducing a finite-sample correction and replacing $n$ with $\frac{n}{1 - (n-1)/N}$. 
        \citet{bardenet2015concentration} obtained a further improvement of Serfling's result, replacing $(n-1)$ with $n$, and then used it to obtain variants of Hoeffding-Serfling inequality, Bernstein-Serfling, and Empirical-Bernstein-Serfling inequalities. 
        We state the  Bernstein-Serfling confidence interval~(CI) obtained by~\citet[Corollary~3.6]{bardenet2015concentration}, which has the additional benefit of being variance adaptive, unlike the Hoeffding-Serling CI.
        We have the following with probability at least $1-\alpha$ for some $\alpha \in (0,1]$, for all $t \geq n$: 
        \begin{align}
            &\mu \in \lb \frac{1}{n} \sum_{i=1}^n X_i \pm \frac{w_n^{(BM)}}{2} \rb, \\
            \quad \text{where} \quad &\frac{w_n^{(BM)}}{2} = \sigma \sqrt{ \frac{2(1-n/N)(1+1/n) \log(2/\alpha)}{n} } + \lp \frac{4}{3} + \sqrt{\lp \frac{N}{n+1}-1 \rp \lp 1- \frac{n}{N} \rp} \rp\frac{  \log(2/\alpha)}{n}. 
        \end{align}
        Since the above construction requires prior knowledge of  $\sigma^2$,~\citet[\S~4]{bardenet2015concentration} also derived an empirical version of this CI in which $\sigma^2$ is replaced by $\widehat{\sigma}_n^2$, the empirical variance, plus additional terms. 

        \citet{waudby2023estimating} proposed a betting-based formalism for constructing CIs~(and their time-uniform variants called confidence sequences) for the with- and without-replacement settings. In the \wor case, the CI after $n$ observation is defined as 
         \begin{align}
             C_n^{(bet)} = \{m \in [0,1]: W_n(m) < 1/\alpha\}, 
         \end{align}
        where the process $\{W_t(m): t \in [n]\}$ for any $m \in [0,1]$ is constructed to satisfy two properties: (i) at the true (unknown) population mean, the process $\{W_t(\mu): t \geq 1\}$ is a non-negative martingale with initial value $1$, and (ii) for $m \neq \mu$, the process grows rapidly. One way of constructing such a collection of stochastic processes is to interpret $\{W_t(m): t \geq 1\}$ as the evolution of the wealth of a gambler betting to disprove the claim that $\mu=m$. We refer the reader to~\cite{waudby2023estimating} for more details. 

        \paragraph{Overview of Our Results.}
        \label{subsec:overview}
         In this paper, we begin with a natural question not addressed in prior work: \begin{quote}Is there a fundamental lower bound on the width of \emph{reasonable} methods of constructing CIs for the population mean, and do there exist methods that achieve or approach this limit?\end{quote}
        In~Sections~\ref{subsec:lower-bound-finite} and~\ref{subsec:new-CI-finite} we partially resolve this question for the case of finite alphabets~(i.e., $|\Sigma| = k <\infty$). In the process, we identify the inverse of the \wor rate-function~(defined in Definition~\ref{def:complexity-function}) as the key complexity measure that characterizes the optimal CI width. In~\Cref{theorem:lower-bound-1}, we show that if there exists a method of constructing CIs that ensures the level-$(1-\alpha_N)$ coverage non-asymptotically in the \emph{moderate accuracy} regime~(Definition~\ref{def:moderate-accuracy}) for all populations $\calX_N$~(meaning of \emph{reasonable}), then its width must satisfy the following bound for some constant $\Alower>0$ and small enough $\alpha_N$~(see~Remark~\ref{remark:lower-bound-interpretation} for precise conditions): 
        \begin{align}
            |C_n(X^n, \alpha_N)| \gtrsim \frac{1}{2}\lp \gplus^{-1}\lp \frac{\Alower\log(1/\alpha_N)}{n} \rp - \gminus^{-1}\lp \frac{\Alower\log(1/\alpha_N)}{n}   \rp \rp, \qtext{where} g_{\pm}(\cdot) = J_{\pm}( \Phat_N, \beta, \cdot ), 
        \end{align}
         $J_\pm$ are the \wor rate functions in~Definition~\ref{def:complexity-function}, $\Phat_N$ is the population distribution, and $\beta = n/N \in (0, 1)$. 

        This lower bound  suggests a natural CI construction: replace the population inverse rate-function~($g_{\pm}^{-1}$) with their empirical versions~($\hat{g}_{\pm}^{-1}$) and add some slack in the argument to account for this. In~\Cref{prop:new-CI-finite}, we show that the resulting CI constructed using this idea contains the true mean with probability $1-\alpha_N$, and then in~\Cref{theorem:inverse-information-projection}, we show that for large enough values of $N$~(see~Remark~\ref{remark:n-large-enough}), the width of this CI satisfies the following bound with probability at least $1-\alpha_N$:
        \begin{align}
            |C_n(X^n, \alpha_N)| \leq     \lp \gplus^{-1}\lp \frac{\Aupper\log(1/\alpha_N)}{n} \rp - \gminus^{-1}\lp \frac{\Aupper\log(1/\alpha_N)}{n}   \rp \rp, 
        \end{align}
        for some $\Aupper > \Alower$, and assuming $\alpha_N$ is small enough~(see~Remark~\ref{remark:empirical-and-population-inverse-information-projections}). The leading constant~$(1/2)$ in the lower bound statement is most likely an artifact of our proof, and could possibly be removed (at least in the asymptotic setting) by translating the stability assumption of~\citet{deep2025asymptotic} to the \wor case. 
        
        Thus, these results establish the inverse rate-function as the fundamental complexity term for the optimal CI width under certain $(n, \alpha_N)$ regimes. 
        From a computational perspective, the construction of our proposed CI requires solving a convex program whose complexity increases with the alphabet size. To address this, we derive a dual formulation of the rate-function terms ($J_\pm$) in~\Cref{theorem:dual-finite}. This reduces the computation to a two-dimensional convex program independent of the alphabet size, making the procedure computationally tractable. Some numerical results reported in~Figure~\ref{fig:dual-vs-primal-comparison} verify the theoretic predictions.

        The derivations of the above results are based fundamentally on the finite alphabet version of~Fact~\ref{fact:coupling}, which we stated in~Fact~\ref{fact:dembo-zeitouni}. Although the approximations of~Fact~\ref{fact:dembo-zeitouni} are asymptotically tight, they may end up being quite loose in the finite-sample regime. Furthermore, our proposed CI is only valid for observations lying in a finite alphabet $\Sigma$, and does not admit a natural extension to continuous alphabets. For these reasons, while~Definition~\ref{def:CI-finite} yields a near-optimal CI in theory, its direct practical utility is limited. In the remainder of the paper, we develop two new practically useful CIs starting with the fundamental identity in~Fact~\ref{fact:coupling}. In~\Cref{subsec:asymp-valid-CI}, we propose an \emph{almost sure} CI that lies midway between nonasymptotic CIs~(such as Bernstein-Serfling) and CLT-based asymptotic CIs in terms of validity and performance, while in~\Cref{subsec:banach-space-CI}, we use~Fact~\ref{fact:coupling} to construct a simple nonasymptotic CI for observations lying in a $(2, D)$-smooth Banach space. We validate the practical usefulness of these CIs via some numerical experiments.

\section{Main Results}
\label{sec:main-results}
\label{sec:finite-alphabet}

    In the with-replacement setting, \citet{shekhar2023near} showed that the width of the CIs for bounded means is characterized by certain inverse relative entropy (or with-replacement large deviation rate function) terms.
    In this section, we will derive similar results for the without-replacement setting, in terms of \wor rate function terms based on $I$, as defined below. 
    \begin{definition}
        \label{def:complexity-function} 
        Let $P$ denote any probability distribution supported on $\Sigma \subset [0,1]$ with $|\Sigma| = k$, and fix a constant $\beta \in (0, 1)$. Then, for any $m \in [0,1]$, we define the following two rate function terms: 
        \begin{align}
            &\Jplus(P, \beta, m) \defined \inf \left\{ I(P, \beta, Q): \min_{s \in \Sigma} Q(s)/P(s) \geq \beta, \; \text{and}\; \mu_Q \geq m  \right\}, \label{eq:information-projection-plus} \\
            &\Jminus(P, \beta, m) \defined \inf \left\{ I(P, \beta, Q): \min_{s \in \Sigma} Q(s)/P(s) \geq \beta, \; \text{and}\; \mu_Q \leq m  \right\}. \label{eq:information-projection-minus}
             \label{eq:information-projection-minus}
        \end{align}
        In both these definitions, $\mu_Q$ denotes the mean value associated with the distribution $Q$;  and the infimum is over all the distributions supported on $[0,1]$ such that $Q \gg P$ with $\min_{s \in \Sigma} Q(s)/P(s) \geq \beta$ with the corresponding mean constraints. 
    \end{definition}

    \begin{remark}
        \label{remark:J} \sloppy Note that for a fixed distribution $P$ with mean $\mu_P$, the term $\Jplus(P, \beta, m)$ is equal to zero for all $m \leq \mu_P$, and $\Jminus(P, \beta, m)$ is equal to zero for all $m \geq \mu_P$. Thus, the maximum of two is $J(P, \beta, m) := \max \{\Jplus(P, \beta, m), \Jminus(P, \beta, m)\} = \boldsymbol{1}_{m \geq \mu_P} \Jplus(P, \beta, m) + \boldsymbol{1}_{m < \mu_P} \Jminus(P, \beta, m)$. 
    \end{remark}
    The next two subsections~(\Cref{subsec:lower-bound-finite} and~\Cref{subsec:new-CI-finite}) present the results for the finite alphabet case, and the final subsection is concerned with CIs for more general alphabets.

    \subsection{Lower Bound for Finite Alphabet}
    \label{subsec:lower-bound-finite}
        Consider a shrinking sequence of confidence levels, denoted by $\{\alpha_N \in (0, 1]: N \geq 1\}$, with $\alpha_N \to 0$ as $N \to \infty$, and introduce the following terms:
        \begin{align}
            r_N = (|\Sigma| + 1)\log(N+1), \quad \text{and} \quad 
            a_N = \log(1/2\alpha_N) - r_N. \label{eq:rN-aN}
        \end{align}
        Here, $|\Sigma|$ denotes the alphabet size, which we assume to be finite, and $N$ denotes the population size. In this section and the next, we will assume that the confidence parameter $\alpha_N$ decays at a polynomial rate with $N$, and we refer to such CIs as \emph{moderately accurate CIs}. 
        {
        \begin{definition}[Moderate Accuracy Regime]
            \label{def:moderate-accuracy}
            Consider a sequence of populations $\{\calX_N: N \geq 1\}$ supported on finite alphabets $\{\Sigma_N: N \geq 1\}$ with means $\{\mu_N: N \geq 1\}$. For each $N$, let $(X_{N,1}, \ldots, X_{N,n})$ with $\beta_N = n/N$ denote a \wor sample, and let $C_n$ denote a CI for the mean $\mu_N$ constructed based on $(X_{N,1}, \ldots, X_{N,n})$. We  say that $\{C_n: n=\beta_N N, \; N \geq 1\}$ is a sequence of  \emph{moderate accuracy} CIs if 
            \begin{align}
                \mathbb{P}\lp \mu_N \in C_n \rp \geq 1-\alpha_N, \qtext{and} \lim_{N \to \infty} \frac{(|\Sigma_N| + 1) \log (N+1)}{\log(1/2\alpha_N)} = 0. 
            \end{align}
            If the alphabet does not change with $N$~($\Sigma_N = \Sigma$ for all $N \geq 1$), then this is equivalent to the condition that $\alpha_N$ decays super-polynomially in $N$; that is,   $\alpha_N = o \lp N^{- c}\rp$ for all $c>0$. 
        \end{definition}
        }
        Our first main result  establishes a fundamental lower bound on the {width of moderately accurate CIs}. 
       \begin{theorem}[Width lower bound]
            \label{theorem:lower-bound-1}
            Let $\mc{C}$ denote a method of constructing confidence intervals for the population mean $\mu_N$ of $\mc{X}_N = \{x_1, \ldots, x_N\} \subset \Sigma^N$. With $\Phat_N \equiv \Phat(\calX_N) \defined \frac{1}{N} \sum_{x \in \mathcal{X}_N} \delta_{x}$,  define
            \begin{align}
                W_n(X^n, \widehat{P}_N, \alpha_N) \coloneqq  |\mc{C}(X^n, \widehat{P}_N, \alpha_N)|.
            \end{align}
            Let $\widetilde{P}_n = \argmin_{t \in \calT_n} I(t, \beta, \Phat_N)$ denote the element of $\mc{T}_n$ closest to $\widehat{P}_N$ in $I$, and define $b^*_+$ and $b^*_-$ as 
            \begin{align}
                 &  b_+^* := \inf\left\{ m \geq 0: \Jplus(\widetilde{P}_n, \beta, m) \geq \frac{a_N}{n} \right\}, 
                \quad \text{and} \quad  b_-^* := \sup\left\{ m \leq 1: \Jminus(\widetilde{P}_n, \beta, m) \geq \frac{a_N}{n} \right\}.
            \end{align}
            Then, for large enough $N$ {with $\alpha_N$ in the moderate accuracy regime~(see~Remark~\ref{remark:N-large-enough-lower-bound} for precise conditions)}, we have with probability at least $1-\alpha_N$, 
            \begin{align}
                W_n(X^n, \widehat{P}_N, \alpha_N) \geq \max \{b^*_+ - \mu_N, \; \mu_N - b^*_-\} \geq \frac{b^*_+ - b^*_-}{2}. \label{eq:lower-bound-finite}
            \end{align}
            Note that the same bound holds for the worst-case width $\max_{x^n \subset \calX_N}|\calC(x^n, \Phat_N, \alpha_N)|$ with probability $1$. 
       \end{theorem}
    \noindent\emph{Proof outline of~\Cref{theorem:lower-bound-1}.}
        The starting point of the argument is the standard duality between confidence intervals and hypothesis tests. In particular, we consider null and alternative population distributions $P$ and $Q$ with means $\mu_P$ and $\mu_Q > \mu_P + w$, where $w$ is such that $\mathbb{P}_P(|C_n| > w) \leq \alpha_N$.  Then, given a level-$(1-\alpha_N)$ CI constructed using the \wor sample $X^n$, denoted by $C_n$, we define a hypothesis test $\Psi$ that rejects the null~(i.e., $P$) if $\mu_Q \in C_n$. The level-$(1-\alpha_N)$ property of the CI implies that its type-I and type-II errors are  at most $2\alpha_N$ and $\alpha_N$ respectively. Next, we show that for every test $\Psi$ with errors at most $(2\alpha_N, \alpha_N)$, there exists a test $\Psi'$ based only on the type~(and not the order of observations) with errors bounded above by $(4\alpha_N, 2\alpha_N)$. Finally, using Fact~\ref{fact:dembo-zeitouni}, we establish that for large enough values of $N$~(precise conditions in Remark~\ref{remark:N-large-enough-lower-bound}), the type $\tilde{P}_n \in \calT_n$ closest to $P$ in terms of $I$ must lie in the acceptance region of any test whose type-I error is smaller than $4\alpha_N$. On some simplification, this leads to the stated lower bounds on the values of $w$.  The details are in~\Cref{proof:lower-bound-1}.

       \begin{remark}
        \label{remark:lower-bound-interpretation}
           \sloppy The {moderate accuracy assumption}  implies that $\log(1/2\alpha_N) - r_N \approx \log(1/2\alpha_N)$ for large enough $N$. This in turn means that $b^*_+ \approx J_+(\widehat{P}_N, \beta, \cdot)^{-1}\lp \log(1/2\alpha_N)/n \rp$ and $b^*_- \approx J_-(\widehat{P}_N, \beta, \cdot)^{-1}\lp \log(1/2\alpha_N)/n \rp$ in this regime. 
           Motivated by the above result, we now propose a new CI whose end points are empirical versions of $b^*_-$ and $b^*_+$ introduced in~\Cref{theorem:lower-bound-1}. 
       \end{remark}
       
    \subsection{Proposed CI for Finite Alphabet}
    \label{subsec:new-CI-finite}
        Having established the lower bound on the width of any CI in~\Cref{theorem:lower-bound-1}, we now use these insights to construct new a level-$(1-\alpha_N)$ CI for $\mu_N$. As before, assume that $\calX_N = \{x_1, \ldots, x_N\}$ take values in the finite alphabet $\Sigma$, and define a new CI based on empirical inverse rate functions as follows. 
        \begin{definition}
            \label{def:CI-finite} 
            Let $X^n = (X_1, \ldots, X_n)$ denote a random sample of size $n = \beta N$ from $\mc{X}_N = \{x_1, \ldots, x_N\}$ drawn uniformly without replacement. Denote the types (or empirical distributions) of $\mc{X}_N$ and $X^n$ with $\widehat{P}_N \equiv \Phat(\calX_N)$ and $\widehat{P}_n \equiv \Phat(X^n)$ respectively. For any $\alpha_N>0$,  define the confidence interval 
            \begin{align}
                &C_n \equiv C_n(X^n, \alpha_N, \beta) = \lb b_-, b_+ \rb, 
            \end{align}
            where  $b_- \equiv b_-(\Phat_n, c_N)$ and $b_+ \equiv b_+(\Phat_n, c_N)$ are defined as follows, with $c_N =  \log(2/\alpha_N) + 2r_N$: 
            \begin{align}
                &b_-(\Phat_n, c_N) = \sup\left\{ m \leq 1: J_-(\widehat{P}_n, \beta, m) \geq \frac{c_N}{n} \right\},  
                 ~ b_+(\Phat_n, c_N) = \inf\left\{ m \geq 0: J_+(\widehat{P}_n, \beta, m) \geq \frac{c_N}{n} \right\}. 
            \end{align}
           We use the convention that $\inf \emptyset = 1$ and $\sup \emptyset = 0$, so that $C_n$ is always a subset of $[0,1]$.  
        \end{definition}
        Recall that the two rate-function terms, $J_-$ and $J_+$,  were introduced in~Definition~\ref{def:complexity-function}, and  $r_N$ was defined in~\eqref{eq:rN-aN}. Intuitively, the first term in $c_N$~(i.e., $\log(2/\alpha_N)$) comes from the usual two-sided tail bound, and the remaining $2r_N$ term accounts for the finite alphabet probability approximation result in~Fact~\ref{fact:dembo-zeitouni}, along with a union bound over all types. 
        Our next result shows that the CI introduced above in~Definition~\ref{def:CI-finite} is a valid level-$(1-\alpha_N)$ CI for the population mean $\mu \equiv \mu_N$ associated with the finite population $\mc{X}_N$. 
        \begin{theorem}
            \label{prop:new-CI-finite} 
            For all $N \geq 1$, the set $C_n$ from Definition~\ref{def:CI-finite} is a $(1-\alpha_N)$ CI for $\mu_N$.
        \end{theorem}
        The proof of this result is given in~\Cref{proof:new-CI-finite}. 
        We emphasize that~\Cref{prop:new-CI-finite} holds for every $N, n, \alpha_N$, avoiding the moderate accuracy condition. 

        \begin{remark}
            \label{remark:finite-CI}
            If  confidence parameter $\alpha_N$ satisfies the \emph{moderate accuracy} condition of~Definition~\ref{def:moderate-accuracy}, such that $\log(1/2\alpha_N) + 2 r_N \asymp \log(1/2\alpha_N)$, then the width of the CI above matches the lower bound obtained in~\Cref{theorem:lower-bound-1}, with two differences: (i) there is an additional factor of $2$ compared to~\eqref{eq:lower-bound-finite}, and (ii) the terms $b_-$ and $b_+$ are defined using $\widehat{P}_n$ instead of $\widehat{P}_N$ that was used in the definitions of $b^*_-$ and $b^*_+$. The factor $2$ is likely an artifact of our lower-bound proof and it should be possible to remove it for large $N$ values, for example, using the arguments developed by~\citet{deep2025asymptotic}. 
        \end{remark}
        In our next result, we analyze the relation between the population and the empirical inverse rate function terms. 
        \begin{theorem}
        \label{theorem:inverse-information-projection}
        Let $g_{\pm}$ and $\widehat{g}_{\pm}$ denote the population and empirical inverse rate functions: 
        \begin{align}
            &\gplus(x) = \inf \{m \geq 0: \Jplus(P, \beta, m) \geq x \}, \qtext{and}  
            \ghatplus(x)  = \inf \{m \geq 0: \Jplus(\Phat_n, \beta, m) \geq x \}, \\
            &\gminus(x) = \sup \{m \leq 1: \Jminus(P, \beta, m) \geq x \}, \qtext{and} 
            \ghatminus(x)  = \sup \{m \leq 1: \Jminus(\Phat_n, \beta, m) \geq x \}.
        \end{align}
        Then, with probability at least $1-\alpha_N$, we have the following with $A = 1 + 2 / \sigma^2$, and large enough values of $n$~(see~Remark~\ref{remark:n-large-enough} for the precise meaning of ``$n$ large enough''): 
        \begin{align}
            &b_+(\Phat_n) = \ghatplus\lp \frac{c_N}{n} \rp \leq \gplus\lp A \times \frac{c_N}{n} \rp, \qtext{and}
            b_-(\Phat_n) = \ghatminus\lp \frac{c_N}{n} \rp \geq \gminus\lp A \times \frac{c_N}{n} \rp. 
        \end{align}
        \end{theorem}

        \noindent \emph{Proof outline of~\Cref{theorem:inverse-information-projection}.}
        To prove this result, we first show in Lemma~\ref{lemma:coarse-CI} that $C_n =[b_-, b_+] \subset \Ctilde_n \coloneqq [\muhat_n \pm \sqrt{\barbeta c_N/2n}]$; that is, the CI constructed in Definition~\ref{def:CI-finite} is contained in a larger CI of half-width $\sqrt{\barbeta c_N/2n}$. Although this new CI is wider than the lower bound, it is important for certain approximation arguments used in rest of the proof. 
        Lemma~\ref{lemma:Taylor-expansion-of-I} and~\ref{lemma:Jplus-upper-bound} show that for small enough $\delta >0$, we have: 
            \begin{align}
                \Jplus(P, \beta, \mu_P + \delta) \leq \frac{\delta^2}{2 \barbeta \sigma^2} + \calO\lp \delta^3 \rp. 
            \end{align}
            This cubic approximation allows us to conclude that with probability $1-\alpha_N$, we have 
            \begin{align}
                \Jplus(\Phat_n, \beta, m) \geq \Jplus(P, \beta, m) - (A-1) \frac{c_N}{n}, \qtext{for all} m \in \Ctilde_n, 
            \end{align}
        yielding the result. Repeating these steps for $\Jminus$ gives the lower bound on $b_-$. The details are in~\Cref{proof:inverse-information-projection}.

        \begin{remark}
            \label{remark:empirical-and-population-inverse-information-projections} 
            Assuming that $\alpha_N$ is small enough such that $\log(2/\alpha_N) + 2 r_N \asymp \log(2/\alpha_N)$, we have the following with some constants $\Aupper > \Alower>0$, and $u_N = \log(2/\alpha_N)/n$: 
            \begin{align}
                \frac{\gplus(\Alower u_N ) - \gminus(\Alower u_N)}{2} \; \leq \; |C_n(X^n, \alpha_N)|\; \leq \;  \gplus(\Aupper u_N) - \gminus(\Aupper u_N). 
            \end{align}
            Thus, our proposed CI width matches the information-theoretic lower bound of~\Cref{theorem:lower-bound-1}, up to the leading  factor $(1/2)$ and the constant $\Alower < \Aupper$ inside the argument of the inverse rate function terms $g_{\pm}$. 
        \end{remark}

    \paragraph{A Duality Result.}
        Since the definitions of $J_+$ and $J_-$ in~Definition~\ref{def:CI-finite} involve optimizing over the entire probability simplex, direct evaluation of these terms can become computationally infeasible with increasing $|\Sigma|$. 
        To address this, we present a dual representation of these rate functions that restates them as two-dimensional convex programs. For simplicity, we assume that the end points $0, 1$ do not belong to $\Sigma$ below.  
        \begin{theorem}
            \label{theorem:dual-finite}
            Let  $P$ denote a probability distribution supported on a finite set $\Sigma \subset (0,1)$,  and consider any $m \in [0,1]$ and $\beta \in (0, 1)$ such that $\beta \mu_P + \barbeta > m$, where  $\mu_P = \mathbb{E}_{X \sim P}[X] = \sum_{s \in \Sigma} s \times p(s)$ and $\barbeta = 1-\beta$.
            Then, we have the following duality result: 
            \begin{align}
            J_+(P, \beta, m) & = \inf\; \left\{  I(P, \beta, Q) \;\boldsymbol{:}\; {Q \in \mathcal{P}([0,1]),\, \mathbb{E}_Q[X] \geq m} \right\} \label{eq:primal-problem}\\
            & = \sup_{(\lambda, \rho) \in \dualdomain} \left\{\mathbb{E}_{P}\lb \log \lp \frac{1 - \barbeta e^{\beta \rho} e^{\beta \lambda X}}{\beta} \rp \rb + \lambda \lp  m - \beta \mu_P \rp + \rho \barbeta \right\}, \label{eq:dual-problem}
        \end{align}
        where $\mathcal{D}_D$ denotes the following dual feasible set for the dual optimization problem: 
        \begin{align}
            \mathcal{D}_D = \left\{ (\lambda, \rho) \; \boldsymbol{:}\;  \lambda \geq 0, \, \rho \in \mathbb{R},  \rho + \lambda \leq \frac 1 {\beta} \log \frac 1 {\barbeta} \right\}. \label{eq:dual-set-Jplus}
        \end{align}
        A similar result holds for $\Jminus$ since $\Jminus(P, \beta, m) = \Jplus(P', \beta, 1-m)$, with $P' = P \circ f^{-1}$ for $f(x) = 1-x$. 
        \end{theorem}
        The proof of this result is in~\Cref{proof:dual-finite}, and it begins by observing that we can restrict our attention to a finite dimensional subset of the domain of the primal problem stated in~\eqref{eq:primal-problem} without loss of optimality. The rest of the proof then follows by some standard finite dimensional convex duality arguments. 

        \begin{remark}
            \label{remark:duality-with-end-points} In the statement of~\Cref{theorem:dual-finite}, we have assumed that the alphabet $\Sigma$ does not contain the end points $\{0,1\}$. If we allow the alphabet $\Sigma$ to contain the end points, the expression of the dual optimization problem stated in~\Cref{theorem:dual-finite} does not change: it is still the same two-parameter concave maximization over the domain $\dualdomain$. However, some steps of the proof as described in~\Cref{proof:dual-finite}, mainly~Lemma~\ref{lemma:dual-finite-1} and~Lemma~\ref{lemma:dual-finite-2}, will need to be  modified accordingly.
        \end{remark}
        \begin{remark}
            As mentioned earlier, this duality result is  practically important: it reduces the task of evaluating the rate function to that of solving a convex optimization problem over a two dimensional domain~$[0, \infty)\times \mathbb{R}$. Importantly, the dimension of the optimization domain does not increase as the support size $|\Sigma|$ increases. This means that our confidence interval proposed in~\Cref{subsec:new-CI-finite} can be constructed in a computationally feasible manner, by using off-the-shelf optimization programs to estimate rate functions. 
        \end{remark}

{
    \begin{remark}
        \label{remark:ganguly-sutter} Our approach for constructing the CI in~Definition~\ref{def:CI-finite} has some parallels with the moderate deviation principle~(MDP)-based  approach developed by~\citet{ganguly2023optimal}, who developed a general framework for CIs by inverting the MDP rate functions for certain estimators. The intervals so constructed are then shown to satisfy an asymptotic notion of minimality as described in~\citet[Problem 3.7]{ganguly2023optimal}. In contrast to their work, we invert the exact rate function of the empirical distribution $\Phat_n$ in the finite-alphabet setting to obtain a nonasymptotic   level-$(1-\alpha_N)$ guarantee, owing to standard type-based approximation results for finite alphabets.  
    \end{remark}
    }

\subsection{CIs for General Alphabets}
\label{sec:general-alphabet}
    The CI constructed in~\Cref{subsec:new-CI-finite} has three main drawbacks. First,  it uses the combinatorial bound of~Fact~\ref{fact:dembo-zeitouni} that makes it too conservative in practice. Second, it applies on the case of finite $\Sigma$, which strongly limits its scope. Finally, its theoretical guarantees are valid only for small values of $\alpha_N$. In this section, we illustrate how the general Bernoulli coupling idea of~Fact~\ref{fact:coupling} can still be useful in constructing more practical CIs that can relax one or more of these conditions. 

     To illustrate the basic idea, consider a population $\calX_N = \{x_1, \ldots, x_N\}$  with elements in an arbitrary alphabet $\Sigma$,  and let $\{J_i: i \in [N]\}$ denote \iid $\Bernoulli(\beta)$ random variables. Fact~\ref{fact:coupling} tells us that conditioned on $S_{N,J} = \sum_{i=1}^N J_i$ being equal to $n = \beta N$, the law of $\Phat_{N,J} = \frac{1}{n} \sum_{i=1}^N J_i \delta_{x_i}$ is equal to that of the \wor distribution $\Phat_n = \frac{1}{n} \sum_{i=1}^n \delta_{X_i}$. 
     Hence, for any subset $E \subset \calP(\Sigma)$, Fact~\ref{fact:coupling} implies the following~(with $\Phat_n = \Phat(X^n)$,  $\Phat_{N,J} = (1/n) \sum_{i=1}^N J_i \delta_{x_i}$, and $S_{N,J} = \sum_{i=1}^N J_i$):  
    \begin{align}
        \mathbb{P}\lp \Phat_n \in E \rp &= \mathbb{P}\lp \Phat_{N, J} \in E \middle \vert S_{N,J}=n \rp = \frac{\mathbb{P}\lp \Phat_{N, J} \in E, \, S_{N,J}=n \rp}{\mathbb{P}\lp S_{N,J}=n \rp} \leq 
        \frac{\mathbb{P}\lp \Phat_{N, J} \in E\rp}{\mathbb{P}\lp S_{N,J}=n \rp} \\
        & \leq \frac{e^2}{\sqrt{2 \pi}} \sqrt{\beta (1-\beta) N}\, \mathbb{P}\lp \Phat_{N, J} \in E \rp \leq 1.18 \sqrt{\beta (1-\beta) N}\, \mathbb{P}\lp \Phat_{N, J} \in E \rp. \label{eq:coupling-upper-bound-1}
    \end{align}
    The bound on $1/\mathbb{P}(S_{N,J} = n)$ follows from the fact that $S_{N, J} \sim \mathrm{Binomial}(N, \beta)$, and by employing Stirling's inequalities for the factorial function. The details are in~\Cref{proof:coupling-upper-bound-1}. 

    Since the \iid empirical distribution~$\Phat_{N,J}$ can be analytically more tractable than the \wor distribution $\Phat_n$, it can help us design new CIs that complement or improve upon existing CIs for more general $\Sigma$, even when taking the $1.18\sqrt{(1-\beta)n}$-factor into account. In particular, we discuss the following two cases: 
    \begin{itemize}
        \item In~\Cref{subsec:asymp-valid-CI}, we construct an ``almost sure'' CI, as defined by~\citet{naaman2016almost}, for the case of $\Sigma = [0,1]$. In terms of assumptions and validity, this new CI lies between non-asymptotic CIs, such as those proposed by~\citet{hoeffding1963probability, serfling1974probability, bardenet2015concentration, waudby2020confidence}, and the asymptotic CI constructed using the central limit theorem~(CLT) for without replacement case~\citep{hoeffding1951combinatorial}. We discuss this further in~Remark~\ref{remark:asymp-CI}. 
        \item In~\Cref{subsec:banach-space-CI}, we consider the case of $\Sigma$ being a $(2, D)$-smooth Banach space, and illustrate how the approach based on~\eqref{eq:coupling-upper-bound-1} can be used to construct a CI whose width is tighter than the existing CI of~\citet{schneider2016probability} in normal parameter ranges~(see~Figure~\ref{fig:width-comparison}). 
    \end{itemize}
    These examples illustrate how the same Bernoulli coupling identity of~Fact~\ref{fact:coupling}, underlying the construction of the near-optimal CI for finite alphabets, can also be used in designing CIs for  more general alphabets. 
\begin{figure}
    \centering
    \includegraphics[width=0.45\linewidth]{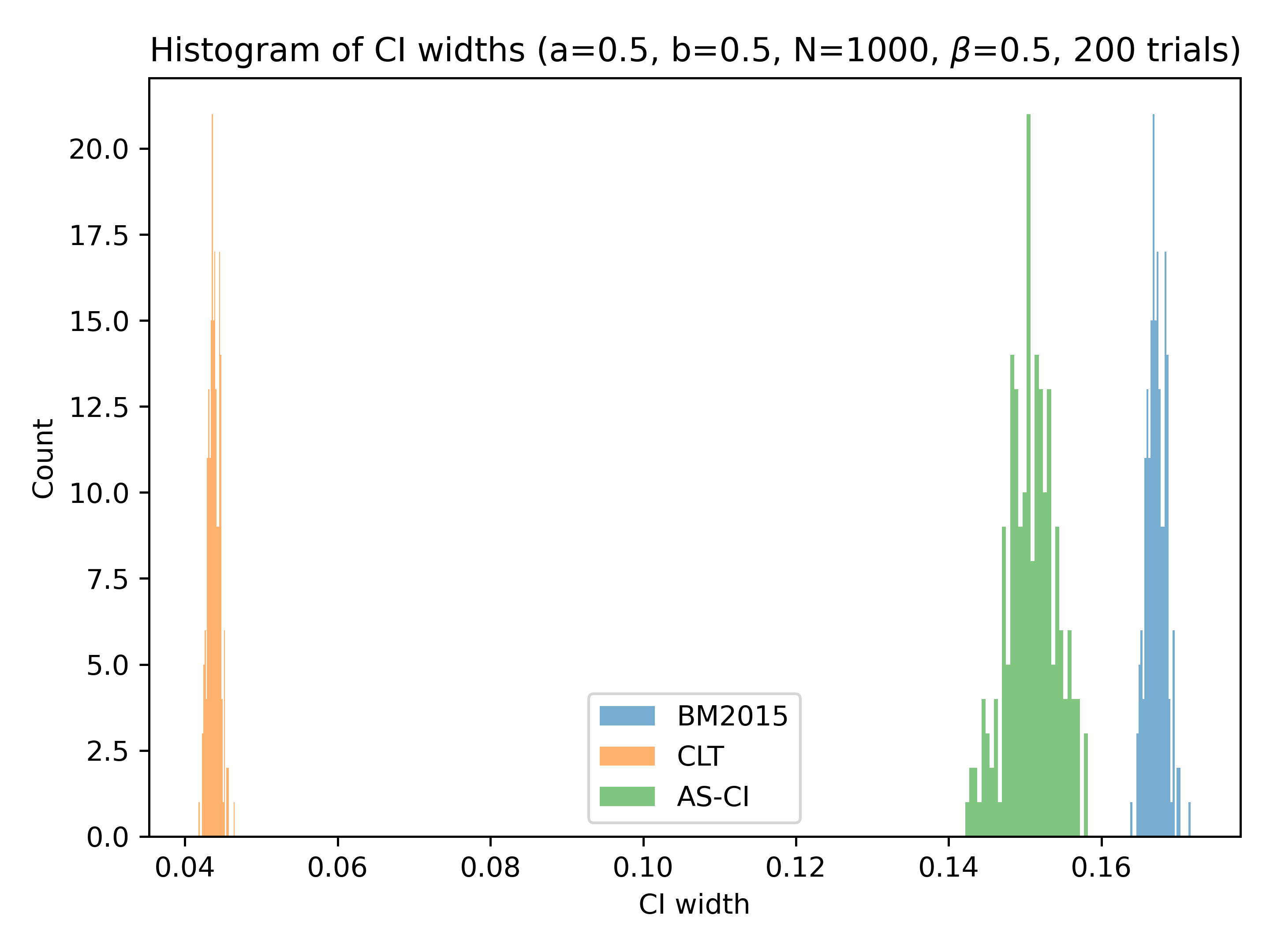}
    \includegraphics[width=0.45\linewidth]{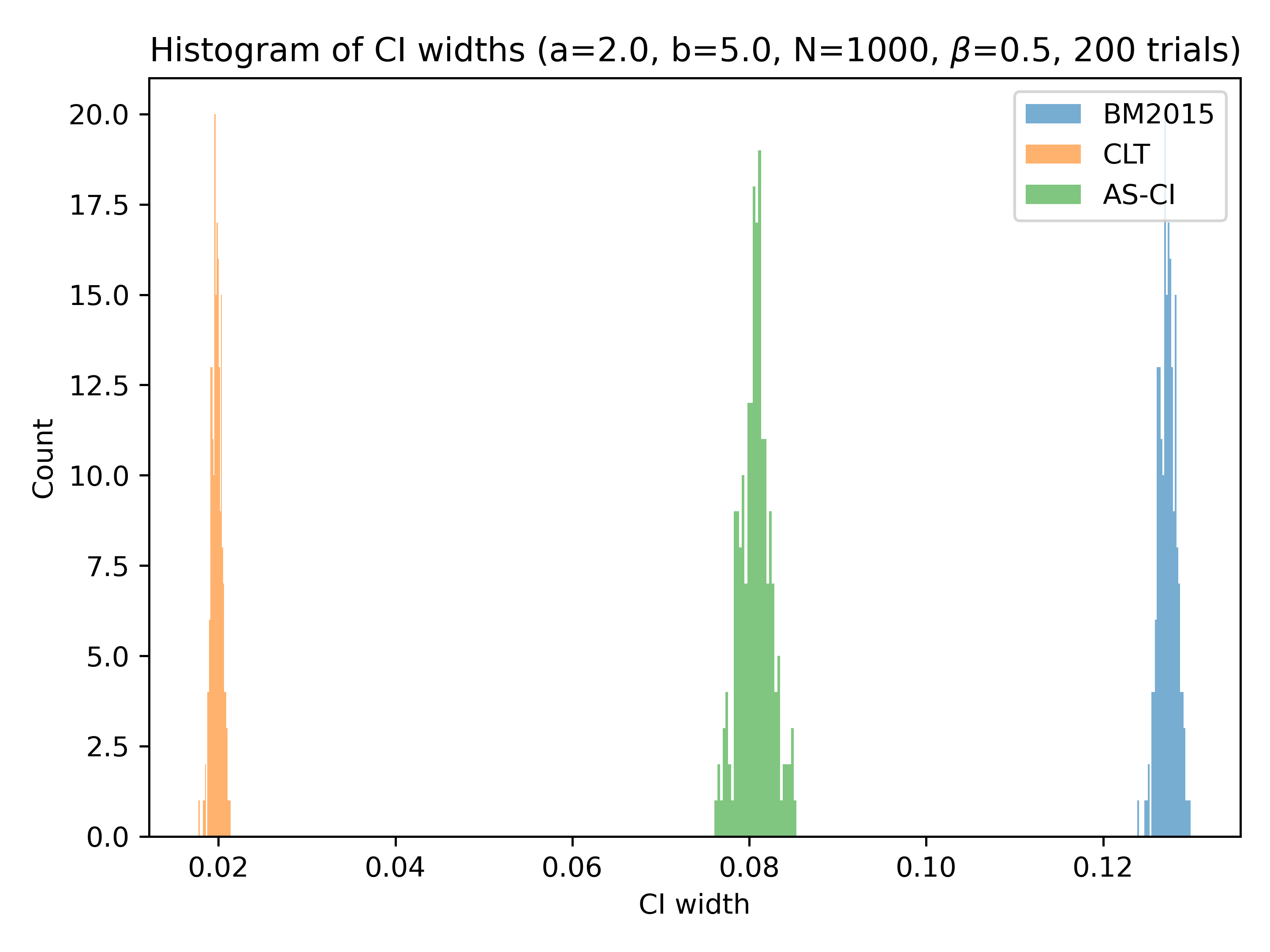}
    \caption{Comparison of the distribution of widths~(over $200$ trials) of the almost sure CI of~\eqref{eq:asympCI-def}~(denoted by AS-CI), the CLT-based asymptotically valid CI, and the empirical Bernstein CI derived by~\citet{bardenet2015concentration} using populations of size $N=1000$ generated using $\mathrm{Beta}(a,b)$ distributions, with \wor samples of size $500$. In both examples the size of our almost sure CI lies midway between the asymptotically valid CLT-CI, and the nonasymptotic CI of~\citet{bardenet2015concentration}, denoted by BM2015.}
    \label{fig:placeholder}
\end{figure}
\subsubsection{\texorpdfstring{An ``Almost Sure'' CI for $\mathbf{\Sigma = [0,1]}$}{An ``Almost Sure'' CI for Sigma=[0,1]}}
\label{subsec:asymp-valid-CI}
We now turn to the simplest continuous alphabet $\Sigma = [0,1]$, although the results of this section are equally valid for any bounded interval $[a,b]$. We will work under the following assumption. 
\begin{assumption}
    \label{assump:limiting-distribution} Let $(\calX_N)_{N \geq 1}$ denote a collection of finite subsets of $ \Sigma = [0,1]$, such that $\Phat_N \equiv \Phat(\calX_N) \to P_X \in \calP(\Sigma)$ weakly as $N \to \infty$, with $\mathbb{E}_{X \sim P_X}[X^2] = m_2>0$. 
    If $\mu_N$ denotes the mean $\mathbb{E}_{\Phat_N}[1]$, then this also implies that there exists a $\mu_X \in [0,1]$ such that $\lim_{N \to \infty} \mu_N = \mu_X$. 
    We assume that $X^n = (X_1, \ldots, X_n)$ denotes a \wor sample of size $n = \beta_N N$, with $\lim_{N \to \infty} \beta_N = \beta$ for some $\beta \in (0, 1)$. 
\end{assumption}
As mentioned earlier, our starting point is the inequality~\eqref{eq:coupling-upper-bound-1} based on~Fact~\ref{fact:coupling}. In particular, for some $\epsilon>0$ to be specified later and for any $\lambda \geq 0$, we have the following:
\begin{align}
    \mathbb{P}\lp \frac{1}{n} \sum_{i=1}^n X_i \geq \mu_N + \epsilon  \rp &\leq 1.18 \sqrt{\barbeta n} \, \mathbb{P}\lp \sum_{i=1}^N J_i x_i - n \mu_N   \geq n \epsilon \rp = 1.18 \sqrt{\barbeta n} \, \mathbb{P}\lp \sum_{i=1}^N (J_i - \beta_N) x_i    \geq n \epsilon \rp \\
    & \leq 1.18 \sqrt{\barbeta n} \exp \lp -n \lambda  \epsilon + \sum_{i=1}^N \log \lp \mathbb{E}\lb e^{\lambda (J_i - \beta_N)x_i} \rb \rp \rp.  \label{eq:asympCI-1}
\end{align}
Since each $J_i \sim \Bernoulli(\beta_N)$, we can calculate the expectations in~\eqref{eq:asympCI-1} exactly as $\mathbb{E}[e^{\lambda (J_i-\beta_N) x_i}] = \beta_N e^{\lambda \barbeta_N x_i} + \barbeta_N e^{-\lambda \beta_N x_i}$. Introduce the term 
\begin{align}
    \Lambda_N(\lambda) = \int f_{N, \lambda}(x) d\Phat_N(x), \quad \text{for} \quad  
    f_{N, \lambda}(x) \defined \log\lp \beta_N e^{\lambda \barbeta_N x} + \barbeta_N e^{-\lambda \beta_N x} \rp, \label{eq:asympCI-2}
\end{align}
and observe that we can plug this into~\eqref{eq:asympCI-1} to get 
\begin{align}
    \mathbb{P}\lp \bar{X}_n \geq \mu_N + \epsilon \rp \leq 1.18 \sqrt{\barbeta_N n} \exp \lp -n\lp \lambda \epsilon - \frac{1}{\beta_N} \Lambda_N(\lambda) \rp \rp, 
\end{align}
where we use $\bar{X}_n$ to denote the \wor sample mean $(1/n) \sum_{i=1}^n X_i$. 
Optimizing the expression above with respect to $\lambda$, and denoting by $\Lambda^*_N$ the convex conjugate of $\Lambda_N$, we get 
\begin{align}
\mathbb{P}(\bar{X}_n \geq \mu_N + \epsilon) \leq 1.18 \sqrt{\barbeta n} \exp \lp - N \Lambda^*_N(\beta_N \epsilon) \rp. 
\end{align}
Thus, for any $\alpha_N \in (0, 1)$,  we can invert the above bound to get 
\begin{align}
    &\mathbb{P}\lp |\bar{X}_n - \mu_N| > \epsilon_N \defined \frac{(\Lambda^*_N)^{-1}\lp t_{N,\alpha_N} \rp}{\beta_N}  \rp \leq \alpha_N,  
    \;\text{where}\;
      t_{N, \alpha_N} = \frac{1}{N} \log \lp \frac{2.36 \sqrt{(1-\beta_N) n}}{\alpha_N} \rp. \label{eq:asympCI-3}
\end{align}
 This implies the following  non-asymptotic level-$(1-\alpha_N)$ CI for the population mean $\mu_N$: 
\begin{align}
    \Corc_N \equiv \Corc_N(\alpha_N) = \lb \bar{X}_n - \epsilon_N, \bar{X}_n + \epsilon_N \rb.  \label{eq:oracle-asympCI}
\end{align}
This ``oracle'' CI depends on $\Lambda^*_N$ which is defined using the unknown distribution $\Phat_N$. Since we only have access to the \wor sample $X^n$, a natural idea is to replace the usage of $\Phat_N$ with $\Phat_n = \frac{1}{n} \sum_{i=1}^n \delta_{X_i}$, to get
\begin{align}
    \Cemp_n &\equiv \Cemp_n(\alpha_N) = \lb \bar{X}_n - \hatepsilon_N, \bar{X}_n + \hatepsilon_N \rb, \quad\\ \text{where} & \quad \hatepsilon_N = \frac{1}{\beta_N} \lp \hat{\Lambda}_N^* \rp^{-1}\lp t_{N, \alpha_N} + \frac{1}{N^{1.1}} \rp, \label{eq:asympCI-def}
      \quad 
     \hat{\Lambda}_N(\lambda) \defined \int f_{N, \lambda}(x) d\Phat_n(x),  
\end{align}
recalling $f_{N, \lambda}$ from~\eqref{eq:asympCI-2}, and $t_{N, \alpha_N}$ from~\eqref{eq:asympCI-3}. 
\begin{theorem}
    \label{theorem:asympCI} 
    Let $E_N$ denote the event that $\{\Corc_N(\alpha_N) \not \subset \Cemp_n(\alpha_N)\} = \{\hat{\epsilon}_N < \epsilon_N\}$ that the purely empirical CI constructed above does not contain the oracle CI with finite-sample validity. Then, under Assumption~\ref{assump:limiting-distribution}, we have 
        $\mathbb{P}\lp E_N \text{ infinitely often} \rp = \mathbb{P}\lp \cap_{N \geq 1} \cup_{N' \geq N} E_{N'}\rp = 0$. %
    Consequently, with $\{\alpha_N \in (0,1]: N \geq 1\}$ such that $\sum_{N=1}^{\infty} \alpha_N < \infty$, we can conclude that the empirical CIs, $\{\Cemp_n: n \geq 1\}$, do not contain the population mean only finitely often almost surely: 
    \begin{align}
        \mathbb{P}\lp \cap_{N \geq 1} \cup_{N'\geq N} \{\mu_N \not \in \Cemp_n(\alpha_N)\} \rp = 0, \quad \iff \quad 
        \sum_{N=1}^{\infty} \boldsymbol{1}_{\mu_N \not \in \Cemp_n(\alpha_N)} \;\stackrel{a.s.}{<}\; \infty. \label{eq:finite-errors}
    \end{align}
\end{theorem}
 The idea behind the proof is that both $\Lambda_N$ and $\Lambdahat_n$ can be well-approximated  by quadratic functions near the origin, which then allows us to then get the required control over their inverses needed to obtain the result stated in~\Cref{theorem:asympCI}.  The details of the proof of are in~\Cref{proof:asympCI} and a detailed formal discussion of the meaning of almost-sure coverage is presented in~\Cref{appendix:formal-EAS-convergence}.
\begin{remark}
    \label{remark:asymp-CI} 
        The event $\{E_N \text{ infinitely often}\}$ having probability zero implies that there is a random finite cutoff  $N^*$, such that with probability one, the  CI defined in~\eqref{eq:asympCI-def} eventually contains the oracle interval $\Corc_N$ for all $N \geq N^*$. This implies $\liminf_{N \to \infty} \mathbb{P}(\mu_N \in \Cemp_n) \geq 1 - \limsup_{N \to \infty} \alpha_N$. If, in addition, we choose a vanishing sequence of $\alpha_N \downarrow 0$ with $\sum_N \alpha_N < \infty$, then by another application of Borel-Cantelli Lemma, we obtain~\eqref{eq:finite-errors}, which implies that only finitely many miscoverages happen almost surely. This is exactly the defining property of ``almost sure CIs''~\citep{naaman2016almost}, and also parallels the construction of hypothesis tests with only finitely many errors by~\citet{cover1973determining}. Similar ideas have also been employed in related problems of prediction~\cite{wu2019being} and classification~\citet{devroye2002almost}.

        Finite-sample valid \wor confidence intervals~(e.g., Serfling-Hoeffding or Serfling-Bernstein CI) guarantee $\mathbb{P}(\mu_N \in C_n) \geq 1-\alpha_N$ for each fixed $N$, but they are generally wider and require prior knowledge of the range~(for example, $[0,1]$) in which the observations lie. Classical CLT-based asymptotic confidence intervals, on the other hand, only guarantee $\lim_{N \to \infty} \mathbb{P}(\mu_N \in C_n) = 1-\alpha$ for a prespecified level $\alpha$, and in particular, may miscover infinitely often along the path. Thus, in practice our empirical CI defined in~\eqref{eq:asympCI-def} sits between these two extremes. It does not require a bound on the range (like CLT, and unlike Serfling-Hoeffding or Serfling-Bernstein), but its coverage guarantee is strictly stronger than the CLT-based confidence intervals. 
\end{remark}

\subsubsection{\wor CI in Smooth Banach Spaces}
\label{subsec:banach-space-CI}
To further illustrate the power of the coupling stated in~Fact~\ref{fact:coupling}, we now consider the case in which $\Sigma$ is a $(2, D)$ smooth Banach space with norm $\|\cdot \|$. Recall that a $(2, D)$-smooth Banach space is a complete normed linear space with the property that for any $u, v \in \Sigma$, we have 
\begin{align}
    \|u+v\|^2 + \|u-v\|^2 \leq 2 \|u\|^2 + 2 D^2 \|v\|^2.  \label{eq:smoothness-condition}
\end{align}
For $D=1$, this reduces to the inequality implied by the parallelogram identity~(that holds with equality for Hilbert spaces), and larger values of $D$ in~\eqref{eq:smoothness-condition} indicate further deviation from the Hilbert space case. 

Given $\{x_1, \ldots, x_N\}\subset \Sigma$, with $\|x_i\|\leq d$ for all $i \in[N]$, a \wor sample $X^n = (X_1, \ldots, X_n)$ with $n/N = \beta \in (0, 1)$,  and a parameter $\alpha \in (0, 1)$, we construct a new level-$(1-\alpha)$ CI for the population mean $\mu_N = (1/N) \sum_{i=1}^N x_i$ using the \wor sample. 

\begin{theorem}
    \label{theorem:banach-space-CI} Suppose $\calX_N = \{x_1, \ldots, x_N\}$ denotes a finite subset in a $(2, D)$-smooth Banach space $(\Sigma, \|\cdot\|)$ for a known $D\geq 1$, $\max_{1 \leq i \leq N}\|x_i\|\leq d$ for a known constant $d > 0$. For any $\alpha, \beta \in (0, 1)$ and $\lambda \geq 0$,  introduce the following terms: 
    \begin{align}
        \ell_n \equiv \ell_n(\alpha, \beta) = \log \lp \frac{2.36 \sqrt{(1-\beta)n}}{\alpha}\rp, \qtext{and} 
        g(\lambda) = \frac{1}{\beta} \log \lp 1 + D^2\lp \beta e^{\lambda \barbeta d} + \barbeta e^{\lambda\beta d} - 1 - 2\lambda \beta \barbeta d \rp\rp.  \label{eq:ell-and-g-definitions}
    \end{align}
    Then, given $X^n = (X_1, \ldots, X_n)$ drawn \wor from $\calX_N$ and an $\alpha \in (0, 1)$, the following is a level-$(1-\alpha)$ CI for the population mean $\mu_N = (1/N) \sum_{i=1}^N X_i$: 
    \begin{align}
        C^*_n(X^n, \alpha, d) =  \{ x \in \Sigma: \|\bar{X}_n - x\| \leq \epsilon^*_n\}, \qtext{with}
         \epsilon^*_n = \inf_{\lambda > 0}  \frac{1}{\lambda} \lp \frac{\ell_n}{n} + g(\lambda) \rp. \label{eq:coupling-pinelis-CI-numerical}
    \end{align}
    For small enough values of $\lambda$, we can show that $g(\lambda) \lesssim \lambda^2$, which leads to the following closed-form CI
    \begin{align}
        C_n(X^n, \alpha, d) = \{x \in \Sigma: \|\bar{X}_n- x\| \leq \epsilon_n\}, \quad \epsilon_n = dD\sqrt{\frac{3\barbeta \ell_n}{n}} \quad \text{for} \; n \geq \frac{4 (\max\{\beta, \barbeta\})^2 \ell_n}{3 D^2 \barbeta}. \label{eq:coupling-pinelis-CI} 
    \end{align}
\end{theorem}

\begin{remark}
    \label{remark:pinelis-CI-width}
    The constant $3$ in the expression for $\epsilon_n$ in~\eqref{eq:coupling-pinelis-CI} can be reduced to $2.4$ for  
    $
        n \geq \frac{5 \ell_n}{0.81 \barbeta  D^2 d^2}. 
    $
    For some typical values of the parameters, such as $\beta=0.3$, $d=1$, $D=1.5$ and $\alpha=0.05$, we see that the condition is satisfied for $N \geq 67$. More generally, for sufficiently large values of $N$~(with other parameters fixed), we can bring the constant arbitrarily close to $2$~(see~Lemma~\ref{lemma:coupling-pinelis-exp-bound} in~\Cref{proof:banach-space-CI}). 
\end{remark}
The starting point of the proof of~\Cref{theorem:banach-space-CI} is again the coupling inequality~\eqref{eq:coupling-upper-bound-1}, which allows us to bound the probability of the event  $\{\|\bar{X}_n - \mu_N\| \geq \epsilon\}$ with $1.18\sqrt{n \barbeta}$ times the probability of the event $\{ \| (1/n) \sum_{i=1}^N (J_i-\beta) x_i \| \geq \epsilon\}$. Since the latter event involves studying the deviation of sum of independent zero-mean $\Sigma$-valued random variables, we can employ classical results such as~\citet[Theorem 3]{pinelis1994optimum} for obtaining a bound. The details are in~\Cref{proof:banach-space-CI}. 

\paragraph{Comparison with~\citet{schneider2016probability}.} This problem of deriving CIs for the mean of Banach space valued objects from a \wor sample was first considered by~\citet{schneider2016probability}. In particular,~\citet{schneider2016probability} combined Serfling's martingale-based arguments with the concentration inequality of~\citet{pinelis1994optimum} to derive the following CI: 
    $\CSch = [\bar{X}_n - \epsilonSch, \bar{X}_n + \epsilonSch], \qtext{with} 
    \epsilonSch = d D\sqrt{ {\big(8 (\barbeta + 1/N) \log(2/\alpha)\big)}/{n}}$. %
Comparing the width of the CI in~\eqref{eq:coupling-pinelis-CI} with that of~\citet{schneider2016probability} stated above, we see that 
    $\lp{\epsilonSch}/{\epsilon_n}\rp^2 \;\geq\;  {(8 \log(2/\alpha))}{(3 \ell_n)}$, 
where we used the fact that $(\barbeta + 1/N)/\barbeta)  \geq 1$. Thus, we see that a sufficient condition for $\epsilon_n$ to be smaller than $\epsilonSch$ is  $\log(1.18 \sqrt{\barbeta n})/\log(2/\alpha) \leq 5/3$. For natural values such as $N=20000$ and $\alpha=0.05$, we can see in~Figure~\ref{fig:width-comparison} that this ratio is larger than $1$ for all values of $\beta \in [0,1]$~(equivalently $n \in \{1, \ldots, N\}$). More generally, if we work in the ``moderately accurate'' regime with $\alpha_N \asymp n^{-\gamma}$ for some $\gamma>0$, then, an asymptotic lower bound is
    $\lp {\epsilonSch}/{\epsilon_n} \rp^2 \gtrsim {(8 \gamma/3)}/{\lp 1/2 + \gamma \rp}$. 
Thus, a sufficient condition for asymptotically tighter CI is if $\alpha_N \asymp n^{-\gamma}$ with $\gamma > 0.3$. We numerically compare these CIs on an MNIST example in~\Cref{fig:width-comparison} in Appendix~\ref{appendix:width-comparison}.

\section{Conclusion}
\label{sec:conclusion}
    \sloppy This paper constructed confidence intervals~(CIs) for means under without-replacement~(\wor) sampling. For the case of finite alphabets, in~\Cref{theorem:lower-bound-1} we established a fundamental lower bound on the width of all \emph{moderately accurate}~(Definition~\ref{def:moderate-accuracy}) confidence intervals in terms of inverse \wor rate functions. Then, by replacing the population rate quantities with their empirical analogs, in~\Cref{subsec:new-CI-finite}, we constructed a non-asymptotically-valid CI  and proved that its width nearly matches the lower bound~(\Cref{theorem:inverse-information-projection}). We also developed a computationally efficient dual characterization in~Theorem~\ref{theorem:dual-finite}, that reduces the evaluation of the rate-function terms to a computationally efficient two-dimensional convex program, independent of the alphabet size. Both our upper and lower bounds for finite alphabets rely crucially on a certain probability approximations that are a consequence of a fundamental Bernoulli coupling identity~(Fact~\ref{fact:coupling}). 

    In~\Cref{sec:general-alphabet}, we move beyond finite alphabets and construct CIs in two practically relevant settings. In~\Cref{subsec:asymp-valid-CI}, we develop an \emph{almost sure} CI for the case of $\Sigma = [0,1]$, which does not need knowledge of the range of the alphabet, and delivers a stronger coverage guarantee than CLT-based CIs. Finally, in~\Cref{subsec:banach-space-CI} we develop a simple and non-asymptotic CI for smooth Banach spaces. 
    A unifying device used throughout the paper is the Bernoulli coupling identity~(Fact~\ref{fact:coupling}), which transfers \wor sampling problems to settings governed by independent empirical measures and their large deviation rate functions. 

    Our work opens up several promising directions for future work, such as  extending our finite alphabet guarantees of~\Cref{subsec:lower-bound-finite} and~\Cref{subsec:new-CI-finite} to general alphabets, sharpening the lower bound of~\Cref{theorem:lower-bound-1} by removing the factor $1/2$, and extending our results to weighted \wor sampling and  time-uniform settings.

\newpage 
\bibliographystyle{abbrvnat}
\bibliography{ref}

\newpage
\appendix 

\section{Additional Background}
\label{appendix:additional-background}
We begin by recording some alternative representations of the term $I(P, \beta, Q)$ used in~Fact~\ref{fact:dembo-zeitouni}. 
  \begin{remark}
        \label{remark:I-is-an-f-divergence} We can rewrite the information term~$I(P, \beta, Q)$ in a more interpretable form by introducing the distribution $R = (Q- \beta P)/\barbeta$. A simple rearrangement then implies that $Q = \beta P + \barbeta R$. Now, plugging this expression of $Q$ into the definition of $I(P, \beta, Q)$ gives us 
        \begin{align}
            I(P, \beta, Q) &= \dkl(P \parallel Q) + \frac{\barbeta}{\beta} \dkl(R \parallel Q)  = \frac{1}{\beta}\lp \beta \dkl(P \parallel \beta P + \barbeta R) + \barbeta \dkl(R \parallel \beta P + \barbeta R) \rp.  \label{eq:I-alternative-formulation}
        \end{align}
        For $\beta = 1/2$, $I(P, \beta, Q)$ simply reduces to twice the Jensen-Shannon divergence between $P$ and $R$. In fact, we can verify that $I(P, \beta, Q)$ is an $f$-divergence, with $f(x) = x \log x + (\barbeta/\beta) \frac{1-\beta x}{\barbeta} \log \lp \frac{1- \beta x}{\barbeta} \rp$. Finally, we note that $I(P, \beta, Q)$ also admits the following interesting formulation 
        \begin{align}
            I(P, \beta, Q) = \frac{1}{\beta} \lp H(Q) - \beta H(P) - \barbeta H(R) \rp = \frac{1}{\beta} \lp H(Q) - \beta H(P) - \barbeta H\lp \frac{Q - \beta P}{\barbeta} \rp \rp, 
        \end{align}
        where $H(P) = \mathbb{E}_{X \sim P}[-\log P(X) ]$ denotes the Shannon entropy of the discrete distribution $P$. 
    \end{remark}
The remark above notes that the rate function $I$ is an instance of an $f$-divergence. Recall that $f$-divergences are a generalization of relative entropy using a convex, lower semicontinuous function $f:(0, \infty) \to \mathbb{R} \cup \{\infty\}$ with $f(1)=0$, such that for any two distributions $P \ll Q$, we have 
\begin{align}
    D_f(P \parallel Q) = \mathbb{E}_{Q} \lb f\lp \frac {dP} {dQ}(X) \rp \rb. 
\end{align}
For the general case when $P \not \ll Q$, see~\citep[Definition~7.1]{polyanskiy2025information}. As with relative entropy, the family of $f$-divergences also satisfy a data processing inequality~(DPI), that we recall next.  
\begin{fact}[Theorem~7.4 of~\cite{polyanskiy2025information}] 
    \label{fact:f-div-dpi} 
    Let $K:\calF_{\calY} \times \calX \to [0,1]$ denote a stochastic kernel associated with two measurable spaces $(\calX, \calF_{\calX})$ and $(\calY, \calF_{\calY})$, and for any $P_X, Q_X$ on $\calX$, let $P_Y = P_X K$ and $Q_Y = Q_X K$.  That is, for any $E \in \calF_{\calY}$, we have $P_Y(E) = \int K(E \mid x) dP_X(x)$~(define $Q_Y(E)$ similarly). Then, for any $f$-divergence $D_f$, we have 
    \begin{align}
        D_f(P_Y \parallel Q_Y) \leq D_f(P_X \parallel Q_X). 
    \end{align}
    In other words, passing  distributions through a common stochastic kernel can only bring them closer in $D_f$. 
\end{fact}
We use this fact in the proof of Lemma~\ref{lemma:coarse-CI} in~\Cref{proof:inverse-information-projection} and in Lemma~\ref{lemma:dual-finite-1} in~\Cref{proof:dual-finite}.  

Finally, we recall a one-dimensional version of the inverse function theorem~\citep[Theorem 2-11]{spivak1965modern} that will be used in the proof of~\Cref{theorem:asympCI} in~\Cref{proof:asympCI}. 
\begin{fact}[Inverse Function Theorem]
    \label{fact:inverse-function-theorem}
    Let $U \subset \mathbb{R}$ be an open set, $f \in C^1(U)$, and consider an $a \in U$ with $f'(a) \neq 0$, and $f(a)=b$. Then, there exist open intervals $V \subset U$~(with $a \in V$) and $W \subset \mathbb{R}$~(with $b \in W$), such that $f:V \to W$ is a bijection with a $C^1$ inverse $f^{-1}:W \to V$. Moreover, for all $y \in W$, we have 
    \begin{align}
        (f^{-1})'(y) = \frac{1}{f'(f^{-1}(y))}. 
    \end{align}
    In particular, this implies that $(f^{-1})'(b) = 1/f'(a)$. 
\end{fact}
\subsection{List of Key Symbols}
\label{appendix:notation}

\renewcommand{\arraystretch}{1.5} %
\newcolumntype{S}{>{$}l<{$}}   %
\newcolumntype{D}{>{\raggedright\arraybackslash}p{0.68\linewidth}}

\begin{small}
\begin{longtable}{@{}l p{0.72\textwidth}@{}}
\hline
Symbol & Meaning \\
\hline
\endfirsthead

\hline
Symbol & Meaning (continued) \\
\hline
\endhead

\hline
\multicolumn{2}{r}{\footnotesize continued on next page} \\
\hline
\endfoot

\hline
\endlastfoot
    $N, n$ & population and sample size \\
    $\Sigma$ & alphabet: a finite set, $[0,1]$, or a smooth Banach space \\
    $\mathcal X_N$ & finite population $\{x_1,\dots,x_N\}$ with elements in $\Sigma$ \\
    $(\mu_N,\sigma_N^2)$ & population mean and variance \\
    $X^n=(X_1,\dots,X_n)$ & WoR sample of size $n=\beta N$ \\
    $\hat P_N=\frac1N\sum_{i=1}^N \delta_{x_i}$ & population distribution \\
    $\hat P_n$ & empirical distribution of $X^n$ \\
    $\mathcal T_n(\Sigma)$ & all types (i.e., empirical distributions) with denominator $n$ \\
    $S_t$ & type class of $t$: $\{x^n\in\Sigma^n:\hat P(x^n)=t\}$ \\
    $I(P,\beta,Q)$ & \wor rate-function $d_{\mathrm{KL}}(P\Vert Q)+\frac{\bar\beta}{\beta}\,d_{\mathrm{KL}}\!\big(\frac{Q-\beta P}{\bar\beta}\big\Vert Q\big)$ \\
    $J_{\pm}(P,\beta,m)$ & mean-constrained rate functions \\
    $g_{\pm}(x)$ & inverse rate-function maps \\
    $\hat g_{\pm}(x)$ & empirical analogs of $g_{\pm}$ (replace $P$ by $\hat P_n$) \\
    $f_{N, \lambda}(x)$ & $\log \lp \beta_N e^{\lambda \barbeta_N x} + \barbeta_N e^{-\lambda \beta_N x} \rp$ \\
    $ \Lambda_N(\lambda)$ & $\int f_{N, \lambda}(x) d\Phat_N(x)$ \\
    $ \hat{\Lambda}_n(\lambda)$ & $\int f_{N, \lambda}(x) d\Phat_n(x)$; empirical analog of $\Lambda_N(\lambda)$\\ 
    $ \Lambda^*_N(y)$ & $\sup_{\lambda \geq 0}\{ \lambda y - \Lambda_N(\lambda) \}$ \\ 
    $t_{N,\alpha}$ & $\frac{1}{N} \log \lp  \frac{2.36 \sqrt{(1-\beta_N)n}}{\alpha}\rp$ \\
    $\epsilon_N$ & $\frac{(\Lambda^*_N)^{-1}\lp t_{N,\alpha_N} \rp}{\beta_N}$; oracle CI width in~\eqref{eq:asympCI-3} \\
    $\hatepsilon_N$ & $\frac{1}{\beta_N} \lp \hat{\Lambda}_N^* \rp^{-1}\lp t_{N, \alpha_N} + \frac{1}{N} \rp$; empirical CI width in~\eqref{eq:asympCI-def} \\
    $\ell_n \equiv \ell_n(\alpha, \beta)$ & $\log \lp \frac{2.36 \sqrt{(1-\beta)n}}{\alpha}\rp$; \Cref{theorem:banach-space-CI}\\
    $g(\lambda)$ & $\frac{1}{\beta} \log \lp 1 + D^2\lp \beta e^{\lambda \barbeta d} + \barbeta e^{\lambda\beta d} - 1 - 2\lambda \beta \barbeta d \rp\rp$; \Cref{theorem:banach-space-CI}
\end{longtable}
\end{small}

    \section{Proof of Theorem~\ref{theorem:lower-bound-1} (Lower Bound on any CI Width)}
    \label{proof:lower-bound-1}
     
        We now present the details of the proof using the outline presented after the statement of~\Cref{theorem:lower-bound-1}. 
        \begin{description}
            \item[Step 1:] 
                 Consider the following setting: Let $\mc{X}^*_N \subset \Sigma^N$ denote a finite collection of items lying in the finite alphabet $\Sigma$ with empirical distribution $P^*$, and suppose that $X^n \equiv (X_1,\ldots, X_n)$ denotes the $n$ without replacement samples from $(\mc{X}^*_N, P^*)$. For two other pairs $(\mc{X}_N, P)$ and $(\mc{Y}_N, Q)$, consider the hypothesis testing problem: \begin{align}
                    H_0: P^*=P, \quad \text{versus} \quad H_1: P^*=Q. 
                \end{align}
                Fix a $w>0$~(to be specified below), and assume that  $\mu_Q  > \mu_P + w$, where $\mu_Q= \frac{1}{N} \sum_{y \in \mc{Y}_N} y$ and $\mu_{P} = \frac{1}{N} \sum_{x \in \mc{X}_N} x$. 
                Since we have a method $\mc{C}$ of constructing CIs for the population mean, we can use it to define a hypothesis test for the above problem. In particular, we define the test  that rejects the null if the CI contains $\mu_Q$:
                \begin{align}
                    \Psi(X^n) = \boldsymbol{1}_{\mu_Q \in C_n}, \quad \text{where} \quad C_n = \mc{C}(X^n, \alpha_N). 
                \end{align}
                Throughout the proof, we will assume that the parameter $w>0$ satisfies $\mathbb{P}_{H_0}(|C_n| \geq w) < \alpha_N$. 
                
                \emph{Type-II error of $\Psi$.} Observe that the type-II error of this test is at most $\alpha_N$; that is 
                \begin{align}
                    \mathbb{P}_{H_1}(\Psi=0) = \mathbb{P}_{H_1}\lp \mu_Q \not \in C_n \rp \leq \alpha_N, 
                \end{align}
                where the inequality follows simply from the defining property of a level-$(1-\alpha_N)$ CI for the mean. Interestingly, the test $\Psi$ is so constructed, that the above property has no dependence on the width of the CI under $Q$. 
        
                \emph{Type-I error of $\Psi$.} A similar argument implies that type-I error of this test is at most $2\alpha_N$: 
                \begin{align}
                    \mathbb{P}_{H_0}\lp \Psi = 1 \rp &= \mathbb{P}_{H_0} \lp \mu_Q \in C_n \rp   \\ &= \mathbb{P}_{H_0}\lp \lp\{\mu_Q \in C_n\}\cap \{\mu_P \in C_n\} \rp \, \bigcup \, \lp \{\mu_Q \in C_n\}\cap \{\mu_P \not \in C_n\} \rp\rp \\ 
                    &\leq \mathbb{P}_{H_0}\lp |C_n| > w \rp  + \mathbb{P}_{H_0}\lp \{\mu_P \not \in C_n\}\rp \\
                    &\leq \alpha_N + \alpha_N = 2\alpha_N. 
                \end{align}
                Here, the first inequality uses the fact that by construction $|\mu_Q - \mu_P| > w$, and thus the event $\{\mu_Q \in C_n\} \cap \{\mu_P \in C_n\}$ is contained in $\{|C_n| > w\}$, which occurs with probability at most $\alpha_N$. 

        \item[Step 2:] We now show, by employing the construction used in the proof of~\citep[Lemma 1]{zeitouni1991universal}, that there exists a hypothesis test, $\Psi'$, that only depends on the type, $\Phat(X^n)$, and not on the order of observations, with error comparable to $\Psi$. In particular, let $A = \{x^n \in \Sigma^n: \Psi(x^n) = 1\}$ denote the set of all sequences in $\Sigma^n$ on which $\Psi$ rejects the null $H_0$. For any type (or empirical distribution with $n$ observations) $t$, let $S_t = \{x^n \in \Sigma^n: \Phat(x^n) = t\}$ denote its associated \emph{type-class}. Recall that the type class associated with $t$ is simply the collection of all observations $x^n \in \Sigma^n$ whose type is $t$. Then, let us define the following set: 
        \begin{align}
            A' = \{t: t=T(x^n) \text{ for } x^n \in A, \text{ and } |S_t\cap A| \geq |S_t|/2 \}. 
        \end{align}
        In words, the set $A' \subset \calT_n(\Sigma) \subset \mc{P}(\Sigma)$, consists of those types with $n$ observations, whose type-classes have a majority of their points in $A$~(see Figure~\ref{fig:lower-bound-illustration} for an illustration). Now, we can define the test 
        \begin{align}
            \Psi'(X^n) = \begin{cases}
                1 & \text{if } \Phat(X^n) \in A' \\ 
                0 & \text{otherwise}. 
            \end{cases}
        \end{align}
        By construction, the performance of this type-based test is comparable to that of $\Psi$:
        \begin{align}
            \mathbb{P}_{H_1}(\Psi'=0) = \mathbb{P}_{H_1}(\Phat(X^n) \in (A')^c) = \sum_{t \in (A')^c} \mathbb{P}_{H_1}(\Phat(X^n)=t) \leq 2 \sum_{x^n \in A^c} \mathbb{P}_{H_1}(x^n) \leq 2 \alpha_N. 
        \end{align}
        Thus, we see that the type-II error of this new type-based test $\Psi'$  is at most two times that of $\Psi$. A similar argument also implies that 
        \begin{align}
            \mathbb{P}_{H_0}(\Psi'=1) = \mathbb{P}_{H_0}(\Phat(X^n) \in A') = \sum_{t \in A'} \mathbb{P}_{H_0}(\Phat(X^n)=t) \leq 2 \sum_{x^n \in A}\mathbb{P}_{H_0}(X^n=x^n) \leq 4\alpha_N. 
        \end{align}
        Thus, the type-based test $\Psi'$ controls both the type-I and type-II errors at level $4\alpha_N$ and $2\alpha_N$ respectively. 

        \begin{figure}[htbp!]
            \centering
            \begin{tikzpicture}
  \draw[line width=1.5pt] (0,0) rectangle (4,4);
  \foreach \x in {1,2,3} {
    \draw[line width=0.5pt] (\x,0) -- (\x,4);
  }
  \foreach \y in {1,2,3} {
    \draw[line width=0.5pt] (0,\y) -- (4,\y);
  }
  \fill[gray!10,draw=gray!50]
    (0.8,2) .. controls (1.5,3.5) and (2.5,3.5) .. (3.2,2)
           .. controls (2.5,0.5) and (1.5,0.5) .. cycle;
\end{tikzpicture}\hspace{1.5em}
            \begin{tikzpicture}
  \draw[line width=1.5pt] (0,0) rectangle (4,4);
  \foreach \x in {1,2,3} {\draw[line width=0.5pt] (\x,0) -- (\x,4);}  
  \foreach \y in {1,2,3} {\draw[line width=0.5pt] (0,\y) -- (4,\y);}  

  \fill[gray!10, draw=gray!50]
    (0.8,2) .. controls (1.5,3.5) and (2.5,3.5) .. (3.2,2)
           .. controls (2.5,0.5) and (1.5,0.5) .. cycle;

  \foreach \i/\j in {0/0,1/0,2/0,3/0,0/1,3/1,0/2,3/2,0/3,1/3,2/3,3/3} {
    \draw[line width=1pt,red] (\i,\j) rectangle (\i+1,\j+1);
  }
  \foreach \i/\j in {1/1,2/1,1/2,2/2} {
    \draw[line width=1pt,blue] (\i,\j) rectangle (\i+1,\j+1);
  }
\end{tikzpicture}\hspace{1.5em}
            \begin{tikzpicture}
  \fill[gray!10, draw=gray!50]
    (0.8,2) .. controls (1.5,3.5) and (2.5,3.5) .. (3.2,2)
           .. controls (2.5,0.5) and (1.5,0.5) .. cycle;

  \foreach \cell in {
    (0,0),(1,0),(2,0),(3,0),
    (0,1),(3,1),
    (0,2),(3,2),
    (0,3),(1,3),(2,3),(3,3)
  } {
  }

  \draw[line width=1.5pt] (0,0) rectangle (4,4);

  \foreach \x in {1,2,3} {\draw[line width=0.5pt] (\x,0) -- (\x,4);}  
  \foreach \y in {1,2,3} {\draw[line width=0.5pt] (0,\y) -- (4,\y);}  
  \foreach \cell in {(1,1),(2,1),(1,2),(2,2)} {
    \fill[pattern=north east lines,pattern color=blue!50] \cell rectangle ++(1,1);
  }
\end{tikzpicture}
            \caption{Illustration of Step 2 in the proof of~\Cref{theorem:lower-bound-1}. In all the four plots, the large square represents the observation space $\Sigma^n$, and the smaller squares represent the different type-classes. The shaded gray region represents the rejection region associated with a given binary hypothesis test, $A = \{x^n \in \Sigma^n: \Psi(x^n)=1\}$. Given any such $A$ associated with a test $\Psi$, we can partition the type classes into red~(less than half intersection with $A$), and blue~(more than half intersection with $A$). This leads to a type-based test shown in the blue shaded region in the third plot, whose type-I and type-II errors are no more than twice that of the original test $\Psi$.}
            \label{fig:lower-bound-illustration}
        \end{figure}
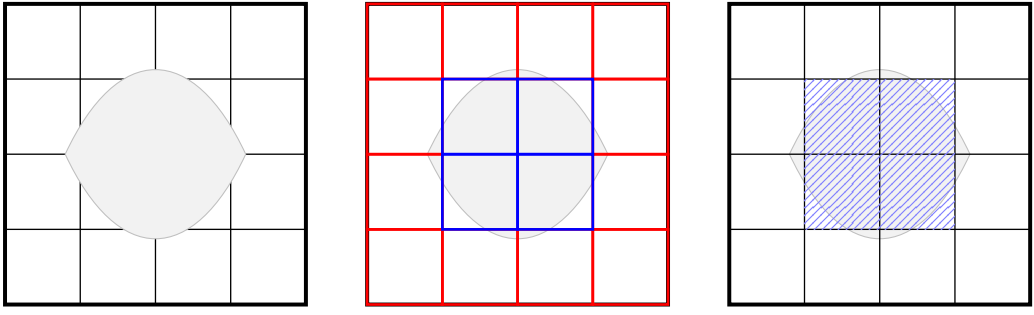

        \item[Step 3:] We now show that  types that are too close to $P$~(the null distribution) cannot lie in $A'$, the rejection region of $\Psi'$. We begin by assuming that $N$ is large enough to ensure that $\log(1/4\alpha_N) > 3 r_N$. Now, suppose that there exists a type $t^*$ in $A'$ such that $I(t^*, \beta, P) \leq g_N$, for some $0 < n g_N < \log(1/4\alpha_N) - r_N$. Then, observe the following: 
        \begin{align}
            4\alpha_N  & \geq   \sum_{t \in A'} \mathbb{P}_{H_0}(\Phat(X^n)=t) \geq \mathbb{P}_{H_0}(\Phat(X^n)=t^*) \geq \exp \lp -n I(t^*, \beta, P) - r_N \rp \\
            & \geq \exp \lp - n g_N - r_N \rp. 
        \end{align}
        On rearranging the above chain, we get that 
        \begin{align}
            \log(1/4\alpha_N) \leq n g_N + r_N \quad \text{or} \quad ng_N \geq \log(1/4\alpha_N) - r_N, 
        \end{align}
        which contradicts the assumption that $n g_N < \log(1/4\alpha_N) - r_N$. 

        Now, with $\mc{T}_n$ denoting the set of all possible types constructed using $x^n \in \Sigma^n$, we define $\widetilde{P}_n$ as the projection~(in terms of the rate function $I$) of $P$ onto $\mc{T}_n$ closest to $P$; that is, 
        \begin{align}
            \widetilde{P}_n = \argmin_{t \in \mc{T}_n} \; I(t, \beta, P) = \argmin_{t \in \mc{T}_n} \; \dkl(t \parallel P) + \frac{1-\beta}{\beta} \dkl\lp \frac{P- \beta t}{1-\beta} \parallel P \rp.  
        \end{align}
        Furthermore, there exists  a $t \in \calT_n$, such that $t_s = P(s) + \delta_s$, where $|\delta_s|\leq 1/n$ and  $\sum_{s\in \Sigma} \delta_s = 0$. Using this fact, along with the inequality $\log(1+x) \leq x$ for all $x \geq 0$, we get 
        \begin{align}
            &\dkl(t \parallel P) = \sum_{s \in \Sigma} P(s) + \delta_s) \log \lp 1 + \frac{\delta_s}{P(s)}  \rp  \leq \sum_{s \in \Sigma} \lp \delta_s + \frac{\delta_s^2}{P(s)} \rp  \leq \frac{|\Sigma| N}{n^2}. 
        \end{align}
        The last inequality above uses the fact that $P(s) \geq 1/N$ for all $s \in \mathrm{supp}(P) \subset \Sigma$. An exactly analogous argument gives us the following bound 
        \begin{align}
            \frac{\barbeta}{\beta} \dkl\lp \frac{P - \beta t}{P} \parallel P \rp \leq \frac{\beta}{\barbeta} \frac{|\Sigma|N}{n^2}. 
        \end{align}
        Combining the two previous displays, we get 
        \begin{align}
            \inf_{t \in \calT_n} I(t, \beta, P) \leq \frac{|\Sigma|}{\beta n} \lp 1 + \frac{\beta}{\barbeta} \rp  = \frac{|\Sigma|}{\beta \barbeta n}. 
        \end{align}
        which is less than $r_N/n$ for $N$ larger than $\exp(1/\beta \barbeta)$. Hence, for such values of $N$~(and hence $n=\beta N$), the type $\widetilde{P}_n$ must lie in the set $(A')^c$ for any level-$4\alpha_N$ type-based hypothesis test $\Psi'$. 
        
        \item[Step 4:] The previous step immediately implies that 
        \begin{align}
            p_1 \defined \mathbb{P}_{H_1}\lp \Phat(X^n) = \widetilde{P}_n \rp \leq \mathbb{P}_{H_1}(\Psi'=0) \leq 2\alpha_N. \label{eq:proof-lower-finite-1}
        \end{align}
        Now, by another application of Fact~\ref{fact:dembo-zeitouni}, we know that 
        \begin{align}
            \left \lvert \frac{1}{n} \log p_1  + I(\widetilde{P}_n, \beta, Q) \right\rvert \leq 2(|\Sigma| + 1) \frac{ \log (N+1)}{n} =  \frac{2r_N}{n}. \label{eq:proof-lower-finite-2}
        \end{align}
        Using this inequality in \eqref{eq:proof-lower-finite-1}, we get 
        \begin{align}
          -I(\widetilde{P}_n, \beta, Q) - r_N  \leq \frac{\log(p_1)}{n} \leq \frac{\log(2\alpha_N)}{n}, 
        \end{align}
        On rearranging, we get the required inequality  
        \begin{align}
            I(\widetilde{P}_n , \beta, Q) + r_N \geq \frac{\log(1/2\alpha_N)}{n}. \label{eq:proof-lower-finite-3}
        \end{align}
        This can be interpreted as follows: for any $w>0$ such that $\mathbb{P}_{P}(|C_n|>w) < \alpha_N$, the distribution $Q$ with mean $\mu_Q \geq \mu_P + w$ must satisfy \eqref{eq:proof-lower-finite-3}. This immediately implies that 
        \begin{align}
            J(\widetilde{P}_n, \beta, \mu_P + w) \defined \inf \{I(\widetilde{P}_n, \beta, Q): \mu_Q \geq w\} \geq \frac{\log(1/2\alpha_N)- r_N}{n}.  \label{eq:lower-bound-proof-20}
        \end{align}
        This leads to the conclusion that 
        \begin{align}
            w \geq J(\widetilde{P}_n, \beta, \cdot)^{-1}\lp \frac{\log(1/2\alpha_N)- r_N }{n}\rp, \qtext{with}\quad r_N = (|\Sigma|+1) \log(N+1). 
        \end{align}
        The right side of the above inequality is nonzero only if $\alpha_N$ is sufficiently small to make the argument of $J^{-1}$ positive. A sufficient condition for this is if $\alpha_N$ is smaller than $(1/2) (N+1)^{-|\Sigma|-1}$, which falls within our definition of \emph{moderate accuracy}~(Definition~\ref{def:moderate-accuracy}). 
        \end{description}

        \paragraph{Completing the proof.} Using the notation of the statement of~Theorem~\ref{theorem:lower-bound-1}, we can interpret~\eqref{eq:lower-bound-proof-20} as $w \geq b^*_+ - \mu_P$. This implies that $\mathbb{P}_P(|C_n| \geq   b^*_+- \mu_P) \geq \mathbb{P}_P(|C_n| \geq w) \geq \alpha_N$. Repeating the same argument with  $\mu_Q \leq \mu_P - w$ yields $\mathbb{P}_P(|C_n| \geq \mu_P-b^*_-) \geq \alpha_N$, and together  they imply 
        \begin{align}
            \mathbb{P}_P\lp |C_n| \geq \max\{b^*_+ - \mu_P, \, \mu_P-b^*_-\} \rp \leq \mathbb{P}_P\lp |C_n| \geq \frac 12 \lp b^*_+ - b^*_-\rp \rp \leq \mathbb{P}_P\lp |C_n| \geq w \rp \leq \alpha_N. 
        \end{align}
        Setting $P \leftarrow \Phat_N$ and $\mu_P \leftarrow \muhat_N$ completes the proof. 

        \begin{remark}
            \label{remark:N-large-enough-lower-bound}
            Step 3 of the proof required that $N$ is large enough~(or $\alpha_N$ is small enough) to ensure that $\log(1/4\alpha_N) > 3 r_N$. This was used to show that types sufficiently close to $P$ cannot lie in $A'$. We then assumed that $N$ is larger than $e^{1/\beta \barbeta}$ to conclude that $\widetilde{P}_n$, the projection of $P$ onto $\calT_n$, cannot lie in $A'$. Thus, the lower bound stated in~\Cref{theorem:lower-bound-1} is nontrivial for $(N, \alpha_N)$ satisfying these two conditions.
        \end{remark}

\section{Proof of Theorem~\ref{prop:new-CI-finite} (Validity of Proposed CI)}
\label{proof:new-CI-finite}

             Given any $x^n \in \Sigma^n$, we denote its type or empirical distribution by $\Phat(x^n)$. Furthermore, for any $n \geq 1$, we use $\mc{T}_n$ to denote the set of all distinct types possible with observations drawn from $\Sigma^n$.  Consider the probability of the event that the population mean $\mu_N$ does not lie in $[b_-, b_+]$:
            \begin{align}
                \mathbb{P}\lp \mu_N \not \in [b_-, b_+]\rp & = \mathbb{P}(b_+ < \mu_N) + \mathbb{P}(b_- > \mu_N).  \label{eq:CI-finite-proof-0}
            \end{align}
            We will show that both the terms on the right side of the inequality above are no larger than $\alpha_N/2$ to complete the proof. Note that $b_+$ depends on the observed data only through its type $\Phat(X^n)$, which implies the following chain: 
            \begin{align}
                \mathbb{P}\big(b_+(\Phat(X^n)) < \mu_N \big) & = \sum_{t \in \mc{T}_n: b_+(t) < \mu_N} \mathbb{P}(\Phat(X^n)=t) \label{eq:CI-finite-proof-1} \\ 
                & \leq \sum_{t \in \mc{T}_n: b_+(t) < \mu_N} \exp \lp -n I(t, \beta, \widehat{P}_N)  + r_N\rp \label{eq:CI-finite-proof-2} \\ 
                & \leq \sum_{t \in \mc{T}_n: b_+(t) < \mu_N} \exp \lp -n J_+(t, \beta, \mu_N)  + r_N\rp \label{eq:CI-finite-proof-3} \\ 
                & \leq \sum_{t \in \mc{T}_n: b_+(t) < \mu_N} \exp \lp -n J_+(t, \beta, b_+(t))  + r_N\rp \label{eq:CI-finite-proof-4}. 
            \end{align}
            
            In the above display: 
            \begin{itemize}
                \item The probability upper bound \eqref{eq:CI-finite-proof-2} follows directly from Fact~\ref{fact:dembo-zeitouni}. 
                
                \item The inequality \eqref{eq:CI-finite-proof-3} uses the fact that $I(t, \beta, \widehat{P}_N) \geq \Jplus(t, \beta, \mu_N) = \inf_{Q: t \ll Q, \, \mu_Q \geq \mu_N} I(t, \beta, Q)$. 
                
                \item The inequality \eqref{eq:CI-finite-proof-4} uses the fact that $\mu_N > b_+(t)$ and that $\Jplus(t, \beta, \cdot)$ is an increasing function of it's third argument. 
            \end{itemize}
            Continuing from~\eqref{eq:CI-finite-proof-4}, we now observe that 
            \begin{align}
                \mathbb{P}(\mu_N>b_+) & \leq \sum_{t \in \mc{T}_n: b_+(t) < \mu_N} \exp \lp - c_N  + r_N\rp \label{eq:CI-finite-proof-5} \\ 
                & = \left\lvert\{t \in \mc{T}_n: b_+(t) < \mu_N\} \right\vert \exp \lp - c_N  + r_N\rp \label{eq:CI-finite-proof-6} \\ 
                & \leq \left\lvert \mc{T}_n \right\vert \exp \lp - c_N  + r_N\rp = \exp \lp - c_N  + r_N  + \log|\mc{T}_n|\rp  \label{eq:CI-finite-proof-7} \\ 
                & \leq \exp \lp - c_N  + 2r_N\rp \label{eq:CI-finite-proof-8}\\
                & = \exp \lp -\log(2/\alpha_N) - 2r_N  + 2r_N\rp = \alpha_N/2. \label{eq:CI-finite-proof-9}
            \end{align}
            Above, we have used the following arguments: 
            \begin{itemize}
                \item The inequality \eqref{eq:CI-finite-proof-5} follows from the definition of $b_+(t) = \inf \{m \geq 0: \Jplus(t, \beta, m) \geq c_N/n\}$; which implies that the value of $\Jplus(b_+(t))$ is at least $c_N/n$. 
                \item The inequality \eqref{eq:CI-finite-proof-8} relies on the fact that the number of distinct types based on  points in $\Sigma^n$ can be upper bounded by $(n+1)^{|\Sigma|}$. This, in turn, implies that $\log(|\mc{T}_n|) \leq r_N = (|\Sigma|+1) \log (N+1)$.  
            \end{itemize}
            Following the exact same steps, we can show that the other term in~\eqref{eq:CI-finite-proof-0}, $\mathbb{P}(b_- > \mu_N)$, can also be bounded by $\alpha_N/2$.  Together,  these two statements lead to the required conclusion of~\Cref{prop:new-CI-finite}. 

\section{Proof of Theorem~\ref{theorem:inverse-information-projection}~(Near-Optimality of Proposed CI)}
\label{proof:inverse-information-projection}

We follow the steps outlined in the discussion after the statement of~\Cref{theorem:inverse-information-projection}.
In the first step of the proof, we construct a wider CI, denoted by~$\Ctilde_n$, that contains the original $C_n$. Note that this inclusion, $C_n \subset \Ctilde_n$, holds for all realizations of the data. 

\begin{lemma}
    \label{lemma:coarse-CI} The following inclusion is true, with $c_N = \log(2/\alpha_N) + 2 r_N$, where recall that $r_N= (1+|\Sigma|) \log(1+N)$: 
    \begin{align}
        C_n = [b_-, b_+] \subset \Ctilde_n \coloneqq [\muhat_n \pm \sqrt{\barbeta c_N/2n}]. 
    \end{align}
    Furthermore, by the finite sample validity of  $C_n = [b_-, b_+]$, we know that the event $E_0$ defined below occurs with probability at least $1-\alpha_N$: 
    \begin{align}
        E_0 = \{\mu_P \in [b_-, b_+]\} \subset \{|\mu_P - \muhat_n| \leq \delta_n\}, \qtext{where} \delta_n = \sqrt{2 \barbeta c_N/n}. 
    \end{align}
\end{lemma}
The proof of this result relies on an application of the data processing inequality~(DPI) for relative entropy, and Pinsker's inequality. The details are in~\Cref{proof:coarse-CI}.

For our next step we obtain a third order Taylor's approximation of the rate function $I$. 
\begin{lemma}
    \label{lemma:Taylor-expansion-of-I} Let $P$ and $Q$ be two distributions in $\calP(\Sigma)$ and let $h$ denote the difference $Q-P$~(i.e., $Q = P+h$). Fix a constant $\kappa \in (0, \barbeta)$, and assume that $|h_i| \leq \kappa p_i$ for all $i \in [k]$. Then we have  the following: 
    \begin{align}
        I(P, \beta, Q) = \frac{1}{2 \barbeta} \sum_{i=1}^k \frac{h_i^2}{p_i} + R_3(h), \qtext{with} |R_3(h)| \leq  \underbrace{\frac{(1+\kappa)(\barbeta + \kappa) + (2-\beta + 2\kappa)}{(1-\kappa)^2(\barbeta - \kappa)^2}}_{\defined C_{\beta,\kappa}}\sum_{i=1}^k \frac{|h_i|^3}{6 p_i^2}. \label{eq:taylor-expansion-I}
    \end{align}
\end{lemma}
 A simple calculation reveals $\nabla I(P, \beta, \cdot) \mid_{Q=P} = \boldsymbol{0}$ and $\nabla^2 I(P, \beta, \cdot)_{Q=P} = \mathrm{diag}(\{1/\barbeta p_i: i \in [k]\})$, which explains the expression in~\eqref{eq:taylor-expansion-I}. The details are in~\Cref{proof:Taylor-expansion-of-I}.

In our  next result, we obtain an upper bound on the rate function term $\Jplus(P, \beta, m)$ using the above approximation. 
\begin{lemma}
    \label{lemma:Jplus-upper-bound} Let $P$ be a distribution supported on $\Sigma = \{s_1, \ldots, s_k\}$ with mean $\mu$ and variance $\sigma^2$. Fix a constant $0<\kappa<\barbeta$, and suppose $\delta>0$ satisfies $\delta < \kappa \sigma^2$. Then, we have the following: 
    \begin{align}
        \Jplus(P, \beta, \mu+\delta) \leq \frac{\delta^2}{2 \barbeta \sigma^2} + \frac{C_{\beta, \kappa}}{6 \sigma^4} \delta^3. 
    \end{align}
    
\end{lemma}
The proof of this statement is in~\Cref{proof:Jplus-upper-bound}.

\paragraph{Completing the proof.} We now have all the tools to complete the proof. In particular, observe that for all $m \in \Ctilde_n = [\muhat_n \pm \sqrt{\barbeta c_N/2n}]$, the following condition is true under the event $E_0$: 
\begin{align}
    \Jplus(\Phat_n, \beta, m) - \Jplus(P, \beta, m) \geq 0 - \frac{\delta_n^2}{2 \barbeta \sigma^2} - K \delta_n^3, \qtext{where}  \delta_n = \sqrt{\frac{2\barbeta c_N}{n}}, \; K = C_{\beta, \kappa}/6 \sigma^4. 
\end{align}
To apply Lemma~\ref{lemma:Jplus-upper-bound}, we have assumed that $n$ is large enough to ensure that $\delta_n \leq \kappa \sigma^2$, or in other words: 
\begin{align}
    \frac{n}{c_N} \geq \frac{2 \barbeta}{\kappa^2 \sigma^4}. \label{eq:n-large-enough-1}
\end{align}

Hence, under $E_0$, we have the following with $\epsilon_n = \delta_n^2/2\barbeta \sigma^2 + K\delta_n^3$: 
\begin{align}
    b_+(\Phat_n) = \inf \{m \geq 0: \Jplus(\Phat_n, \beta, m) \geq c_N/n \} \leq \inf \{m \geq 0: \Jplus(P, \beta, m) \geq c_N/n + \epsilon_n \}. 
\end{align}
Now, observe that we can write $\epsilon_n \leq  (c_N/n)\times(2/\sigma^2)$ assuming that $n$ is large enough to ensure that 
\begin{align}
    \frac{n}{c_N} \geq \frac{2 \barbeta^3 C_{\beta, \kappa}^2}{9 \sigma^4} \label{eq:n-large-enough-2}. 
\end{align}
In other words, we have proved that for large enough values of $n$, we have the following: 
\begin{align}
    b_+(\Phat_n) = \inf \{m \geq 0: \Jplus(\Phat_n, \beta, m) \geq c_N/n\} \leq \inf \{m \geq 0: \Jplus(P, \beta, m) \geq d_N/n\}, 
\end{align}
where $d_N = c_N \times (1 + 2/\sigma^2) = A c_N$. This completes the proof for the upper bound on $b_+$. Applying the same steps with $\Jminus$ gives us the lower bound on $b_-$, and we omit the details to avoid repetition.

\subsection{Proof of Lemma~\ref{lemma:coarse-CI}}
\label{proof:coarse-CI}
Consider any $Q \in \calP([0,1])$ with mean $\mu_Q \geq m$ with $I(P, \beta, Q) < \infty$, and let $K:2^{\{0,1\}} \times [0,1]$ denote a channel~(i.e., stochastic kernel), such that for any $x \in [0,1]$, $K(\cdot, x)$ is the $\mathrm{Bernoulli}(x)$ distribution. In other words, suppose we input a $[0,1]$-valued random variable $X$ to this channel, then the output $Y$ has an unconditional distribution $P_Y \sim \mathrm{Bernoulli}(\mathbb{E}[X])$. Now, let $Y\sim P_Y$ and $Y' \sim Q_Y$ denote the outputs after passing $X \sim P$ and $X' \sim Q$ respectively through the channel $K$. Hence, $P_Y = \mathrm{Bernoulli}(\mu_P)$ and $Q_Y = \mathrm{Bernoulli}(\mu_Q)$. 

\sloppy Recall that in Remark~\ref{remark:I-is-an-f-divergence}  we established that $I$ is an $f$-divergence with $f(x) = x \log x + (\barbeta/\beta)\frac{1-\beta x}{\barbeta} \log \lp \frac{1-\beta x}{\barbeta} \rp$.  Hence, by the data processing inequality for $f$-divergences~\citep[Theorem 7.4]{polyanskiy2025information}, we get that 
\begin{align}
    I(P, \beta, Q) \geq I(P_Y, \beta, Q_Y) = d_{KL}(\mu_P \parallel \mu_Q) + \frac{\barbeta}{\beta} d_{KL}(\mu_R \parallel \mu_Q), \qtext{where} \mu_R = \frac{\mu_Q - \beta \mu_P}{\barbeta}. 
\end{align}
The next step is to apply Pinsker's inequality, which says that $\dkl(P \parallel Q) \geq 2 \|P - Q\|_{TV}^2$, to get 
\begin{align}
    I(P, \beta, Q) \geq  2|\mu_P - \mu_Q|^2 + \frac{2\barbeta}{\beta} \left\lvert \frac{\mu_Q - \beta \mu_P - \barbeta \mu_Q}{\barbeta} \right\rvert^2 = \lp 2 + \frac{2 \beta}{\barbeta} \rp |\mu_Q - \mu_P|^2 = \frac{2}{\barbeta} |\mu_Q - \mu_P|^2. 
\end{align}
This implies that for any $m \geq \mu_P$, we have  
\begin{align}
    \Jplus(P, \beta, m)  = \inf_{Q: \mu_Q \geq m}I(P, \beta, Q) \geq \inf_{\mu_Q \geq m \geq \mu_P} \frac{2}{\barbeta} |\mu_Q-\mu_P|^2 = \frac{2}{\barbeta} |m-\mu_P|^2. 
\end{align}
Combining this inequality with the definition of $b_+ = \inf \{m \geq 0: \Jplus(\Phat_n, \beta, m) \geq c_N/n\}$, we get 
\begin{align}
    b_+ \leq \muhat_n + \sqrt{\frac{\barbeta c_N}{2n}}, \qtext{where} \muhat_n = \frac{1}{n} \sum_{i=1}^n X_i = \mathbb{E}_{\Phat_n}[X]. 
\end{align}
An exactly analogous argument gives us $b_- \geq \muhat_n -  \sqrt{\frac{\barbeta c_N}{2n}}$, which concludes the proof.

\subsection{Proof of Lemma~\ref{lemma:Taylor-expansion-of-I}}
\label{proof:Taylor-expansion-of-I}

We obtain this result by using Taylor approximation of the functions $\phi_i(q) = p_i \log(p_i/q) + (\barbeta/\beta) r \log (r/q)$, around $q=p_i$ and with $r = (q-\beta p_i)/\barbeta$. In particular, observe that for $i\in[k]$, we have 
\begin{align}
    \phi_i(q) = \phi_i(p_i) + \phi_i'(p_i) (q-p_i) + \frac{\phi_i''(p_i)}{2}(q-p_i)^2 + \frac{\phi_i'''(z_i)}{6}(q-p_i)^3, \qtext{for some} z_i \in [p_i, q] \cup [q, p_i]. 
\end{align}
Let us now look at these terms individually. 
\begin{itemize}
    \item At $q=p_i$, we have $\phi_i(p_i) = 0$. This is because 
    \begin{align}
        \phi_i(p_i) = p_i \log \lp \frac{p_i}{p_i} \rp + \frac{\barbeta}{\beta} \lp \frac{p_i - \beta p_i}{\barbeta} \rp \log \lp \frac{p_i - \beta p_i}{\barbeta p_i} \rp = 0. 
    \end{align}
    \item Next, we look at the first derivative $\phi_i'(q)$: 
    \begin{align}
        \phi_i'(q) &= \frac{d}{dq}\lp p_i \log \lp \frac{p_i}{q} \rp  + \frac{\barbeta}{\beta} r \log \lp \frac{r}{q} \rp\rp \\
        &= - \frac{p_i}{q} + \frac{\barbeta}{\beta} \lp r'\log\lp \frac r q \rp + r\lp \frac{r'}{r} - \frac{1}{q} \rp\rp 
         = - \frac{p_i}{q} + \frac{1}{\beta} \log \lp \frac r q \rp + \frac{\barbeta}{\beta} \lp r' - \frac r q \rp \\
         & =- \frac{p_i}{q} + \frac{1}{\beta} \log \lp \frac r q \rp + \frac{\barbeta}{\beta} \lp \frac{1}{\barbeta} - \frac{q - \beta p_i}{\barbeta q} \rp  
         =- \frac{p_i}{q} + \frac{1}{\beta} \log \lp \frac r q \rp + \frac{p_i}{q} \\
         & = \frac{1}{\beta} \log \lp \frac r q \rp.  \label{eq:derivative-of-phi-i} 
    \end{align}
    Plugging in the value of $q = p_i$~(which implies $r=p_i$ as well), we get that $\phi'_i(p_i)=0$. 
    \item The second derivative of $\phi_i$ has the following expression: 
    \begin{align}
        \phi_i''(q) &= \frac{d}{dq} \phi_i'(q) = \frac{d}{dq}\lp \frac{1}{\beta} \log \lp \frac r q\rp \rp = \frac{1}{\beta} \lp \frac{r'}{r} - \frac{1}{q} \rp  = \frac{1}{\beta} \lp \frac{1}{q-\beta p_i} - \frac{1}{q} \rp  \\
        & = \frac{p_i}{(q-\beta p_i)q}. 
    \end{align}
    This implies that $\phi_i''(p_i) = 1/\barbeta p_i$. 

    \item Finally, we need to evaluate the third derivative of $\phi_i$: 
    \begin{align}
        \phi_i'''(q) = \frac{d}{dq} \phi_i''(q) = \frac{d}{dq}\lp \frac{p_i}{q(q-\beta p_i)} \rp  = \frac{q(q-\beta p_i) - p_i(2q-\beta p_i)}{q^2(q-\beta p_i)^2}. 
    \end{align}
    Now, under the assumption that $q = p_i + h_i$, with $|h_i|\leq \kappa p_i$, we get 
    \begin{align}
        \left \lvert \phi_i'''(p_i+h_i) \right\rvert \leq \frac{|q(q-\beta p_i)| + p_i|2q-\beta p_i|}{(p_i - \kappa p_i)^2(p_i - \kappa p_i-\beta p_i)^2} \leq \frac{(1+\kappa)(\barbeta + \kappa) + (2-\beta + 2\kappa)}{(1-\kappa)^2(\barbeta - \kappa)^2} \frac{1}{p_i^2}. 
    \end{align}
\end{itemize}
Plugging all these values, we get 
\begin{align}
    \sum_{i=1}^k \phi_i(p_i + h_i) = 0 + \sum_{i=1}^k\frac{h_i^2}{2\barbeta p_i} + \sum_{i=1}^k\frac{\phi'''_i(z_i)}{6} h_i^3, 
\end{align}
where each $z_i \in [p_i +h_i, p_i] \cup [p_i, p_i + h_i]$. This completes the proof.

\subsection{Proof of Lemma~\ref{lemma:Jplus-upper-bound}}
\label{proof:Jplus-upper-bound}

Recall that by definition 
\begin{align}
    \Jplus(P, \beta, \mu+\delta) &= \inf \{I(P, \beta, Q): Q \in \calP([0,1]),\; \mu_Q \geq \mu+\delta\} \\ 
    &\leq \inf \{I(P, \beta, Q): Q \in \calP(\Sigma),\; \mu_Q \geq \mu+\delta\}. 
\end{align}
Let $Q = P + h$, where $h = (h_1, \ldots, h_k)$ is such that $\sum_i h_i = 0$ and $\sum_{i=1}^k h_i s_i \geq \delta$. Let us fix a $\kappa \in (0, \barbeta)$, and define $\calH_\kappa = \{h \in [-1,1]^{\Sigma}: \sum_i h_i = 0, \; \sum_i h_i s_i \geq \delta, \; |h_i|/p_i \leq \kappa \}$. Observe that this leads to 
\begin{align}
    \Jplus(P, \beta, \mu+\delta) \leq \inf_{h \in \calH_\kappa} \frac{1}{2\barbeta} \sum_i \frac{h_i^2}{p_i} + R_3(h). \label{eq:Jplus-uppper-bound-proof-1}
\end{align}
Let us define $h^*$ as the solution of 
\begin{align}
h^* \in \argmin_{h \in \calH_\kappa}    \;  \frac{1}{2 \barbeta} \sum_{i=1}^k \frac{h_i^2}{p_i}. 
\end{align}
We will derive a closed-form expression for $h^*$ in the following steps: 
\begin{itemize}
    \item First, we will consider the optimization problem without the ``box constraints''; that is, without the constraints that $|h_i| \leq \kappa p_i$. 
    \item For this relaxed problem, we will show that we can replace the mean inequality constraint with an equality. 
    \item Then, we will obtain a closed form solution to this problem using Lagrange multipliers. 
    \item Finally, we will verify that the solution of this problem also satisfies the box constraints of the original problem. 
\end{itemize}

Let $h$ be any strictly feasible vector satisfying $\sum_i h_i = 0$ and $\sum_i h_i s_i = \delta + t$ for some $t > 0$. Let let $\theta = \delta/(\delta+t)$, and define $h' = \theta h$. Observe that $\sum_i h'_i = 0$ and $\sum_i h'_i s_i = \theta (\delta+t) = \delta$, and the objective at $h'$ strictly decreases: 
\begin{align}
    \sum_i \frac{h_i'^2}{2p_i} = \theta^2 \sum_{i} \frac{h_i^2}{2p_i} < \sum_{i} \frac{h_i^2}{2p_i}.  
\end{align}
Thus, without loss of optimality, we can work with equality mean constraint, and write the following Lagrangian with $\lambda, \gamma \in \mathbb{R}$: 
\begin{align}
    L(h, \lambda, \gamma) = \sum_i \frac{h_i^2}{2p_i} - \lambda\sum_i h_i - \gamma \lp \sum_i s_i h_i - \delta \rp. 
\end{align}
The stationarity condition then implies 
\begin{align}
    \frac{\partial L}{\partial h_i} = \frac{h_i}{p_i} - \lambda - \gamma s_i = 0 \quad \implies \quad h_i = p_i(\lambda + \gamma s_i). 
\end{align}
On enforcing the two equality constraints, we get 
\begin{align}
    &\sum_i h_i = \lambda \sum_i p_i + \gamma \sum_i p_i s_i = \lambda + \gamma \mu = 0 \\
    & \sum_i s_i h_i = \lambda \sum_i p_i s_i + \gamma \sum_i p_i s_i^2 = \lambda \mu  + \gamma(\sigma^2 + \mu^2) = \delta. 
\end{align}
Solving for $(\lambda, \gamma)$, we get 
\begin{align}
    \gamma = \frac{\delta}{\sigma^2},  \lambda = -\frac{\delta}{\sigma^2}\mu,  \qtext{which implies} h^*_i = p_i \lp - \frac{\delta}{\sigma^2} \mu + \frac{\delta}{\sigma^2}s_i \rp  = \frac{\delta}{\sigma^2} p_i (s_i - \mu). 
\end{align}

Finally, we observe that this solution also satisfies the ``box constraint''  necessary for the Taylor expansion in Lemma~\ref{lemma:Taylor-expansion-of-I} to hold; that is, 
\begin{align}
    \frac{|h^*_i|}{p_i} = \frac{\delta}{\sigma^2}|s_i - \mu| \leq \frac{\delta}{\sigma^2} \leq \kappa, 
\end{align}
where the last inequality follows from the assumption made on $\delta$. Plugging this value of $h^*$ in~\eqref{eq:Jplus-uppper-bound-proof-1}, we get 
\begin{align}
   \Jplus(P, \beta, \mu+\delta) \leq \frac{\delta^2}{2 \barbeta \sigma^2} + \frac{C_{\beta, \kappa} }{6} \sum_{i=1}^k \frac{\delta^3 p_i |s_i-\mu|^3}{\sigma^6 }  \leq \frac{\delta^2}{2 \barbeta \sigma^2} + \frac{C_{\beta, \kappa}}{6 \sigma^4} \delta^3. 
\end{align}
This completes the proof.

\begin{remark}[Meaning of $n$ large enough]
\label{remark:n-large-enough}
The precise meaning of the phrase ``$n$ large enough'' used in the statement of~\Cref{theorem:inverse-information-projection} is that we assume $n$ is larger than a value $n_0$ that is defined as follows: For a fixed $\beta \in (0, 1)$, $\kappa \in (0, 1-\beta)$, distribution $P \in \calP(\Sigma)$ with variance $\sigma^2>0$, define $n_0$ as 
\begin{align}
n_0 = \inf \lbr n \geq 1: \frac{\beta N}{\log(2/\alpha_N) + 2(|\Sigma|+1)\log(1 + n/\beta)}  \geq \max \lbr \frac{2 \barbeta}{\kappa^2 \sigma^4},\;\frac{2 \barbeta^3 C_{\beta, \kappa}^2}{9 \sigma^4} \rbr \rbr, 
\end{align}
where the constant $C_{\beta, \kappa} = \frac{1}{\beta} \lp \frac{1}{(\barbeta - \kappa)^2} + \frac{1}{(1-\kappa)^2} \rp$ was introduced in the statement of Lemma~\ref{lemma:Taylor-expansion-of-I}. The two terms inside~$\max$ above correspond to the conditions in~\eqref{eq:n-large-enough-1} and~\eqref{eq:n-large-enough-2} respectively. 

\end{remark}

    \section{Proof of Theorem~\ref{theorem:dual-finite} (Dual Rate Function)} 
    \label{proof:dual-finite}
        We begin by stating the primal form of the definition of the (mean-constrained) rate function: 
        \begin{align}
              \text{minimize}\; \;  &I(P, \beta, Q) \defined \dkl(P \parallel Q) + \frac{1-\beta}{\beta} \dkl\lp \frac{Q - \beta P}{1-\beta} \parallel Q \rp  \label{eq:primal-problem-1}\\
              \text{subject to}\; \; &Q \in \mc{P}([0,1]) \\ 
                & m-\mathbb{E}_Q[X] \leq 0 \\
               & \beta P(s) - Q(s)  \leq 0,\; \text{for } s \in \Sigma.   
        \end{align}
        We will show that the optimal value of this primal minimization problem is the same as the optimum value of the dual problem described in~\Cref{theorem:dual-finite}.

        \paragraph{Outline of the proof.} Recall that we have $\Sigma = \{s_1, \ldots, s_k\} \subset [0,1]$, with $k = |\Sigma|$ elements. For each $s_i \in \Sigma$, we will use $q_i$ and $p_i$ to represent $Q(\{s_i\})$ and $P(\{s_i\})$ respectively. Furthermore, for any $s \in [0,1]$, we will use $q_s$ to denote the  value $Q(\{s\})$. For any $\beta \in (0, 1)$, we will use $\barbeta$ to denote $1-\beta$. We now describe the steps involved in our proof. 
        \begin{itemize}
            \item \emph{Reduction to a finite-dimensional problem:}  The problem as stated in~\eqref{eq:primal-problem-1} has an infinite-dimensional domain: The space of all probability distributions on the bounded domain with two additional restrictions. As the first step, we will observe in Lemma~\ref{lemma:dual-finite-1} that an optimal solution to~\eqref{eq:primal-problem-1} must lie in a finite-dimensional domain, that we denote henceforth by~$\primaldomain$. 
            
            \item \emph{Verification of strong duality:} We begin by showing in Lemma~\ref{lemma:dual-finite-2} that there exists a point $Q$ in~$\primaldomain$ that is strictly feasible, and thus Slater's condition for strong duality is satisfied. This implies that the duality gap is zero and that the primal and dual problems have the same optimal values.  

            \item \emph{Characterization of the primal optimal:} Next, we will use first-order conditions to characterize the primal optimal solution. In particular, in this step, we will show that the primal optimal solution places a mass $q^*_i = \frac{\beta p_i}{1 - \barbeta e_i}$ with $e_i = \exp(\beta(a + b s_i+ c_i))$ on the point $s_i \in \Sigma$. 

            \item \emph{Simplification of dual objective:} Finally, by using the KKT conditions, we will simplify the dual objective problem and reduce it to a two-dimensional optimization problem~(as compared to the possibly infinite-dimensional primal problem).
        \end{itemize}

        \paragraph{Reduction to a finite-dimensional problem.} The first step of the proof is that we can reduce~\eqref{eq:primal-problem-1} to a finite-dimensional problem. 
        \begin{lemma}
            \label{lemma:dual-finite-1} For any feasible point, $\Qtilde$, for the optimization problem stated in~\eqref{eq:primal-problem-1}, there exists another point $Q \in \mathcal{P}(\Sigma_1) \subset \mathcal{P}([0,1])$ with $\Sigma_1 \defined \Sigma \cup \{1\}$, such that $I(P, \beta, Q) \leq I(P, \beta, \Qtilde) $, and $\mu_Q \geq \mu_{\Qtilde}$. 

            As a consequence, $\Jplus(P, \beta, m)$ can be written as the optimum value associated with the following minimization problem: 
            \begin{align}
                  \text{minimize}\; \;  &I(P, \beta, Q) \quad    
                  \text{subject to}\; \; Q \in \primaldomain, \label{eq:primal-problem-finite-dimensional}
            \end{align}
            where the primal feasible set $\primaldomain$ is defined as  
            \begin{align}
             \primaldomain \defined \{Q \in \calP(\Sigma_1): \mathbb{E}_Q[X] \geq m, \; \text{and} \; Q(s) \geq \beta P(s),\, \forall s \in \Sigma\}, \quad \text{with} \quad \Sigma_1 \defined \Sigma \cup \{1\}.  \label{eq:primal-domain-finite-dim} 
            \end{align}
        \end{lemma}
        \begin{proof} Any feasible $\Qtilde \in \mc{P}([0,1])$ for which $I(P, \beta, \Qtilde) < \infty$ must place non-zero mass at each $s \in \Sigma$.  Now, let us define $Q \in \mathcal{P}(\Sigma_1)$ as follows: 
        \begin{align}
            Q(s) = \begin{cases}
                \Qtilde(s), & \text{ if } s \in \Sigma \\
                \Qtilde([0,1]\setminus \Sigma), & \text{ if } s = 1. 
            \end{cases}
        \end{align}
        Since $\Qtilde$ is feasible, we have the following: 
        \begin{align}
            m \leq \mathbb{E}_{\Qtilde}[X] &= \sum_{s \in \Sigma} s\times \Qtilde(s) + \int_{[0,1]\setminus \Sigma} s d\Qtilde(s) \\
            &= \sum_{s \in \Sigma} s\times Q(s) + \int_{[0,1]\setminus \Sigma} s d\Qtilde(s) && (\text{since } Q(s) = \Qtilde(s),\;\forall s \in \Sigma) \\
            &\leq \sum_{s \in \Sigma} s\times Q(s) + \int_{[0,1]\setminus \Sigma} 1 d\Qtilde(s) && (\text{since } s \leq 1) \\
            & = \sum_{s \in \Sigma_1} s \times Q(s) = \mathbb{E}_Q[X]
        \end{align}
        Thus $Q$ is also feasible for~\eqref{eq:primal-problem-1}. 
        Furthermore, we can interpret $Q$ as a distribution obtained by passing $\Qtilde$ through a stochastic kernel~(or a channel in information-theoretic terminology) $\calK:\mathcal{B}_{[0,1]} \times [0,1] \to [0,1]$, defined as 
        \begin{align}
            \calK(\cdot\mid s) = \begin{cases}
                \delta_s, & \text{ if } s \in \Sigma, \\
                \delta_1, & \text{ if } s \in [0,1]\setminus \Sigma, 
            \end{cases}
        \end{align}
        where $\delta_s$ denotes the unit point-mass at $s \in [0,1]$. Note that this channel has no effect on a $P$ supported on $\Sigma$. Thus, an application of data processing inequality~(DPI) for $f$-divergences~\citep[Theorem~7.4]{polyanskiy2025information} implies that 
        \begin{align}
            I(P, \beta, \Qtilde) \geq I( \calK P, \beta, \calK \Qtilde) = I(P, \beta, Q). 
        \end{align}
        Thus, for every feasible $\Qtilde \in \calP([0,1])$, there exists a feasible $Q \in \primaldomain$ with equal or lower objective function value.   This completes the proof.
        \end{proof}

        \paragraph{Verification of strong duality.} Due to the convexity of relative entropy, it follows that $I(P, \beta, Q)$ is a convex function over the (convex) domain $\primaldomain$. We now show that for this convex optimization problem, sufficient conditions for strong duality to hold are satisfied.
        \begin{lemma}
            \label{lemma:dual-finite-2} 
             Introduce the dual variables $\rho \in \mathbb{R}$, $\lambda \in [0, \infty)$ and $\boldgamma = (\gamma_1, \ldots, \gamma_k) \in [0, \infty)^k$ lying in the dual feasible set   $\mathbb{R} \times [0, \infty)^{k+1}$. Recall that $k$ denotes the cardinality of the finite alphabet $\Sigma \subset (0,1)$.  Then, we have the following:   
            \begin{align}
            &\Jplus(P,\beta, m) = \sup_{(\rho, \lambda, \boldgamma) \in \mathbb{R} \times [0, \infty)^{k+1}}\; \inf_{ Q \in [0, \infty)^{k+1}} L(Q, \rho, \lambda, \boldgamma), \quad \text{where}  \label{eq:dual-definition-1}\\
            \quad 
            & L(Q, \rho, \lambda, \boldgamma) \defined I(P, \beta, Q) + \rho \lp 1 - \sum_{s \in \Sigma_1} Q(s) \rp + \lambda \lp m - \sum_{s \in \Sigma_1} s \times Q(s) \rp + \sum_{i =1}^k \gamma_i(\beta P(s_i) - Q(s_i)).  \label{eq:lagrangian}
        \end{align}
        \end{lemma}

        \begin{proof}
            Following~\citet[\S~5.2.3]{boyd2004convex}, a sufficient condition~(Slater's constraint qualification) for strong duality is if there exists a strictly feasible point in the domain of the primal problem. 
            
            To verify this, let us define $Q_\epsilon = \epsilon P + (1-\epsilon) \delta_1$, where $\delta_1$ denotes the point-mass at $x=1$~(i.e., the probability distribution that places its entire mass at $x=1$). By construction $Q_\epsilon$ belongs to $\mc{P}(\Sigma_1)$ for all $\epsilon \in (0, 1)$.  Let $\mu_P$ denote $\mathbb{E}_{P}[X]$, and observe that for $0<\epsilon < (1-m)/(1-\mu)$, the distribution $Q_\epsilon$ satisfies $m-\mathbb{E}_{Q_\epsilon}[X]<0$, where $\Sigma_1 = \Sigma \cup \{1\}$. Furthermore, for any $s \in \Sigma \setminus \{1\}$, we have $\beta P(s) - Q_\epsilon(s) = (\beta - \epsilon) P(s) < 0$ for $\epsilon > \beta$. 
            
            Thus, with $\beta < \epsilon <  (1-m)/(1-\mu_P)$, the point $Q_\epsilon$ is strictly primal feasible, which in turn implies that strong duality holds. As a result, the rate function term $\Jplus(P, \beta, m)$ as defined in~\eqref{eq:primal-problem-1} is also equal to the dual formulation, defined in~\eqref{eq:dual-definition-1}.   
        \end{proof}

        \paragraph{Primal and dual optimal solutions.} The objective function in the dual optimization problem is 
        \begin{align}
            M(\rho, \lambda, \boldgamma) \defined \inf_{Q \in [0, \infty)^{k+1}} L(Q, \rho, \lambda, \boldgamma), \quad \text{and} \quad \Jplus(P, \beta, m) = \sup_{(\rho, \lambda, \boldgamma) \in \mathbb{R} \times [0, \infty)^{k+1}} M(\rho, \lambda, \boldgamma).  
        \end{align}

        Our next result presents a nearly complete characterization of the primal optimal variable. 

        \begin{lemma}
            \label{lemma:dual-simplification-1} 
            Suppose the finite alphabet $\Sigma$ does not contain the end-point $1$, and let $(\optdualparams) \in  \mathbb{R} \times [0, \infty)^{k+1}$ and $Q^* \in \primaldomain$ denote the optimal dual and primal variables. Then, for every $s \in \Sigma$, the following is true: 
            \begin{align}
                \rho^* \leq \frac{1}{\beta}\log \lp \frac 1 \barbeta \rp - \lambda^*, \quad \text{and}\quad 
                Q^*(s) = \frac{\beta P(s)}{1 - \exp \lp -\beta \lp \lambda^* (1-s) - \gamma^*_s\rp \rp}.
            \end{align}
     
        \end{lemma}
        The proof of this result is given in~\Cref{proof:dual-simplification-1}.

        \paragraph{Simplifying the dual objective.} For the final step, we use the complementary slackness conditions to obtain the final simplified form of the dual optimization problem. 
        \begin{lemma}
            \label{lemma:dual-simplification-2}
            The dual optimal $(\optdualparams)$ must satisfy $\gamma_s^*=0$ for all $s \in \Sigma$. This fact implies that the rate function $\Jplus(P, \beta, m)$ is equal to the optimum value of the following problem: 
            \begin{align}
                \Jplus(P, \beta, m) = \sup_{(\lambda, \rho) \in \dualdomain} \lbr \mathbb{E}_P\lb \log \lp \frac {1 - \barbeta e^{\beta (\lambda X + \rho)}}{\beta} \rp \rb + \lambda (m-\beta \mu_P) +\rho \barbeta \rbr, 
            \end{align}
            where $\dualdomain = \{ (\lambda, \rho) \in [0, \infty) \times \mathbb{R}: \rho + \lambda \leq (1/\beta) \log(1\barbeta)\}$ and  $\mu_P = \sum_{s \in \Sigma} s\times P(s)$. 
        \end{lemma}
        The proof of this result is in~\Cref{proof:dual-simplification-2}

        \subsection{Proof of Lemma~\ref{lemma:dual-simplification-1}} 
        \label{proof:dual-simplification-1}
            Let $\optdualparams$ and $Q^*$ denote a pair of dual and primal optimal variables. Then, we must have the following: 
            \begin{align}
                \left. \frac{\partial}{\partial Q(s)} L(Q, \optdualparams)\right \rvert_{Q=Q^*} = 0, \quad \text{for } s \in \Sigma \subset (0, 1).  
            \end{align}
            At the end point $s=1$, we have the inequality constraint 
            \begin{align}
                \frac{\partial}{\partial dQ(1)} L(Q, \optdualparams) = \frac{1}{\beta} \log \lp \frac 1 \barbeta \rp - \rho^* - \lambda^* \geq 0 \label{eq:dLdQ-1}
            \end{align}
            For other values of $s \in \Sigma \setminus \{1\}$, the condition reduces to~(dropping the arguments of $L$)
            \begin{align}
                \frac{\partial L}{\partial Q(s)}  &= \frac{\partial}{\partial Q(s)}\lp P(s) \log \lp \frac{P(s)}{Q(s)}\rp \rp +
                \frac{\partial}{\partial Q(s)}\lp \frac{\barbeta}{\beta} \lp \frac{Q(s) - \beta P(s)}{\barbeta} \rp \log \lp \frac{Q(s)-\beta P(s)}{\barbeta Q(s)} \rp\rp \\
                &  \quad - \frac{\partial}{\partial Q(s)}\lp Q(s) \lp  \rho^* + \lambda^* + \gamma_s^*\rp \rp \\
                & =: T_1 + T_2 - \lp \rho^* + \lambda^* + \gamma^*_s \rp  \label{eq:dLdQ-2}
            \end{align}
            We now derive the expressions of $T_1$ and $T_2$. 
            \begin{align}
                T_1 & = \frac{\partial}{\partial Q(s)}\lp P(s) \log \lp \frac{P(s)}{Q(s)}\rp \rp = - \frac{P(s)}{Q(s)}. \label{eq:dLdQ-T1} 
            \end{align}
            Similarly, to evaluate $T_2$, we introduce the notation $R(s) = (Q(s)-\beta P(s))/\barbeta$, and proceed as follows:
            \begin{align}
                \frac{\beta}{\barbeta} T_2 &= \frac{\partial}{\partial Q(s)}\lp  R(s) \log \lp \frac{R(s)}{Q(s)} \rp\rp  \\
                & = \frac{\partial R(s)}{\partial Q(s)} \log \lp \frac{R(s)}{Q(s)} \rp + R(s) \frac{\partial}{\partial Q(s)} \log \lp \frac{R(s)}{Q(s)} \rp \\
                & = \frac{1}{\barbeta} \log \lp \frac{R(s)}{Q(s)} \rp + R(s) \frac{\partial}{\partial Q(s)} \log \lp \frac{R(s)}{Q(s)} \rp && (\text{since } \partial R(s)/\partial Q(s) = 1/\barbeta) \\
                & = \frac{1}{\barbeta} \log \lp \frac{R(s)}{Q(s)} \rp + R(s) \frac{\partial}{\partial Q(s)}\lp  \log R(s) - \log Q(s) \rp 
            \end{align}
            Now, observing that $\partial \log R(s)/\partial Q(s) = (1/R(s)) \times \partial R(s)/\partial Q(s) = (1/R(s))\times (1/\barbeta)$, we get 
            \begin{align}
                \frac{\beta}{\barbeta} T_2 &=\frac{1}{\barbeta} \log \lp \frac{R(s)}{Q(s)} \rp + R(s) \lp \frac{1}{\barbeta R(s)} - \frac{1}{Q(s)} \rp \\ 
                &=\frac{1}{\barbeta} \log \lp \frac{R(s)}{Q(s)} \rp + \frac{Q(s)-\beta P(s)}{\barbeta} \lp \frac{\beta P(s)}{Q(s) \, \lp Q(s)-\beta P(s) \rp} \rp \\
                &=\frac{1}{\barbeta} \log \lp \frac{R(s)}{Q(s)} \rp + \frac{\beta}{\barbeta} \frac{P(s)}{Q(s)} \\
                &=\frac{1}{\barbeta} \log \lp \frac{R(s)}{Q(s)} \rp - \frac{\beta}{\barbeta} T_1,  \label{eq:dLdQ-T2}
            \end{align}
            where the last equality uses the fact that $T_1 = P(s)/Q(s)$ as derived in~\eqref{eq:dLdQ-T1}. Plugging the above equality back into~\eqref{eq:dLdQ-2}, we get for any $s \in \Sigma \setminus \{1\}$: 
            \begin{align}
                \frac{\partial L}{\partial Q(s)}  = \frac{1}{\beta} \log \lp \frac{Q(s) - \beta P(s)}{\barbeta Q(s)} \rp - \lambda^* s - \rho^* - \gamma_s^*. 
            \end{align}
  
            Setting this equal to zero gives us the required $Q^*(s)$
            \begin{align}
                &\frac{1}{\beta} \log \lp \frac{Q^*(s) - \beta P(s)}{\barbeta Q^*(s)} \rp - \lambda^* s - \rho^* - \gamma_s^* = 0 \\  
                \implies & \frac 1  - \beta \frac{P(s)}{Q^*(s)} = {\barbeta}\exp \lp \beta \lp \lambda^* s + \rho^* + \gamma^*_s\rp \rp \\
                \implies &  \beta \frac{P(s)}{Q^*(s)} = 1 - {\barbeta}\exp \lp \beta \lp \lambda^* s + \rho^* + \gamma^*_s\rp \rp \\
                 \implies &Q^*(s) = \frac{\beta P(s)}{1 - {\barbeta}\exp \lp \beta \lp \lambda^* s + \rho^* + \gamma^*_s\rp \rp}. \label{eq:qstar-def}
            \end{align}
            This completes the proof of Lemma~\ref{lemma:dual-simplification-1}.

        \subsection{Proof of Lemma~\ref{lemma:dual-simplification-2}}
        \label{proof:dual-simplification-2}
            This result follows by an application of the following complementary slackness conditions: 
            \begin{align}
                &\lambda^*\lp m - \sum_{s \in \Sigma_1} s \times Q^*(s) \rp = 0 \label{eq:comp-slack-2}\\
                & \gamma^*_s \lp Q^*(s) - \beta P(s) \rp = 0, \quad \text{for all} \quad s \in \Sigma. \label{eq:comp-slack-3}
            \end{align}

            \begin{itemize}
                \item By~construction, the optimum $Q^*(s)$ must be nonnegative, which means that $1-\barbeta e^{\beta (\lambda^* s + \rho^* + \gamma^*_s)} \in (0, 1)$, which leads to the strict inequality $Q^*(s) > \beta P(s)$. Hence, to ensure the complementary slackness condition~\eqref{eq:comp-slack-3}, we must have $\gamma^*_s = 0$ for all $s\in \Sigma$. 

                \item Next, the condition that $Q^*(s)$ must be in $[0, \infty)$  means that $\lambda^*$ cannot be zero. This combined with~\eqref{eq:comp-slack-2} means that we must have $\sum_{s \in \Sigma_1} s \times Q^*(s) = m$. 
            \end{itemize}
            These points allow us to reduce the dual to a two-variable optimization problem, with   the primal optimal variables completely characterized by $\lambda^*>0$, and $\rho^* \in \mathbb{R}$ such that $\rho^* \in \mathbb{R}$ such that $\rho^* \leq (1/\beta) \log(1/\barbeta) - \lambda^*$: 
            \begin{align}
                Q^*(s) \equiv Q_{\lambda^*, \rho^*}(s) = \frac{\beta P(s)}{1 - \barbeta e^{\beta( \lambda^* s + \rho^*)}}, \; \text{for } s \in \Sigma \subset (0,1), \qtext{and}   Q^*_{\lambda^*, \rho^*}(1) = 1 - Q^*_{\lambda^*, \rho*}(\Sigma).  \label{eq:Q-star}
            \end{align}
            In other words, we can find $\Jplus(P, \beta, m)$ by optimizing $I(P, \beta, Q_{\lambda, \rho})$, over all $Q_{\lambda, \rho}  \in \mathcal{P}(\Sigma_1)$ of the form stated in~\eqref{eq:Q-star}, over all feasible $(\lambda, \rho)$: 
            \begin{align}
                &\Jplus(P, \beta, m) = \sup_{(\lambda, \rho) \in \dualdomain} \left\{ \dkl(P \parallel Q_{\lambda, \rho}) + \frac{\barbeta}{\beta} \dkl\lp R_{\lambda, \rho} \parallel Q_{\lambda, \rho} \rp \right\},    \label{eq:dual-simplification-10} \\
                \text{where} \quad &
                R_{\lambda, \rho} = (Q_{\lambda, \rho} - \beta P)/\barbeta,  \qtext{and} \dualdomain = \{(\lambda, \rho) \in [0, \infty) \times \mathbb{R}: \lambda + \rho \leq 1/\beta \log(1/\barbeta)\}
            \end{align}
            The first term of the dual objective is equal to 
            \begin{align}
                \dkl(P \parallel Q_{\lambda, \rho}) &= \sum_{s \in \Sigma} P(s) \log \lp \frac{P(s)}{Q_{\lambda, \rho}(s)} \rp = \sum_{s \in \Sigma} P(s) \log \lp \frac{P(s) \lp 1 - \barbeta e^{ \beta (\lambda s + \rho)} \rp}{\beta P(s)} \rp \\
                & = \sum_{s \in \Sigma} P(s) \log \lp \frac{1 - \barbeta e^{ \beta (\lambda s + \rho)}}{\beta}\rp.  \label{eq:dual-simplification-9}
            \end{align}
            The second term of the dual objective  is equal to 
            \begin{align}
                \frac{\barbeta}{\beta} \dkl(R_{\lambda, \rho} \parallel Q_{\lambda, \rho}) & = \frac{\barbeta}{\beta}\sum_{s \in \Sigma_1} R_\lambda(s) \log \lp \frac{R_\lambda(s)}{Q_{\lambda, \rho}(s)} \rp \\ 
                & =  \sum_{s \in \Sigma} \lp Q_{\lambda, \rho}(s)-\beta P(s) \rp\frac{1}{\beta} \log \lp \frac{Q_{\lambda, \rho}(s) - \beta P(s)}{\barbeta Q_{\lambda, \rho}(s)} \rp + \frac{Q_{\lambda, \rho}(1)}{\beta} \log \lp \frac 1 {\barbeta} \rp. \label{eq:dual-simplification-1}
            \end{align}
            Now, observe that  for $s \in \Sigma$, we have the following, with $e_s = \barbeta e^{\beta(\lambda s + \rho)}$: 
            \begin{align}
                \frac{Q_{\lambda, \rho}(s)-\beta P(s)}{\barbeta Q_{\lambda, \rho} (s)} & = \frac{ \beta P(s)/(1-e_s) - \beta P(s)}{\barbeta \beta P(s)/(1-e_s)} = \frac{e_s}{\barbeta} = \exp( \beta (\lambda s  + \rho)),
            \end{align}
            which implies that 
            \begin{align}
                \frac{1}{\beta} \log \lp \frac{Q_\lambda(s)-\beta P(s)}{\barbeta Q_\lambda (s)} \rp = \lambda s + \rho.  \label{eq:dual-simplification-2}
            \end{align}
            Hence, combining this with~\eqref{eq:dual-simplification-1}, we get the following: 
            \begin{align}
                \frac{\barbeta}{\beta} \dkl\lp R_{\lambda, \rho} \parallel Q_{\lambda, \rho} \rp& = \sum_{s \in \Sigma} \lp Q_{\lambda, \rho}(s) - \beta P(s) \rp (\rho + \lambda s) + \frac{Q_{\lambda, \rho}(1)}{\beta}\log \frac 1 {\barbeta}.   \label{eq:dual-simplification-3}
            \end{align}
            Next, we use the fact that $\sum_{s \in \Sigma} (Q_{\lambda, \rho}(s) - \beta P(s)) = \barbeta - Q_{\lambda, \rho}(1)$, and that $\sum_{s \in \Sigma} s \lp Q_{\lambda, \rho}(s) - \beta P(s) \rp = m - \beta \mu_P - Q_{\lambda, \rho}(1)$, and plug it into~\eqref{eq:dual-simplification-3} to get 
            \begin{align}
                \frac{\barbeta}{\beta} \dkl (R_{\lambda, \rho} \parallel Q_{\lambda, \rho})
                & = \rho \lp \barbeta - Q_{\lambda, \rho}(1) \rp + \lambda \lp m - \beta \mu_P - Q_{\lambda, \rho}(1) \rp + \frac{Q_{\lambda, \rho}(1)}{\beta} \log \frac 1 {\barbeta}
                \label{eq:dual-simplification-4}\\
                & = \rho \barbeta + \lambda(m-\beta \mu_P) + Q_{\lambda, \rho}(1) \underbrace{\lp  \frac 1 {\beta} \log \frac 1 {\barbeta}  - \rho - \lambda  \rp}_{\coloneqq c(\lambda, \rho)}. 
            \end{align}
            Now, we argue that at any $(\lambda, \rho)$ such that $c(\lambda, \rho) >0$ the primal optimal $Q^*_{\lambda, \rho}$ will set $Q^*_{\lambda, \rho}(1)=0$ to drive this term to $0$.  This leads to the final expression 
            \begin{align}
                \Jplus(P, \beta, m) = \sup_{(\lambda, \rho) \in \dualdomain} \lbr \mathbb{E}_P\lb \log \lp \frac {1 - \barbeta e^{\beta (\lambda X + \rho)}}{\beta} \rp \rb + \lambda (m-\beta \mu_P) +\rho \barbeta \rbr, 
            \end{align}
            where $\dualdomain = \{ (\lambda, \rho) \in [0, \infty) \times \mathbb{R}: \rho + \lambda \leq (1/\beta) \log(1\barbeta)\}$.  This completes the proof.

   \begin{remark}
            \label{remark:concavity-of-dual-objective}
            As a sanity check, observe that the dual problem~\eqref{eq:dual-problem} is a convex program; that is, it is a maximization problem with a concave objective function and a convex domain of optimization. The convexity of the domain $\dualdomain$, follows from the fact that it is the intersection of two convex sets:  $[0, \infty) \times \mathbb{R}$ and $ \{ \rho + \lambda \leq (1/\beta) \log (1/\barbeta)\}$. To establish the concavity of the objective function, let us rewrite it as $\lp\sum_{s \in \Sigma} f_s(\lambda, \rho) \rp  + \lambda (m - \beta \mu_P) + \rho \barbeta$, with $f_s(\lambda, \rho) \defined P(s) \log \lp \lp1 - \barbeta e^{\beta(\rho + \lambda s)} \rp/ \beta \rp$. As we show next, each $f_s$ is jointly concave in its arguments, which means that the overall objective function in~\eqref{eq:dual-problem} is also concave. In particular, for any $a \in (0, 1)$ and two points $(\lambda_1, \rho_1), (\lambda_2, \rho_2) \in [0, \infty) \times \mathbb{R}$, let $\lambda_a = a \lambda_1 + \bar{a} \lambda_2$ and $\rho_a = a \rho_1 + \bar{a} \rho_2$ denote their convex combinations. Then, we have the following chain: 
            \begin{align}
                & \barbeta e^{\beta (\lambda_a s + \rho_a)} \leq a  \barbeta e^{\beta(\lambda_1 s + \rho_1)} + \bar{a} \barbeta e^{\beta(\lambda_2 s + \rho_2)} \\
                \implies & 1-  \barbeta e^{\beta (\lambda_a s + \rho_a)} \geq a\lp 1-  \barbeta e^{\beta(\lambda_1 s + \rho_1)} \rp+ \bar{a}\lp 1 - \barbeta e^{\beta(\lambda_2 s + \rho_2)} \rp\\
                \implies & \log (1- \barbeta e^{\beta (\lambda_a s + \rho_a)}) \geq a \log \lp 1- \barbeta e^{\beta(\lambda_1 s + \rho_1)} \rp+ \bar{a} \log \lp 1 - \barbeta e^{\beta(\lambda_2 s + \rho_2)} \rp.
            \end{align}
            The first inequality follows from the convexity of $\exp(\cdot)$, and the last inequality uses the concavity of $\log(\cdot)$. 
            Multiplying both sides of the last inequality by $P(s)$ and subtracting $P(s) \log(\beta)$ gives us the required $f_s(\lambda_a, \rho_a) \geq a f_s(\lambda_1, \rho_1) + \bar{a} f_s(\lambda_2, \rho_2)$. 
        \end{remark}
\begin{figure}[t!]
        \centering
        \includegraphics[width=0.3\linewidth]{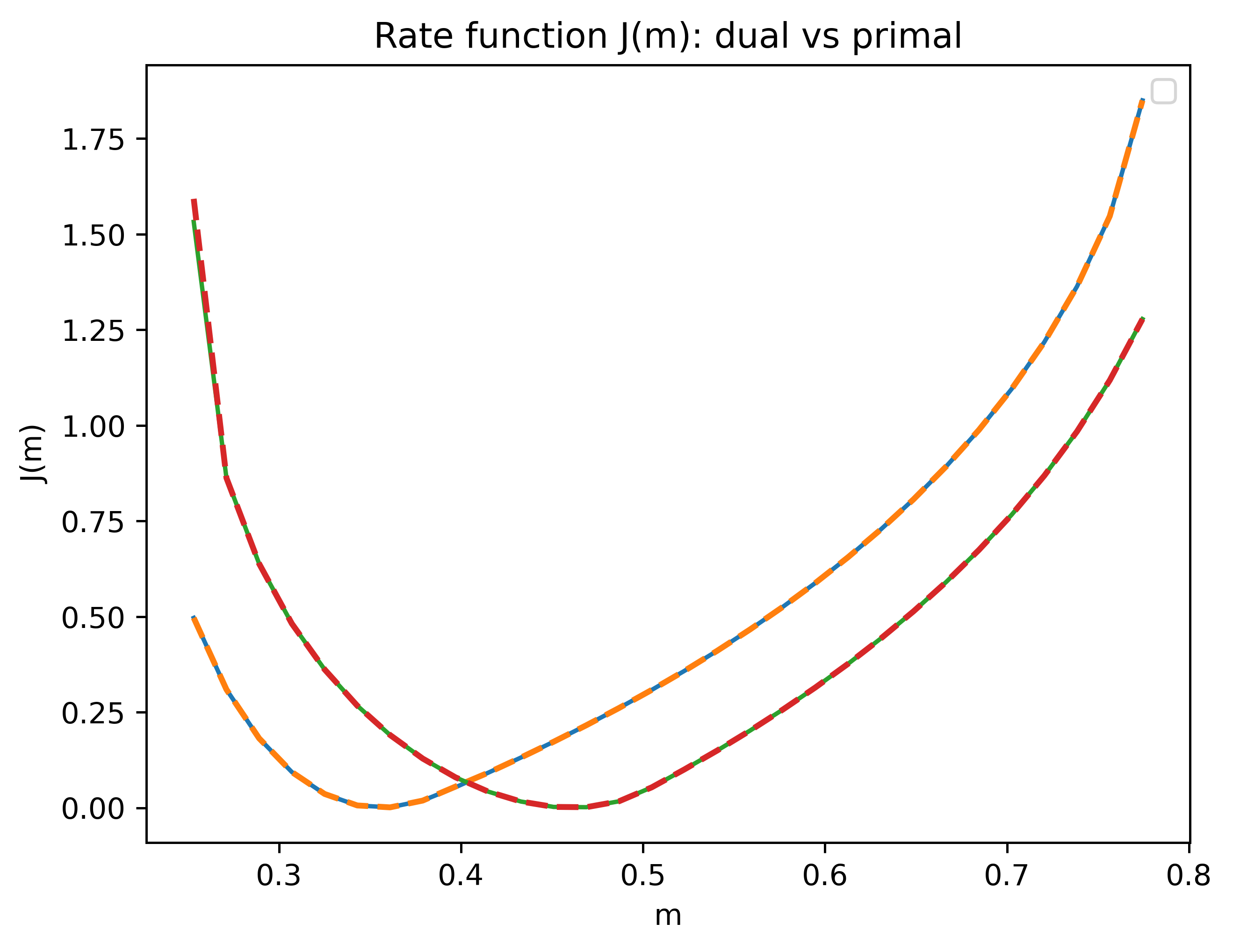}
        \includegraphics[width=0.3\linewidth]{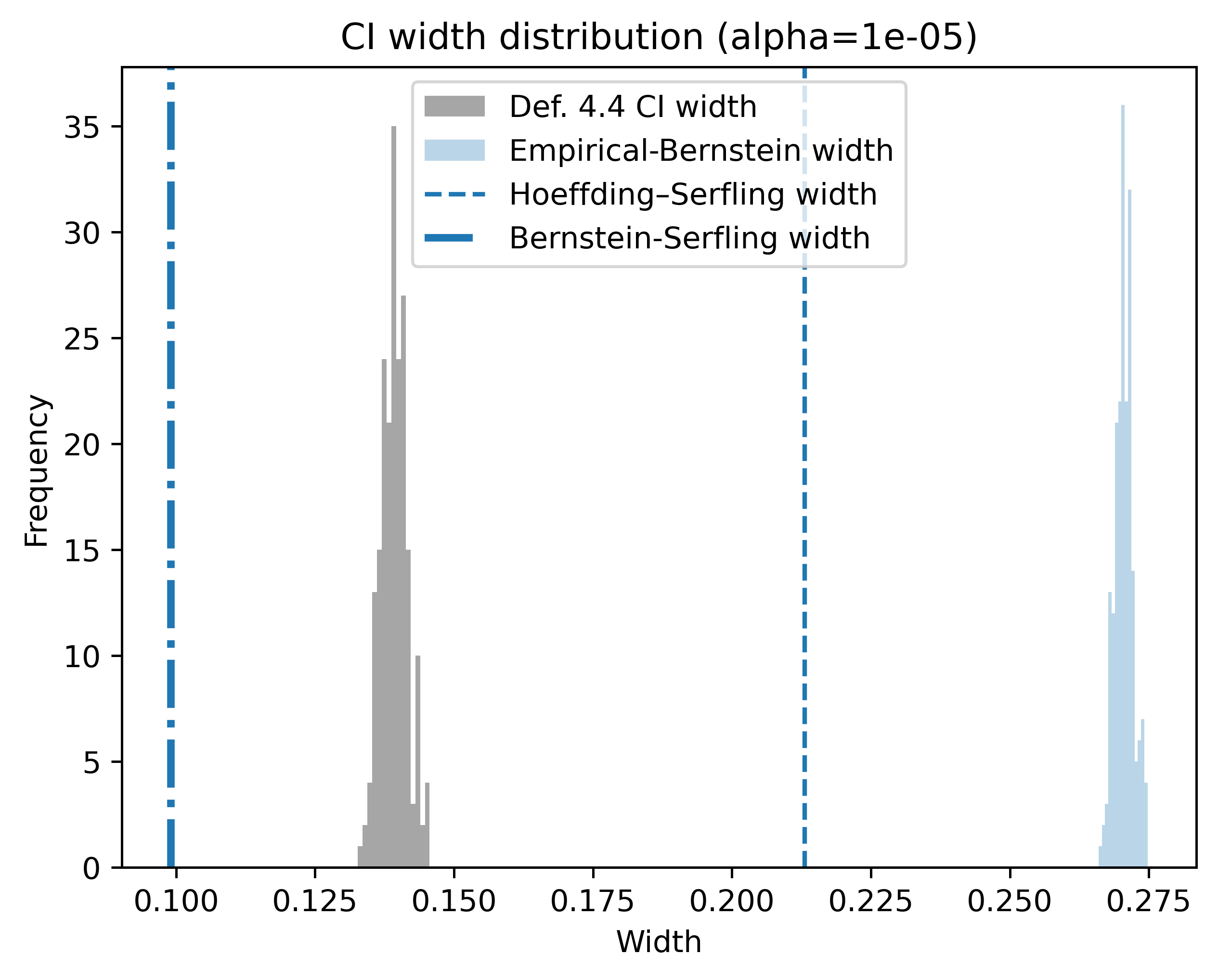}
        \includegraphics[width=0.3\linewidth]{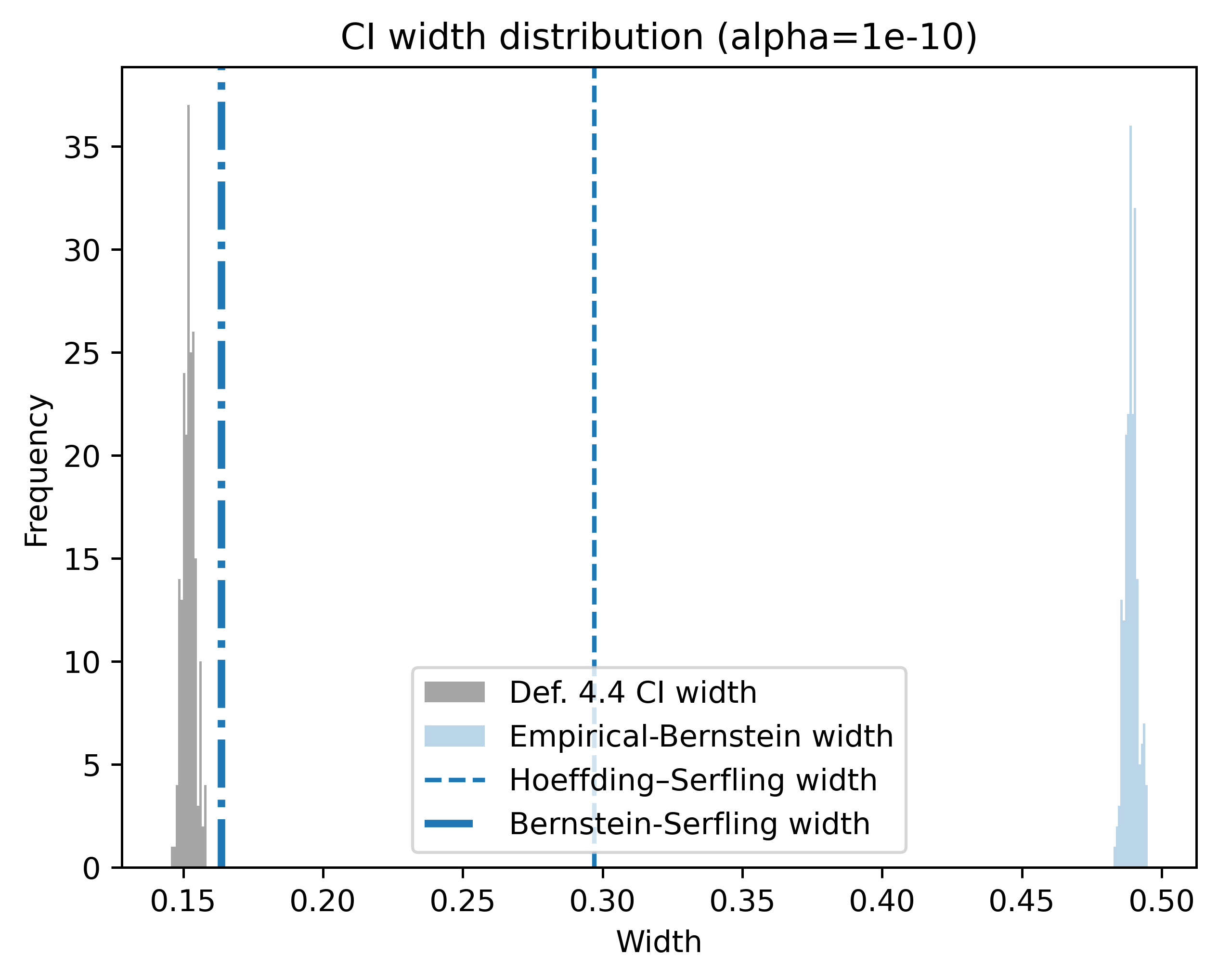}
        \caption{The first figure numerically verifies the correctness of the dual formulation of $J(m) \equiv J(P, \beta, m) \coloneqq \max\{ \Jplus(P, \beta, m), \Jminus(P, \beta, m)\}$ for two randomly generated distributions supported on a finite alphabet of size $4$. The dashed curves show $J(m)$ computed by solving the dual derived in~\Cref{theorem:dual-finite}, and the solid curves represent $J(m)$ computed by directly solving the primal problem~(Definition~\ref{def:complexity-function}). The other two figures compare the widths~(over $200$ trials) of our proposed CI~(gray histogram), Hoeffding, Bernstein, and Empirical Bernstein CI~\citep{bardenet2015concentration}, on a problem with $|\Sigma|=10$, $N=1000$, $\beta=0.35$, and $\alpha \in \{10^{-5}, 10^{-10}\}$. As discussed in~Remark~\ref{remark:empirical-and-population-inverse-information-projections}, in the regime of small enough $\alpha$, our CI of Definition~\ref{def:CI-finite} performs better than the other three methods.}
        \label{fig:dual-vs-primal-comparison}
    \end{figure}       

\section{Proof of Theorem~\ref{theorem:asympCI} (Almost Sure CI)}
\label{proof:asympCI}

\paragraph{Step 1: Approximating $\mathbf{\lambda^*_N(y)}$.} For any $y \geq 0$, let $\lambda^*_N(y)$ denote the value of $\lambda$ that attains the supremum in the definition of $\Lambda^*_N(y)$. Then, our first result shows that for large enough values of $N$, and in a sufficiently small interval around the origin, the parameter $\lambda^*_N(y)$ is a Lipschitz-continuous function of $y$.  
\begin{lemma}
    \label{lemma:asympCI-1} 
    Under Assumption~\ref{assump:limiting-distribution}, there exist constants $N_0 \in \mathbb{N}$, $c_*>0$, $L_0>0$, and $M_3 <\infty$, such that for all $N \geq N_0$: 
    \begin{enumerate}
        \item $\Lambda''_N(0) = \beta_N \barbeta_N m_{2,N} \geq c_* = \beta \barbeta m_2/8$ and $\inf_{|\lambda| \leq L_0} \Lambda''_N(\lambda) \geq c_*/2$, where $m_{2,N} = \int x^2 d\Phat_N(x)$, and $m_2 = \int x^2 dP_X(x) >0$. 
        \item The map $\lambda \mapsto \Lambda'_N(\lambda)$ is strictly increasing on $[0, L_0]$, with range $[0, y_{0,N}]$, where $y_{0,N} = \Lambda'_N(L_0) \geq (c_*/2)L_0$~(that is, the term $y_{0,N}$ is bounded from below by a positive constant independent of $N$, for all $N \geq N_0$). Hence $g_N = (\Lambda_N')^{-1}$ is well-defined on $[0, y_{0,N}]$, and 
        \begin{align}
            0 \leq \lambda^*_N(y) =  g_N(y) \leq \frac{2}{c_*} y, \qtext{for} y \in [0, y_{0, N}]. 
        \end{align}
        \item Let $c_{2,N} = \Lambda''_N(0) = \beta_N \barbeta_N m_{2,N}$. Then, we have the following quadratic approximation for all $y \in [0, y_{0,N}]$: 
        \begin{align}
            \lv \lambda^*_N(y) - \frac{y}{c_{2,N}} \rv \leq  \frac{2 M_3}{c_*^3} y^2, \qtext{where} M_3 \equiv M_3(L_0) = \sup_{N \geq N_0} \sup_{|\lambda|\leq L_0} |\Lambda'''_N(\lambda)| < \infty. 
        \end{align}
        Furthermore, we also have the simpler inequality $\lambda^*_N(y) \leq \frac{y}{c_*} + \frac{2 M_3(L_0)}{c_*^3}y^2$ for all $y \in [0, c_*L_0/2]$. 
    \end{enumerate}
  
\end{lemma}

\begin{proof}
These statements follow from the weak convergence condition of Assumption~\ref{assump:limiting-distribution} combined with some standard approximation arguments. 
\begin{enumerate}
    \item 
Assumption~\ref{assump:limiting-distribution} tells us that $\beta_N \to \beta \in (0, 1)$ and $m_{2,N} \to m_2 = \int x^2 dP_X(x)>0$. Hence, there exists an $N_0$, such that for all $N \geq N_0$, we have $\beta_N \in [\beta/2, (1+\beta)/2]$, and $m_{2,N} \in [m_2/2, 2m_2]$. Hence, we have the following: 
\begin{align}
    \inf_{N \geq N_0} \Lambda''_N(0) = \inf_{N \geq N_0} \beta_N \barbeta_N m_{2,N} \geq \frac{\beta}{2} \frac{\barbeta}{2} \frac{m_2}{2} \eqcolon c_* >0,
\end{align}
where the first equality uses the expression of $\Lambda''_N(0)$ derived in~\Cref{proof:properties-of-f-N-lamdba}. 

Next, let $M_3(L) \coloneqq \sup_{N \geq N_0} \sup_{|\lambda|\leq L} |\Lambda'''_N(\lambda)|$, and let $L_0>0$ be a value of $L$ such that $M_3(L_0)L_0 \leq c_*/2$. Then, observe that for all $|\lambda| \leq L_0$, we have 
\begin{align}
    \Lambda''_N(\lambda) \geq \Lambda''_N(0) - M_3(L_0) L_0 \geq c_* - M_3(L_0) L_0 \geq c_*/2. 
\end{align}
\item  The last result  implies that for all $N \geq N_0$, the first derivative $\Lambda'_N$ is strictly increasing on the domain $[0, L_0]$, which means that $g_N = (\Lambda_N')^{-1}$ is well-defined on $[0, y_{0, N}]$ with $y_{0, N} = \Lambda'_N(L_0) \geq (c_*/2) L_0$. Furthermore, by an application of the inverse function theorem, recalled in Fact~\ref{fact:inverse-function-theorem}, we have 
\begin{align}
    \sup_{y \in [0, y_{0,N}]} |g'_N(y)| \leq \frac{1}{\inf_{\lambda \in [0, L_0]} \Lambda''_N(\lambda)} \leq \frac{2}{c_*}. \label{eq:local-lipschitz-1}
\end{align}
Since we know that $\lambda^*_N(y) = g_N(y)$ for all $y \in [0, y_{0,N}]$,  the above result implies that $\lambda^*_N(y) \leq \lambda^*_N(0) + (2/c_*) y$ for all $y \in [0, y_{0,N}]$. Finally, since $\Lambda'_N(0) =0$, we have $\lambda^*_N(0)=0$, which concludes the proof of part 2. of Lemma~\ref{lemma:asympCI-1}. 

\item We begin with the Taylor series approximation of $\Lambda_N$ around $0$, and for any $\lambda \in [0, L_0]$: 
\begin{align}
    \Lambda_N(\lambda) = \Lambda_N(0) + \Lambda_N'(0) \lambda + \Lambda''_N(0) \frac{\lambda^2}{2} + R_3(\lambda), \qtext{with} |R_3(\lambda)| \leq \frac{M_3(L_0)}{6} |\lambda|^3. 
\end{align}
Recall that from the derivations in~\Cref{proof:properties-of-f-N-lamdba}, we know  $\Lambda_N(0) = \Lambda'_N(0)=0$ and $\Lambda''_N(0) = c_{2,N} = \beta_N \barbeta_N m_{2,N}$. Furthermore, for $N \geq N_0$, we have $c_*/8 \leq c_{2,N} \leq c^* \coloneqq (2\beta) (1-\beta/2) 2 m_2$. Together, these facts imply 
\begin{align}
    &\frac{c_{2,N}}{2} \lambda^2 - \frac{M_3(L_0)}{6} \lambda^3 \leq \Lambda_N(\lambda) \leq \frac{c_{2,N}}{2} \lambda^2 +  \frac{M_3(L_0)}{6}\lambda^3, \\ 
    \qtext{$\implies$}   
    &c_{2,N} \lambda - \frac{M_3(L_0)}{2}\lambda^2 \leq   \Lambda'_N(\lambda) \leq c_{2,N} \lambda + \frac{M_3(L_0)}{2}\lambda^2. 
\end{align}
For any $y \in [0, y_{0,N}]$, the solution $\lambda^*_N(y)$ satisfies $y = \Lambda_N'\lp \lambda^*_N(y) \rp$, which coupled with the previous inequalities leads to 
\begin{align}
   \lv \Lambda_N'\lp \lambda^*_N(y) \rp - c_{2,N} \lambda^*_N(y) \rv \leq  \frac{M_3(L_0)}{2} \lp\lambda^*_N(y) \rp^2 \qtext{$\implies$}
    \lv y - c_{2,N} \lambda^*_N(y) \rv \leq  \frac{M_3(L_0)}{2} \lp\lambda^*_N(y) \rp^2.  
\end{align}
Dividing the last inequality throughout by $c_{2,N} = \barbeta_N \beta_N m_{2,N}>0$~(due to the $N \geq N_0$ assumption), and using the Lipschitz property of $y \mapsto \lambda^*_N(y)$ implied by~\eqref{eq:local-lipschitz-1}, we get the required result 
\begin{align}
    \lv \lambda^*_N(y) - \frac{y}{c_{2,N}} \rv  \leq \frac{M_3(L_0)}{2c_{2,N}} \frac{4}{c_*^2} y^2 \qtext{$\implies$} \lambda^*_N(y) \leq \frac{y}{c_*} + \frac{2 M_3(L_0)}{c_*^3}y^2, 
\end{align}
for all $y \in [0, y_{0,N}]$.  This concludes the proof. 
\end{enumerate}

\end{proof}

\paragraph{Step 2: Approximating $\mathbf{\hat{\lambda}_n(y)}$.} In the next step, we obtain a similar characterization of $\lambdahat_n(y)$, which denotes $\argmax_{\lambda \in \mathbb{R}} \lambda y - \Lambdahat_n(\lambda)$. We can use the exact same argument as in the proof of Lemma~\ref{lemma:asympCI-1}, except that instead of $m_{2,N} = \int x^2 d\Phat_N(x)$, we have to work with $\hat{m}_{2,n} = \int x^2 d\Phat_n(x) = \frac{1}{n} \sum_{i=1}^n X_i^2$. Unlike $m_{2,N}$, this is a random quantity and we expect it to ``stabilize'' away from $0$ for large values on $N$. 

\begin{lemma}
    \label{lemma:asympCI-2}
    Let $G_n$ denote the event $\{ |\hat{m}_{2, n} - m_{2, N}| \leq \sqrt{\log N/\beta_N N}\}$. Then, we have the following: 
    \begin{align}
        \mathbb{P}\lp G_N^c \; \text{infinitely often} \rp = \mathbb{P}\lp \cap_{N \geq 1} \cup_{N' \geq N} G^c_{N'} \rp =  0. 
    \end{align}
    Consequently, there is an almost surely finite $\widetilde{N}$, such that for all $N \geq \widetilde{N}$, 
    \begin{align}
        \hat{m}_{2,n} \geq m_{2,N}/2, \qtext{and} \hat{c}_{2,N} = \beta_N \barbeta_N \hat{m}_{2,n} \geq c^{\dag} >0, 
    \end{align}
    where $c^{\dag} = 0.63 \beta \barbeta m_2 > c_* >0$.  
\end{lemma}
\begin{proof}
    Using Hoeffding's concentration inequality for sampling without replacement,  we have the following for any $N \geq 1$: 
    \begin{align}
        \mathbb{P}\lp |\hat{m}_{2,n} - m_{2, N}| \geq \rho \rp \leq 2 \exp \lp -2 n \rho^2\rp \qtext{$\implies$} \mathbb{P}(G_N) =  \mathbb{P}\lp |\hat{m}_{2,n} - m_{2, N}| \geq \rho_N \rp \leq \frac{2}{N^2}, 
    \end{align}
    where $\rho_N = \sqrt{\frac{\log N}{n}}$. 
    Hence, $\sum_{N \geq 1}  \mathbb{P}(G_N^c) < \infty$, and an 
    application of Borel-Cantelli Lemma then implies 
    \begin{align}
        \mathbb{P}\lp G_N^c \;\text{i.o.} \rp = \mathbb{P}\lp \cap_{N \geq 1} \cup_{M \geq N} G_M^c \rp = 0. 
    \end{align}
    This means that the event $G_N$ occurs for all but finitely many $N$ with probability $1$. 

    For the next part of the proof, we introduce the following terms: 
    \begin{itemize}
        \item By Assumption~\ref{assump:limiting-distribution}, we know that $m_{2,N} \to m_2>0$. Hence, there exists a deterministic $N_1$, such that for all $N \geq N_1$, we have $m_{2,N} \geq m_2 \sqrt[3]{0.9}$, $\beta_N \geq \sqrt[3]{0.9}\beta$ and $\barbeta_N \geq \sqrt[0.3]{0.9}\barbeta_N$. 
        \item Since $\rho_N \to 0$ with $N$, there exists a deterministic $N_2$, such that for all $N \geq N_2$, we have $\rho_N \leq m_2/4$. 
        \item On the a.s.\ event under which $G_n^c$ holds for all but finitely many $N$, there exists a random $N_{BC}$, such that $|\hat{m}_{2,n} - m_{2,N}| \leq \rho_N$ for all $N \geq N_{BC}$. 
    \end{itemize}
    Thus, we can define an almost surely finite random time $\widetilde{N} = \max\{N_1, N_2, N_{BC}\}$, under which 
    \begin{align}
        \hat{m}_{2,n} \geq m_{2,N} - \rho_N \geq \sqrt[3]{0.9} m_2- \frac{m_2}{4} \geq 0.7m_2. 
    \end{align}
    Furthermore, for $N \geq \widetilde{N}$, we also have the following bound on $\hat{c}_{2,N}$: 
    \begin{align}
        \hat{c}_{2,N} = \beta_N \barbeta_N \hat{m}_{2,n} \geq 0.63 \beta \barbeta m_2 \eqcolon c^{\dag} > c_* >0. 
    \end{align}
\end{proof}
Fix an $L_0$ small enough to ensure that $L_0 M_3(L_0) \leq c_*/2$, where $M_3(\cdot)$ and $c_*$ were defined in the proof of Lemma~\ref{lemma:asympCI-1}. On the event $N \geq \widetilde{N}$, we have $\hat{c}_{2,N} \geq c^\dag$, which implies that 
\begin{align}
    \inf_{|\lambda|\leq L_0} \hat{\Lambda}''_n(\lambda) \geq \hat{c}_{2,N} - M_3(L_0) L_0 \geq c^{\dag} - c_*/2 \geq c^\dag/2 > 0.  
\end{align}
Hence, $\hat{\Lambda}_n'$ is strictly increasing on $[0, L_0]$ on the event $N \geq \widetilde{N}$, and the same Taylor series approximation argument holds for $\hat{\lambda}_n$. In particular, for all $y \in [0, c^\dagger L_0/2]$, we have  
\begin{align}
    \hat{\lambda}_n(y) \leq \frac{3y }{\beta \barbeta m_2} + C_2 y^2,
\end{align}
for some universal constant $C_2>0$~(again, under the event that $\widetilde{N} \leq N$).

\paragraph{Step 3: Get the bound $\mathbf{y_N = \calO(N^{-1/2})}$ and $\mathbf{\hat{y}_n = \calO(N^{-1/2})}$.} We now show a natural upper bound on the size of $y_N$ which is the value at which $\Lambda^*(y_N) = t_{N, \alpha}$. In our oracle CI construction, $y_N = \beta_N \epsilon_N$. 

\begin{lemma}
    \label{lemma:asympCI-3}
    Let $y_N$ denote the value such that $\Lambda^*_N(y_N) = t_N$. Then, for large enough values of $N$, 
    \begin{align}
        y_N = \sqrt{c_{2,N}t_{N,\alpha}}\lp1 +  \calO(\sqrt{t_{N,\alpha}}) \rp. 
    \end{align}
\end{lemma}
\begin{proof}
For any $y> 0$, recall that we have 
\begin{align}
    \Lambda^*_N(y) = y \lambda^*_N(y) - \Lambda_N(\lambda^*_N(y)). 
\end{align}
Now, by Lemma~\ref{lemma:asympCI-1}, we know that for all $y<y_0$~(where the constant $y_0 = L_0 c_*/2$ does not depend on $N$), we have  $\lambda^*_N(y) \approx y/c_{2,N} + \calO(y^2)$ for some constant $c_{2,N}> c_*/8 >0$. Plugging this into the definition of $\Lambda^*_N(y)$, we get 
\begin{align}
     \Lambda^*(y) = \frac{y^2}{2c_{2,N}}  + \calO(y^3). 
\end{align}
Now, let $y_N$ denote the value of $y$ such that $\Lambda^*_N(y_N) = t_{N, \alpha}$, which implies that for some universal constant $C<\infty$, we have 
\begin{align}
    \frac{y_N^2}{c_{2,N}} - C y_N^3 \leq t_N \leq \frac{y_N^2}{c_{2,N}} + C y_N^3. \label{eq:yN-comparison}
\end{align}
Now, if $y_N \leq 1/(2 c_* C)$, then we have $y_N \leq \sqrt{2 c_{2,N} t_{N, \alpha}}$. 
Applying this bound to the inequality~\eqref{eq:yN-comparison},  we get the sharper upper bound
\begin{align}
     y_N \leq \sqrt{c_{2,N} t_N}\lp 1 + \calO(\sqrt{t_{N,\alpha}}) \rp. 
\end{align}
\end{proof}
Similarly, let us denote $\hat{y}_n = \hat{\Lambda}_n(t_N)$,  and the same argument implies that 
\begin{align}
    \hat{y}_n = \sqrt{\hat{c}_{2,N} t_N}\lp 1 + \calO(\sqrt{t_{N,\alpha}}) \rp, \qtext{for all } N \geq \widetilde{N}, 
\end{align}
for an $\widetilde{N}< \infty$ almost surely.

\paragraph{Step 4: Bounding the deviation of the conjugate CGFs.}  So far, we have seen that for $y \leq y_0$ and $\lambda \leq L_0$~(where $y_0$ and $L_0$ are universal constants), we have $\lambda^*_N(y) = y/c_2 + \calO(y^2)$, where $c_2 = \beta_N \barbeta_N m_{2,N}$ and $m_{2,N} = \int x^2 d\Phat_N(x)$. A similar bound holds for $\lambdahat_n(y)$ for $N$ large enough. Then, we established that $y_N$ and $\hat{y}_n$ are both $\calO(\sqrt{t_N}) = \calO(\sqrt{\log N/N})$ for large enough values of $N$. Together these results imply that for large values of $N$, we have $y_N, \hat{y}_n \leq y_0$, and $\lambda^*_N(y_N), \lambdahat_n(\hat{y}_n) \leq L_0$. 

Let us now consider $M_N = N^{-0.45}$ and $L_N = N^{-0.4}$. For large values on $N$, our previous results imply that with $|y|\leq M_N$, we  have $\max\{|\lambda^*_N(y)|, \,|\lambdahat_n(y)|\} \leq L_N$. For any $y \in [0, M_N]$, we have  by triangle inequality: 
\begin{align}
    \lambda y - \Lambda_N(\lambda) \leq \lambda y - \hat{\Lambda}_n(\lambda) + | \hat{\Lambda}_n(\lambda) - \Lambda_N(\lambda)|, 
\end{align}
which implies that 
\begin{align}
    \Lambda^*_N(y) &= \sup_{|\lambda|\leq L_N } \lambda y - \Lambda_N(\lambda) \leq \sup_{|\lambda|\leq L_N} \lp \lambda y - \hat{\Lambda}_n(\lambda) + | \hat{\Lambda}_n(\lambda) - \Lambda_N(\lambda)|\rp  \\
     &\leq \lp \sup_{|\lambda|\leq L_N}  \lambda y - \hat{\Lambda}_n(\lambda) \rp  + \lp \sup_{|\lambda|\leq L_N} | \hat{\Lambda}_n(\lambda) - \Lambda_N(\lambda)|\rp   
     = \hat{\Lambda}_n^*(y) + \sup_{|\lambda|\leq L_N} | \hat{\Lambda}_n(\lambda) - \Lambda_N(\lambda)|. 
\end{align}
The crucial fact is that for any $y \leq M_N = N^{-0.45}$, the optimal $\lambda^*_N(y)$ and $\hat{\lambda}_n(y)$ are both $\calO(N^{-0.45})$, and thus for large enough $N$ lie in the interval $[0, L_N]$, since $L_N = N^{-0.4}$.  

Repeating the argument and swapping the roles of $\Lambda_N^*$ and $\hat{\Lambda}_n$, we can conclude that 
\begin{align}
    |\Lambda^*_N(y) - \hat{\Lambda}_n(y)| & \leq \sup_{|\lambda|\leq L_N} | \hat{\Lambda}_n(\lambda) - \Lambda_N(\lambda)| = \sup_{|\lambda|\leq L_N} \lv \frac{\hat{m}_{2,n}-m_{2, N}}{2} \lambda^2 + \frac{D}{3} \lambda^3 \rv, 
\end{align}
where $D = M_3(L_0) L_0$, and $M_3(\cdot)$ and $L_0$ were introduced in the proof of Lemma~\ref{lemma:asympCI-1}. 
Since $|\hat{m}_{2,n} - m_{2, N}| = \calO( \sqrt{\log N/N})$ and $|\lambda| \leq L_N = N^{-0.4}$, we have the following for all $|y|\leq M_N = N^{-0.45}$
\begin{align}
    |\Lambda^*_N(y) - \hat{\Lambda}_n(y)| & \leq \calO \lp \frac{\sqrt{\log N}}{N^{1.3}} + \frac{1}{N^{1.2}} \rp = \calO\lp N^{-1.2} \rp. 
\end{align}

\paragraph{Step 5: Comparing the widths.} Now, we know that $t_{N,\alpha} = \Lambda^*_N(y_N)$, and so we have  the following for all large enough values of $N$: 
\begin{align}
    \hat{\Lambda}_n^*(y_N) \leq \Lambda^*_N(y_N) + \calO(N^{-1.2}) = t_{N,\alpha} + \calO(N^{-1.2}) \leq \hat{t}_{N,\alpha} \coloneqq t_{N, \alpha} + N^{-1.1}, 
\end{align}
Since $\hat{y}_n$ is defined as the value of $y$ at which $\hat{\Lambda}^*_n(\hat{y}_n) = \hat{t}_n$, by the monotonicity of $\hat{\Lambda}^*_n$, we can conclude that $\hat{y}_n \geq y_N$, which implies that $\hat{\epsilon}_n \geq \epsilon_N$. To summarize, we have proved that there exists a random $N^*$, such that $N^* < \infty$ almost surely, and we have $\hat{\epsilon}_n > \epsilon_N$ for all $N \geq N^*$. Equivalently, we have shown that 
\begin{align}
    \mathbb{P}\lp \hat{\epsilon}_n < \epsilon_N \; \text{infinitely often} \rp =  0. 
\end{align}
This completes the proof of the first part of~\Cref{theorem:asympCI}. To prove the second part, note that the argument also works with an arbitrary sequence $\{\alpha_N: N \geq 1\}$ that are ``moderately small''; that is, $\alpha_N$ decay at a polynomial rate with $N$. Consider such a sequence $\{\alpha_N: N \geq 1\}$ with $\sum_{N \geq 1} \alpha_N < \infty$. For example, we can set $\alpha_N = N^{-2}$. Now, observe that 
\begin{align}
    \sum_{N = 1}^{\infty} \mathbb{P}\lp \mu_N \not \in \Corc_N \rp \leq \sum_{N =1}^{\infty} \alpha_N < \infty, 
\end{align}
which means that $\mathbb{P}(\mu_N \not \in \Corc_N \text{ i.o.}) = 0$. In other words, there is a random time $\Norc$ such that $\Norc < \infty$ almost surely, and for all $N \geq \Norc$, we know that $\mu_N \in \Corc_N$. Thus, by the first part of the theorem, we also know that for all $N \geq \Nemp \coloneqq N^* \vee \Norc$, we  have $\mu_N \in \Cemp$.  In other words, we have 
\begin{align}
 \sum_{N = 1}^{\infty} \boldsymbol{1}_{\mu_N \not \in \Cemp}    \leq \Nemp = \max\{\Norc, N^*\}  \stackrel{a.s.}{<}  \infty, 
\end{align}
as required. This completes the proof of~\Cref{theorem:asympCI}.

\subsection{Formal Description of Almost Sure Coverage}
\label{appendix:formal-EAS-convergence}
The statement of~\Cref{theorem:asympCI} contains two ``almost sure'' events; one on the containment of the oracle CI within the empirical or computable CI, and the other about the empirical CIs covering the population mean. In this section, we construct the relevant probability spaces to provide a rigorous definition of these events. 

We begin by considering a triangular array $\{\calX_N: N \geq 1\}$ where $\calX_N = \{x^{(N)}_1, \ldots, x^{(N)}_N\}$ denotes the $N^{th}$ population in the infinite sequence. We assume that these arrays satisfy Assumption~\ref{assump:limiting-distribution}: 
\begin{align}
    \Phat_N = \frac{1}{N} \sum_{i=1}^N \delta_{x^{(N)}_i} \Rightarrow P_X, \qtext{and}  \mu_N = \frac{1}{N} \sum_{i=1}^N x^{(N)}_i \to \mu. 
\end{align}

For every $N \geq 1$, let $\calS_N$ denote the set of all permutations of $[N] = \{1, \ldots, N\}$ with the associated uniform law $\mathbb{P}_N$, and define the product space 
\begin{align}
    (\Omega, \calF_\infty, \mathbb{P}) \quad = \quad \lp \prod_{N \geq 1} \calS_N ,\, \underset{N \geq 1}{\mathlarger{\otimes}} 2^{\calS_N}, \, \underset{N \geq 1}{\mathlarger{\otimes}}  \mathbb{P}_N \rp. 
\end{align}
Here we use $\prod$ to denote the cartesian product of the sample spaces, and $\otimes$ to denote the construction of product sigma-algebra and product measure. 
Let $\Pi_N: \Omega \to \calS_N$ denote the $N^{th}$ coordinate projection, and for $N \geq 1$, the \wor sample can be written as 
\begin{align}
    X^{(N)}_i(\omega) = x^{(N)}_{\Pi_N(\omega)(i)}, \qtext{for} \omega \in \Omega, \; i \in [\beta_N N].     
\end{align}

Fix a sequence $\{\alpha_N: N \geq 1\}$ with $\alpha_N \to 0$ and $\sum_{N\geq 1} \alpha_N < \infty$, and define the following non-containment and miscoverage events, respectively: 
\begin{align}
    M_N^{\mathrm{orc}} = \{\mu_N \not \in \Corc_N \}, \quad E_N = \{\Corc_N \not \subset \Cemp_n \}, \qtext{for} N \geq 1.
\end{align}
Both of these events are $2^{\calS_N}$, and therefore $\calF_\infty$, measurable. Using these, we can define the two almost sure events used in~\Cref{theorem:asympCI}. The first is about the eventual containment of $\Corc_N$ in $\Cemp_n$: 
\begin{align}
    E_\infty &= \lbr \omega \in \Omega: \exists N_0(\omega) < \infty, \; \text{s.t.}\; \forall N \geq N_0(\omega), \; \Corc_N(\omega) \subset \Cemp_n(\omega) \rbr \\
    &= \lbr \omega \in \Omega: \lv \lbr N \geq 1: \Corc_N(\omega) \not \subset \Cemp_n(\omega) \rbr \rv < \infty \rbr \\
    & = \cup_{N \geq 1} \cap_{N' \geq N} E_N^c  \\
    & = \lp \cap_{N \geq 1} \cup_{N' \geq N} E_N \rp^c. 
\end{align}
It follows that $E_\infty$ is $\calF_\infty = \otimes_{N \geq 1} 2^{\calS_N}$-measurable, and the first statement of~\Cref{theorem:asympCI} says that $\mathbb{P}(E_\infty) = 1$, and the empirical CI is wide enough to contain the population CI after a random finite number of rounds.  

The second almost sure event in~\Cref{theorem:asympCI} is a consequence of the choice of the sequence of $\alpha_N \downarrow 0$, and considers the event 
\begin{align}
    M_\infty & = \lbr \omega \in \Omega: \exists N_1(\omega) < \infty, \; \text{s.t.}\; \forall N \geq N_1(\omega), \mu_N \in \Cemp_n(\omega) \rbr \\
    &= \lbr \omega \in \Omega: \lv\lbr N \geq 1: \mu_N \not \in \Cemp_n(\omega)  \rbr\rv < \infty \rbr  \\
    & = \cup_{N \geq 1} \cap_{N'\geq N} M_N^c \\
    & = \lp \cap_{N \geq 1} \cup_{N' \geq N} M_N \rp^c. 
\end{align}
Again, $M_\infty$ is $\calF_\infty$-measurable, and the second part of~\Cref{theorem:asympCI} says that $\mathbb{P}(M_\infty)=1$. Together, these two events give us an ``almost sure'' coverage guarantee as used by~\citet{naaman2016almost}.

    \section{Proof of~Theorem~\ref{theorem:banach-space-CI} (Banach Space CI)}
    \label{proof:banach-space-CI}

    For any $\epsilon>0$, we have the following by Fact~\ref{fact:coupling}: 
    \begin{align}
        \mathbb{P}\lp \| \bar{X}_n - \mu_N \| \geq \epsilon \rp &\leq 1.18 \sqrt{\barbeta n}\; \mathbb{P}\lp \left\lVert \frac{1}{n} \sum_{i=1}^N J_i x_i  - \mu_N\right\rVert  \geq \epsilon\rp  \\
        & = 1.18 \sqrt{\barbeta n}\; \mathbb{P}\lp \left\lVert \frac{1}{n} \sum_{i=1}^N (J_i -\beta)x_i \right\rVert  \geq \epsilon\rp, 
    \end{align}
    where $(J_1, \ldots, J_N) \stackrel{i.i.d.}{\sim} \mathrm{Bernoulli}(\beta)$ for $\beta = n/N$. Since the random variable $W \defined \lVert \sum_{i=1}^N Z_i\rVert$ with $Z_i = (J_i-\beta) x_i$ is non-negative, we have the following for any $\lambda, \epsilon \geq 0$: 
    \begin{align}
        \mathbb{P}\lp W \geq n \epsilon \rp \leq e^{-\lambda n \epsilon} \mathbb{E}\lb e^{\lambda W} \rb. \label{eq:copuling-pinelis-proof-0}
    \end{align}
    We can extract a bound on the moment generating function~(MGF) $\mathbb{E}[e^{\lambda W}]$ from the proof of~\citet[Theorem 3.2]{pinelis1994optimum} to get 
    \begin{align}
        \mathbb{E}\lb e^{\lambda W} \rb \leq 2 \mathbb{E}\lb \cosh(\lambda W) \rb \leq 2\prod_{i=1}^N (1 + e_i), \qtext{where} 
        e_i = D^2 \mathbb{E}\lb \lp e^{\lambda \|Z_i\|} - 1 - \lambda \|Z_i\| \rp\rb. \label{eq:coupling-pinelis-proof-1}
    \end{align}
    To complete the construction of the CI, we need to obtain an upper bound on each $e_i$, and optimize over $\lambda \geq 0$. We present two methods of doing this: the first method gives a tighter CI whose width is the solution of a one-dimensional optimization problem, while the second method gives a looser CI with a closed-form expression of its width. 

    We start with  expression of $e_i$ in the MGF bound~\eqref{eq:coupling-pinelis-proof-1}, and observe that  
    \begin{align}
        \frac{e_i}{D^2} &= \mathbb{E} \lb e^{\lambda \|Z_i\|} - 1 - \lambda \|Z_i\| \rb \\
        & = \beta \lp e^{\lambda \barbeta \|x_i\|}  - 1 - \lambda \barbeta \|x_i\|\rp + \barbeta \lp e^{\lambda \beta \|x_i\|}  - 1 - \lambda \beta \|x_i\|\rp && (\text{since } Z_i = (J_i-\beta) x_i) \\
        & \leq \beta \lp e^{\lambda \barbeta d} - 1 - \lambda \barbeta d\rp + \barbeta \lp e^{\lambda \beta d} - 1 - \lambda \beta d\rp, && (\text{since } \|x_i\|\leq d) \\
        & = \beta e^{\lambda \barbeta d} + \barbeta e^{\lambda \beta d} - 1 - 2\lambda \beta \barbeta d. 
    \end{align}
    where the inequality uses the fact that $x \mapsto e^{x} - 1 - x$ is an increasing function. 
    \begin{align}
        \mathbb{P}(W > n \epsilon) \leq e^{-n \lambda \epsilon + n g(\lambda)}, \qtext{where} g(\lambda) = \frac{1}{\beta} \log \lp 1 + D^2\lp \beta e^{\lambda \barbeta d} + \barbeta e^{\lambda \beta d} - (1 + \lambda d)\rp \rp.  \label{eq:banach-space-CI-numerical-1}
    \end{align}
    This implies 
    \begin{align}
        \mathbb{P}\lp \|\bar{X}_n - \mu_N \| > \epsilon \rp \leq 2.36 \sqrt{\barbeta n}\, e^{-n (\lambda \epsilon - g(\lambda))}, \qtext{for all}  \lambda \geq 0. 
    \end{align}
    Inverting this tail inequality, we get the following result. 
    \begin{align}
        &\mathbb{P}\lp \|\bar{X}_n - \mu_N\| > \epsilon_n(\lambda) \rp \leq \alpha,  \\
        \text{where} \quad 
        & \epsilon_n(\lambda) = \frac{1}{\lambda}\lp \frac{2.36 \log(\sqrt{(1-\beta)n}/\alpha)}{n}  + g(\lambda) \rp. 
    \end{align}
    Since this statement is true for all $\lambda >0$, we can numerically optimize to find the optimal width of the CI as 
    \begin{align}
        \epsilon_n^* = \inf_{\lambda >0} \; \frac{1}{\lambda} \lp \frac{\ell_n}{n} + g(\lambda)  \rp \qtext{and} \ell_n \equiv \ell_n(\alpha, \beta) = \log \lp \frac{2.36 \sqrt{(1-\beta)n}}{\alpha}\rp. \label{eq:banach-CI-numerical-2}
    \end{align}

    \subsection{Closed-form CI} 
    In our second approach, we will obtain a more tractable upper bound on MGF. 
    \begin{lemma}
        \label{lemma:coupling-pinelis-exp-bound} 
        Let $h(x) = (e^x - 1- x)/x^2$ for $x>0$ and set $h(0)=1/2$ by continuity. Then, $h(x)$ is an increasing function on the domain $[0,1]$. As a result, we can conclude that 
        \begin{align}
            e^x - 1 - x \leq h(a) x^2, \qtext{for all} x \in [0,a] \subset [0,1]. 
        \end{align}
        In particular, selecting $a=1$, we get $e^x - 1 - x \leq (3/4) x^2$ for all $x \in [0,1]$, and with $a\approx 0.52$, we get $e^x - 1 - x \leq (3/5) x^2$ for all $x \in [0, 0.52]$. 
    \end{lemma}
    Using this result with $a=1$, we get 
    \begin{align}
    \lambda \times \|Z_i\| \leq \lambda d \times \max\{\beta, 1-\beta\} \leq 1 \qtext{$\implies$}
        e_i = \mathbb{E}\lb e^{\lambda \|Z_i\|} - 1 - \lambda \|Z_i\| \rb \leq \frac{3}{4} \lambda^2 \mathbb{E}\lb \|Z_i\|^2 \rb. \label{eq:coupling-pinelis-proof-2}
    \end{align}
    Now, observe that 
    \begin{align}
        \mathbb{E}[\|Z_i\|^2] = \mathbb{E}\lb \|(J_i-\beta)x_i \|^2 \rb = \beta \barbeta \|x_i\|^2 \leq \beta \barbeta d^2, \label{eq:coupling-pinelis-proof-3}
    \end{align}
    where the last inequality uses the assumption that $\|x_i\| \leq d$ for all $i \in [N]$.

    Combining~\eqref{eq:coupling-pinelis-proof-1}, \eqref{eq:coupling-pinelis-proof-2}, and \eqref{eq:coupling-pinelis-proof-3} with the fact that $(1 + x) \leq e^x$, we obtain 
    \begin{align}
        \mathbb{E}\lb e^{\lambda W } \rb \leq 2 \exp \lp \frac{3}{4} \lambda^2 D^2 \sum_{i=1}^N  \barbeta \beta d^2 \rp = 2 \exp \lp \frac{3 \barbeta d^2 D^2 \lambda^2 n}{4} \rp = 2 \exp \lp A \lambda^2 n \rp, 
    \end{align}
    where we use $A$ to denote $3 \barbeta d^2 D^2/4$. 
    Plugging this MGF bound back into~\eqref{eq:copuling-pinelis-proof-0}, we get 
    \begin{align}
        \mathbb{P}\lp W \geq n \epsilon \rp \leq 2\exp \lp -n \lp \lambda \epsilon - A \lambda^2 \rp  \rp. 
    \end{align}
    This upper bound is optimized at $\lambda = \epsilon/2A$. 

    We will restrict our attention to $\epsilon$ is small enough to ensure that $\epsilon/2A \leq 1/(d \max\{\beta, (1-\beta)\})$. Putting in $\lambda = \epsilon/2A$, we get 
    \begin{align}
        \mathbb{P}\lp \|\bar{X}_n - \mu_N \| > \epsilon \rp  \leq 1.18 \sqrt{(1-\beta) n}\, \mathbb{P}\lp W \geq n \epsilon \rp
        \leq 2.36 \sqrt{(1-\beta) n} e^{-n \epsilon^2/4A}. 
    \end{align}
    This implies that 
    \begin{align}
        \mathbb{P}\lp \|\bar{X}_n - \mu_N\| \geq \epsilon_n \rp \leq \alpha, \qtext{for} 
        \epsilon_n = Dd \sqrt{ \frac{3\barbeta\ell_n}{n}}, 
    \end{align}
    where recall that $\ell_n$ was introduced in~\eqref{eq:banach-CI-numerical-2}. 
    This CI is valid as long as $\epsilon_n/2A \leq 1/d\max\{\barbeta, \beta\}$, and a sufficient condition for this is if 
    \begin{align}
        \frac{n}{\log(2.36\sqrt{\barbeta n}/\alpha)} \geq \frac{4 (\max\{\beta, \barbeta\})^2 }{3 D^2 \barbeta}.  \label{eq:condition-for-subgaussian-tail}
    \end{align}
    For $\beta = 0.3, D=1.5, d=1, \alpha=0.05$, this inequality holds for all $n \geq 2$.

\section{\texorpdfstring{Comparison of Banach-Space CI with~\citet{schneider2016probability}}{{Comparison of Banach-Space CI with~Schneider~(2016)}}}   
\label{appendix:width-comparison}
\begin{figure}[htb!]
    \centering
    \hspace{-3.5em}
      \begin{minipage}[t]{0.49\linewidth}
    \centering
    \begin{tikzpicture}
  \pgfmathsetmacro{\alpha}{0.05}

  \begin{axis}[
    title={(Theorem~\ref{theorem:banach-space-CI}) vs \citep{schneider2016probability}},
    title style={font=\small},
    xlabel={$\beta = n/N$},
    ylabel={{\large $\epsilonSch/\epsilon_n$ }},
    domain=0.001:0.99,
    samples=150,
    xmin=0, xmax=1,
    ymin=0,ymax=1.8,
    grid=none,
    legend style={
      at={(0.95,0.45)}, anchor=north east,
      font=\footnotesize,
      /tikz/every even column/.append style={column sep=1em}, draw=none
    },
    legend columns=1
  ]

    \addplot[black!80, thick, smooth, mark=none] {
      sqrt(
        (8*(1 - x + 1/200) * ln(2/\alpha))
        /
        (3*(1 - x)      * ln(2.36*sqrt(x*(1 - x)*200)/\alpha))
      )
    };
    \addlegendentry{$N=200$}

    \addplot[black!80,  semithick, densely dashed, mark=none] {
      sqrt(
        (8*(1 - x + 1/1000) * ln(2/\alpha))
        /
        (3*(1 - x)        * ln(2.36*sqrt(x*(1 - x)*1000)/\alpha))
      )
    };
    \addlegendentry{$N=1000$}

    \addplot[black!60, ultra thick, loosely dotted, mark=none] {
      sqrt(
        (8*(1 - x + 1/2000) * ln(2/\alpha))
        /
        (3*(1 - x)      * ln(2.36*sqrt(x*(1 - x)*2000)/\alpha))
      )
    };
    \addlegendentry{$N=2000$}

    \addplot[black!50, thick,  dash pattern=on 9pt off 3pt on 2pt off 3pt] {
      sqrt(
        (8*(1 - x + 1/10000) * ln(2/\alpha))
        /
        (3*(1 - x)        * ln(2.36*sqrt(x*(1 - x)*10000)/\alpha))
      )
    };
    \addlegendentry{$N=10000$}

    \addplot[black,  thick, densely dotted, mark=none] {
      sqrt(
        (8*(1 - x + 1/20000) * ln(2/\alpha))
        /
        (3*(1 - x)         * ln(2.36*sqrt(x*(1 - x)*20000)/\alpha))
      )
    };
    \addlegendentry{$N=20000$}

    \addplot[dashed,gray,domain=0:1,samples=2,mark=none] {1};
  \end{axis}
\end{tikzpicture}
  \end{minipage}\hfill
      \begin{minipage}[t]{0.49\linewidth}
    \centering
    \begin{tikzpicture}
\begin{axis}[
  xlabel={Sample size $n$},
  ylabel={$\|\mu_n - \mu_N\|_k$},
  title={Kernel Mean Embedding on MNIST},
  xmin=0, xmax=1000,
  legend pos=north east,
  legend style={font=\footnotesize, cells={align=left}, draw=none},
  xticklabel style={font=\footnotesize},   %
  grid=none,
  thick,
  every axis plot/.append style={line cap=round} %
]

\addplot[black!80, very thick, mark=none] 
  table[col sep=comma, x=n, y=mean]{./figures/kernel_mnist_data.csv};
\addlegendentry{Average $\|\mu_n - \mu_N\|_k$}

\addplot[
  black!70,           %
  semithick,
  dash pattern=on 10pt off 3pt,
] table[col sep=comma, x=n, y=mean_sch]{./figures/kernel_mnist_data.csv};
\addlegendentry{$\epsilon_{n,\mathrm{Sch}}$~\citep{schneider2016probability}}

\addplot[
  black!45,
  ultra thick,
  dash pattern=on 2pt off 2pt on 8pt off 2pt,
] table[col sep=comma, x=n, y=mean_cp]{./figures/kernel_mnist_data.csv};
\addlegendentry{$\epsilon_{n}$ (Theorem~\ref{theorem:banach-space-CI})}

\addplot[
  black!55,
  very thick,
  dotted,
] table[col sep=comma, x=n, y=mean_star]{./figures/kernel_mnist_data.csv};
\addlegendentry{$\epsilon^*_{n}$ (Theorem~\ref{theorem:banach-space-CI})}

\end{axis}
\end{tikzpicture}
  \end{minipage}\hfill

        \caption{\textbf{(Left)} We plot the ratio of the width of the CI of~\citet{schneider2016probability}, denoted by $\epsilonSch$, and the CI obtained in~\Cref{theorem:banach-space-CI}, denoted by~$\epsilon_n$ as the sample size $n$ varies from $1$ to $N$. We fix $\alpha=0.05$ and consider five values of $N \in \{200, 1000, 2000, 10000, 20000\}$. In all instances, the CI derived in~\Cref{theorem:banach-space-CI} uniformly dominates that of~\citet{schneider2016probability}, and the improvement is more pronounced for smaller values of $N$. \textbf{(Right)} The solid black line is the average value of the $\|\mu_N-\mu_n\|_k$ over $100$ independent trials. Here $\mu_N$ and $\mu_n$ denote  the empirical kernel mean embeddings defined using $N=1000$ MNIST images and a \wor sample of size $n\leq N$ respectively using the Mat\'ern-$3/2$ kernel. As the plot indicates, the CIs of~\Cref{theorem:banach-space-CI} are tighter than the CI of~\citet{schneider2016probability} shown by the dashed curve in the entire range of $n \in \{1, \ldots, N\}$.}
        \label{fig:width-comparison}
\end{figure}
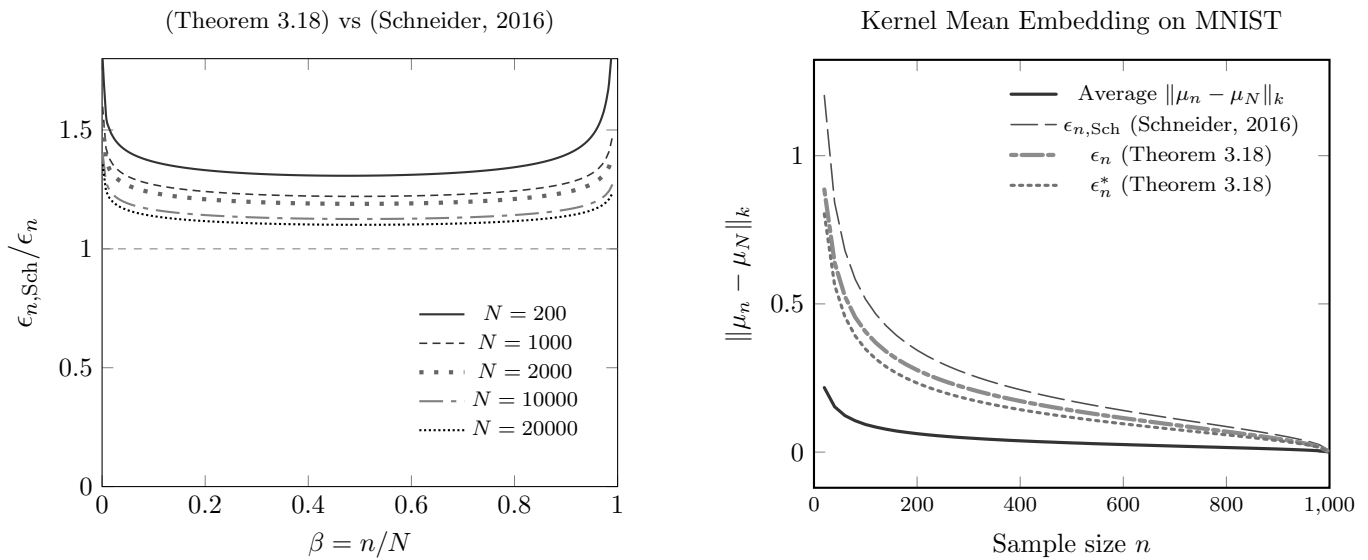

\paragraph{Application to kernel mean embedding approximation.} A special case of a $(2, D)$-smooth Banach space is a reproducing kernel Hilbert space~(RKHS) associated with a positive definite kernel $k$, for which the condition~\eqref{eq:smoothness-condition} holds with an equality with $D=1$. An important statistical application of RKHSs is that of comparing two probability measures $P$ and $Q$ via their kernel mean embeddings: $\mu_P \defined \mathbb{E}_{X \sim P}[k(X, \cdot)]$  and $\mu_Q  \defined \mathbb{E}_{Y \sim Q}[k(Y, \cdot)]$. For certain classes of kernels, called characteristic kernels, the mapping $P \mapsto \mu_P$ is injective, and hence $d_{MMD}(P, Q) = \|\mu_P - \mu_Q\|_k$~(called the kernel maximum mean discrepancy or kernel-MMD) is a metric on the space of probability measures.  In practice, we often have access to a dataset $\calX_N = \{x_1, \ldots, x_N\}$, and we are interested in evaluating the empirical kernel mean embedding $\mu_N = \frac{1}{N} \sum_{i=1}^N k(X_i, \cdot)$. When $N$ is large, it may be preferable to sample \wor a smaller set $X^n = (X_1, \ldots, X_n)$ to construct an estimate $\mu_n = \frac{1}{n} \sum_{i=1}^n k(X_i, \cdot)$. The CIs derived in~\Cref{theorem:banach-space-CI} and by~\citet{schneider2016probability} can be used to quantify the deviation between $\mu_N$ and $\mu_n$. 

To illustrate this, we follow~\citet{schneider2016probability} and select a subset of $N=1000$ images from the MNIST dataset to create the set $\calX_N = \{x_1, \ldots, x_N\}$. Our goal is to approximate the associated kernel mean embedding $\mu_N$ defined using the Mat\'ern kernel with parameter $3/2$. We select $n$ values in the range $\{20, 40, \ldots, 1000\}$, and obtain the average $\|\mu_n - \mu_N\|_k$ value over $100$ independent trials. This curve is then compared with the deviation bounds $\epsilonSch$ and $\epsilon_n$, obtained by~\citet{schneider2016probability} and~\Cref{theorem:banach-space-CI} respectively, over the entire range of $n$ values in Figure~\ref{fig:width-comparison}. Both CIs introduced in~\Cref{theorem:banach-space-CI} are strictly tighter than the CI of~\citet{schneider2016probability} for all values of $n$. 
        
\section{Additional Technical Details}
    \subsection{Justification of Equation~(\ref{eq:coupling-upper-bound-1})}
    \label{proof:coupling-upper-bound-1}
    We begin by recalling the following bounds for the factorial function for $k \geq 1$: 
    \begin{align}
        \sqrt{2\pi k} \lp \frac{k}{e} \rp^k e^{1/12k} \leq \; k! \; \leq \; \sqrt{2\pi k}  \lp \frac{k}{e} \rp^{k} e^{1/(12k+1)}, 
    \end{align}
    which can be simplified to the following for $k \geq 2$: 
    \begin{align}
        \sqrt{2\pi k} \lp \frac{k}{e} \rp^k  \leq \; k! \; \leq \;  e \sqrt{n}  \lp \frac{k}{e} \rp^{k}. 
    \end{align}
    We can then use these inequalities to bound $p_n \defined \mathbb{P}(S_{N,J} = n)$ as follows: 
    \begin{align}
        p_n &= \frac{N! \beta^n \barbeta^{N-n}}{(N-n)! n!}  = \frac{N! \beta^n \barbeta^{N-n}}{( \barbeta N)! (\beta N)!} \geq \frac{\lp  \sqrt{2\pi N} (N/e)^{N} \rp \beta^{n}\barbeta^{N-n}}{\lp e \sqrt{ N \barbeta} (\barbeta N/e)^{\barbeta N} \rp \lp e \sqrt{\beta N} (\beta N /e)^{\beta N}\rp} 
         = \frac{\sqrt{2\pi} }{e^2\sqrt{\beta \barbeta N }}, 
    \end{align}
    In particular, this implies $(1/p_n) \leq e^2 \sqrt{ \beta (1-\beta) N/2\pi} \leq 1.18 \sqrt{N \beta (1-\beta)} = 1.18 \sqrt{n (1-\beta)}$. 

    \subsection{\texorpdfstring{Properties of $f_{N, \lambda}$}{Properties of f-N-lambda}}
    \label{proof:properties-of-f-N-lamdba}

        We use the Taylor expansion of the function $f_{N, \lambda}(x) = \log \lp \beta_N e^{\lambda \barbeta_N x} +  \barbeta_N e^{-\lambda \beta_N x} \rp$ in terms of the variable $\lambda$ around $0$ in the proof of~\Cref{theorem:asympCI}. First, we observe that it's value at $0$: 
        \begin{align}
            f_{N, 0}(x) = \log \lp \beta_N + \barbeta_N \rp = \log 1 = 0. 
        \end{align}
        \paragraph{First derivative.} We now evaluate the first derivative of $f_{N, \lambda}$. 
        \begin{align}
            \frac{d f_{N, \lambda}(x)}{d \lambda} = \frac{\beta_N \barbeta_N x \lp e^{\lambda \barbeta_N x} -  e^{- \lambda \beta_N x}\rp}{\beta_N e^{\lambda \barbeta_N x} + \barbeta_N e^{- \lambda \beta_N x}} := \frac{A(\lambda)}{B(\lambda)}. 
        \end{align}
        Observe that $A(0) = 0$, which implies that $d f_{N, \lambda}(x)/d\lambda|_{\lambda=0} = 0$, and that $B(\lambda) \geq \min\{\beta_N, \barbeta_N\}$. Thus, for $|\lambda|\leq L$, we can upper bound the magnitude of the first derivative of $f_{N, \lambda}$ as 
        \begin{align}
            \sup_{x \in [0,1]}\sup_{|\lambda| \leq L} \lv \frac{d f_{N, \lambda}(x)}{d \lambda} \rv  \leq  \frac{\beta_N \barbeta_N e^{L}}{\min\{ \beta_N, \barbeta_N\} } \leq e^L.
        \end{align}

        \paragraph{Second derivative.} The second derivative of $f_{N, \lambda}$ has the following expression: 
        \begin{align}
            \frac{d^2 f_{N, \lambda}(x)}{d \lambda^2}  = \frac{d}{d\lambda} \frac{A(\lambda)}{B(\lambda)} = \frac{B(\lambda) A'(\lambda) - A(\lambda) B'(\lambda)}{B^2(\lambda)}, 
        \end{align}
        where 
        \begin{align}
            &A'(\lambda) = \beta_N \barbeta_N x^2 \lp \barbeta_N e^{\lambda \barbeta_N x} + \beta_N e^{-\lambda \beta_N x}  \rp,   \quad \text{and} \quad 
            B'(\lambda) = \beta_N\barbeta_N x \lp e^{\lambda \barbeta_N x} - e^{-\lambda \beta_N x} \rp = A(\lambda). 
        \end{align}
        Hence, we have 
        \begin{align}
            \frac{d^2 f_{N, \lambda}(x)}{d \lambda^2}  = \frac{A'(\lambda)}{B(\lambda)} - \lp\frac{A(\lambda)}{B(\lambda)} \rp^2 = \frac{A'(\lambda)}{B(\lambda)} - \lp \frac{d f_{N, \lambda}(x)}{d \lambda}   \rp^2. 
        \end{align}
        At $\lambda = 0$, this value is equal to $\beta_N \barbeta_N x^2$, and furthermore for $|\lambda|\leq L$, we can get the following upper bound: 
        \begin{align}
            \sup_{x \in [0,1]} \sup_{|\lambda|\leq L} \lv \frac{d^2 f_{N, \lambda}(x)}{d \lambda^2}  \rv & \leq e^{L} + e^{2L} \leq 2 e^{2L}. 
        \end{align}

        \paragraph{Third derivative.} Finally, we evaluate the third derivative of $f_{N, \lambda}(x)$ as  
        \begin{align}
            \frac{d^3 f_{N, \lambda}(x)}{d \lambda^3}  &= \frac{d}{d\lambda} \lp \frac{A'(\lambda)}{B(\lambda)} -  \lp \frac{d f_{N, \lambda}(x)}{d \lambda}   \rp^2\rp = \frac{A''(\lambda)}{B(\lambda)} - \frac{A'(\lambda)B'(\lambda)}{B^2(\lambda)} - 2 \frac{d f_{N, \lambda}(x)}{d\lambda} \frac{d^2 f_{N, \lambda}(x)}{d\lambda^2} \\
            & = \frac{A''(\lambda)}{B(\lambda)} - \frac{A'(\lambda)B'(\lambda)}{B^2(\lambda)} - 2 \frac{A(\lambda)}{B(\lambda)} \lp \frac{A'(\lambda)}{B(\lambda)} - \lp \frac{A(\lambda)}{B(\lambda)} \rp^2 \rp \\
            & = \frac{A''(\lambda)}{B(\lambda)} - \frac{A'(\lambda)B'(\lambda)}{B^2(\lambda)} - 2 \frac{A(\lambda) A'(\lambda)}{B^2(\lambda)}  + 2 \frac{A^3(\lambda)}{B^3(\lambda)}. 
        \end{align}
        The only new term we need to evaluate is $A''(\lambda)$ which is equal to $\beta_N \barbeta_N x^3\lp \barbeta_N^2 e^{\lambda \barbeta_N x} - \beta_N^2 e^{-\lambda \beta_N x}   \rp$. Next, we can loosely upper bound the magnitude of the third derivative  as 
        \begin{align}
            \sup_{x \in [0,1]} \sup_{|\lambda|\leq L} \lv \frac{d^3 f_{N, \lambda}(x)}{d \lambda^3}  \rv & \leq \sup_{x \in [0,1]} \sup_{|\lambda|\leq L} \lp \lv \frac{A''(\lambda)}{B(\lambda)} \rv + \lv \frac{A'(\lambda)B'(\lambda)}{B^2(\lambda)}\rv  + 2 \lv \frac{A(\lambda) A'(\lambda)}{B^2(\lambda)} \rv  + 2 \lv \frac{A^3(\lambda)}{B^3(\lambda)}\rv \rp \\
            & \leq e^L + e^{2L} + 2 e^{2L} + 2e^{3L} \leq 6 e^{3L}. 
        \end{align}

    \subsection{Proof of Lemma~\ref{lemma:coupling-pinelis-exp-bound}}
    \label{proof:coupling-pinelis-exp-bound}
        It suffices to show that the derivative of $h$ is non-negative in the domain $[0,1]$. Observe that by the quotient rule, we have the following for $x>0$:
        \begin{align}
            h'(x) &= \frac{d}{dx} \lp \frac{e^x - 1 - x}{x^2} \rp = \frac{(e^x-1)x^2 - (e^x-1-x)2x}{x^4} = \frac{ e^x(x^2 - 2 x) + x^2 + 2x}{x^4} \\
            & = \frac{e^x(x-2) + x + 2}{x^3}. 
        \end{align}
        We can see that $\lim_{x\downarrow 0} h'(x)=0$, and that for $x>0$, the numerator of $h'(x)$ is an increasing function of $x$, while the denominator is positive. Thus, we can conclude that $h'(x)>0$ for all $x>0$, and in particular, this means that 
        \begin{align}
            \sup_{x \in [0,a]} h(x) = h(a)  \qtext{$\implies$} e^x - 1 - x \leq h(a)\, x^2 \quad \text{for all } x \in [0,a]. 
        \end{align}

\end{document}